\documentclass{article}          
\usepackage{amsmath,amssymb,amsthm}
\usepackage{mhchem}
\usepackage[margin=1in]{geometry}
\usepackage{comment}
\usepackage{enumitem}
\usepackage{mathtools}
\usepackage{xcolor}
\usepackage{diagbox}

\numberwithin{equation}{section}
\theoremstyle{plain}
\newtheorem{theorem}{Theorem}[section]

\newtheorem{lemma}[theorem]{Lemma}
\newtheorem{corollary}[theorem]{Corollary}

\theoremstyle{definition}
\newtheorem{definition}[theorem]{Definition}
\newtheorem{example}[theorem]{Example}

\newcommand{\N}{\mathbb{N}}

\newcommand{\R}{\mathbb{R}}

\newcommand{\lbd}{\lambda}
\newcommand{\Lbd}{\Lambda}
\newcommand{\ep}{\epsilon}

\def\ol{\overline}
\def\ul{\underline}
\def\tF{\tilde F}
\def\ul{\underline}
\def\tG{\tilde G}
\def\Q{\mathbb{Q}}

\def\bay{\begin{array}}
\def\eay{\end{array}}
\def\tld{\tilde}
\def\Lra{\Leftrightarrow}
\def\lng{\langle}
\def\rng{\rangle}
\def\E{\mathcal{E}}
\def\osc{\text{Osc}}

\def\1{\mathbf{1}}
\def\pmtx{\begin{pmatrix}}
\def\pmtrx{\end{pmatrix}}
\def\F{\mathcal{F}}

\def\A{\mathcal{A}}

\def\c{\mathfrak{c}}
\def\dlt{\delta}
\def\gma{\gamma}
\def\Gma{\Gamma}
\def\lbd{\lambda}
\def\lph{\alpha}
\def\cd{ \ \stackrel{(d)}{\rightarrow} \ }
\def\cp{ \ \stackrel{(p)}{\rightarrow} \ }
\def\nid{\noindent}
\def\tbf{\textbf}

\def\cdlg{c{\`a}dl{\`a}g}
\def\enumrom{\begin{enumerate}[noitemsep,label={(\roman*)}]}
\def\enumar{\begin{enumerate}[noitemsep,label={\arabic*.}]}
\def\enumalph{\begin{enumerate}[noitemsep,label={(\alph*)}]}
\def\itemgo{\begin{itemize}[noitemsep]}
\def\itemend{\end{itemize}}
\def\enumend{\end{enumerate}}
\def\cp{\stackrel{\text{p}}{\to}}
\def\cd{\stackrel{\text{d}}{\to}}
\def\lcd{\stackrel{\text{ld}}{\to}}

\def\ptl{\partial}
\def\th{\theta}
\def\env{\mathrm{env}}
\def\piv{\mathrm{piv}}
\def\sgn{\mathrm{sgn}}
\def\sub{\mathrm{sub}}
\def\cP{\mathcal{P}}
\def\dom{\mathrm{dom}}
\begin{document}
\date{}
\author{Eric Foxall}
\title{Limit processes and bifurcation theory of \\ quasi-diffusive perturbations}
\maketitle

\begin{abstract}
The bifurcation theory of ordinary differential equations (ODEs), and its application to deterministic population models, are by now well established. In this article, we begin to develop a complementary theory for diffusion-like perturbations of dynamical systems, with the goal of understanding the space and time scales of fluctuations near bifurcation points of the underlying deterministic system. To do so we describe the limit processes that arise in the vicinity of the bifurcation point. In the present article we focus on the one-dimensional case.
\end{abstract}


\normalsize

\section{Introduction}
A broad class of individual-based stochastic population models, under suitable mixing assumptions, can be interpreted as diffusion-like perturbations of an underlying smooth dynamical system. The general framework is known as density-dependent Markov chains \cite{ddmc}; examples include chemical reactions \cite{vk}, infection spread \cite{primer}, population genetics \cite{kimura} and evolutionary games \cite{popovic}. In each case, there is a system size parameter $N$ and functions $F$ and $G$ such that, letting $\ep=1/\sqrt{N}$, for small $\ep>0$, trajectories of the vector of population densities $x_\ep \in \R_+^d$ resemble solutions to a stochastic differential equation (SDE) of the form
\begin{align}\label{eq:QDP}
dx = F(x)\,dt + \epsilon\, \sqrt{G(x)}\,dB,
\end{align}
where $B$ is a $d$-dimensional standard Brownian motion and $\sqrt{G(x)}$ is the square root of the positive (semi)-definite matrix $G(x)$. The interpretation is made rigorous through limit theorems (see \cite{ddmc}) such as
\enumrom
\item \textbf{the law of large numbers}: letting $\phi_t$ denote the solution flow of $x'=F(x)$,
$$\text{if}  \ \ x_\ep(0) \to x_0  \ \ \text{as} \ \ \ep \to 0 \ \ \text{then for fixed} \ \ T>0,  \ \ \sup_{t \le T}|x_\ep(t)-\phi_t(x_0)| \cp 0 \ \ \text{as} \ \ \ep \to 0.$$
\item \textbf{the central limit theorem}: letting $Y_\ep(t)= \ep^{-1}(x_\ep(t)-\phi_t(x_0))$ and $(Y(t))$ denote the solution to the initial-value problem
$$Y(0) = Y_0 \ \ \text{and} \ \ dY = DF(\phi_t(x_0))Ydt + \sqrt{G(\phi_t(x_0))}dB,$$
if $x_\ep(0) \to x_0$ and $Y_\ep(0) \to Y_0$ as $\ep \to 0$ then $Y_\ep \cd  Y$ as $\ep \to 0$.
\enumend
Briefly, on the natural time scale and on the population density scale, such processes resemble solutions to the deterministic system $x'=F(x)$, with random fluctuations of size $O(\ep)$. When $F$ and $G$ are non-degenerate, this description gives an accurate sense of the typical behaviour of sample paths. For example, if $x_\star$ is a linearly stable equilibrium point of the deterministic system $x'=F(x)$, $G(x_\star) \ne 0$ and $x_\ep(0)=x_\star+O(\ep)$, then fluctuations in $x_\ep(t)-x_\star$ of size $\ep$ are observed on the natural time scale, and larger than $O(\ep)$ fluctuations occur only as the result of brief, rare excursions, as described by the theory of moderate to large deviations (see for example \cite{fw}).\\

On the other hand, when either $F$ or $G$ is degenerate the description given by (i)-(ii) above becomes uninformative. For example, if $x_\star$ is stable but non-hyperbolic in the sense that $DF(x_\star)$ is non-invertible, then as we will see, larger than $O(\ep)$ fluctuations are observed on a longer time scale, while if $G(x_\star)=0$, fluctuations can still be non-zero, and can be large or small depending on $F$. This has already been observed in particular cases, such as the SIS \cite{crit-scale} and SIR \cite{martinlof} models of infection spread. In both references, $\alpha_1,\alpha_2,\alpha_3>0$ are found such that if $x_\ep(0)=x_\star + \ep^{\alpha_1}$ and $\lbd=\lbd_\star + \ep^{\alpha_2}$, then $\ep^{-\alpha_1}(x_\ep(\ep^{-\alpha_3}t)-x_\star)$ converges in distribution as $\ep \to 0$ to a diffusion. Our aim is to develop a sufficiently general theory that we can accurately describe all limits of this kind, subject only to existence of a Taylor expansion for $F$ and $G$ near bifurcation points of the deterministic approximation.\\

We refer to the models under consideration as \emph{quasi-diffusive perturbations} or QDPs; a precise definition is given in Section \ref{sec:def}. In this article, we develop the theory of limit processes for QDPs, including both degenerate and parametrized systems. This enables us to accurately describe the spatial (i.e., population density) and temporal scales of fluctuations of $(x_\ep)$ in a neighbourhood of bifurcation points $(x_\star,\lbd_\star)$ of the deterministic approximation  $x'=F(x,\lbd)$, as a function of the Taylor expansion of $F$ and $G$.\\

As a result of this theory we obtain enhanced versions of the usual deterministic bifurcation diagrams, since they also account for the effect of the stochastic terms. Because of the added complexity of this endeavour, here we mostly restrict our attention to the one-dimensional case, i.e., $x \in \R$. It should also be noted that, although it deals with stochastic processes and treats the same types of bifurcations, our theory bears no direct resemblance to the stochastic bifurcation theory of random dynamical systems \cite{rds}. In the next section we describe our approach and the main results, and give an overview of the rest of the article.

\section{Overview and main results}

In this section we give an informal overview of our approach and main results. Precise statements are deferred to the section in which they appear, as a good deal of exposition is required to state them.\\

\nid\tbf{Assumption: strong stochasticity. }We begin by noting an important assumption of the theory, which is that locally, $F=O(G)$, i.e., $|F|$ is uniformly bounded by a multiple of $|G|$ in a neighbourhood of $(x_\star,\lbd_\star)$, that we refer to as \emph{strong stochasticity}. This ensures that diffusion tends to dominate at small scales, i.e., small values of $x-x_\star$, while drift dominates at large scales. All density-dependent Markov chains are strongly stochastic, a fact that is not hard to show but is deferred to a later work.\\

\nid\tbf{Approach. }Given a point $(x_\star,\lbd_\star)$, our approach is to consider all rescalings of the form
\begin{align}\label{eq:Yresc}
Y_\ep(t) = a_\ep (x_\ep(b_\ep t; \lbd_\ep)-x_\star)
\end{align}
with $a_\ep \to\infty$, and $\lbd_\ep \to \lbd_\star$ as $\ep \to 0$, then to find all choices of $(a_\ep),(b_\ep),(\lbd_\ep)$ for which $(Y_\ep)$ converges to the solution of a non-trivial ordinary or stochastic differential equation (ODE/SDE). We refer to such sequences as \emph{limit scales} for the family $(x_\ep)$. We consider not only the case where $x_\star$ is a constant but also the case where $x_\star=x_\star(\lbd)$ is a non-constant equilibrium branch of $F$, i.e., is such that $F(x_\star(\lbd),\lbd))=0$. To fix ideas, we begin with the former case.\\

When $(a_\ep),(b_\ep),(\lbd_\ep)$ is a limit scale, the limiting equation for $(Y_\ep)$ has the form
\begin{align}\label{eq:lim}
dY = \tF(Y)dt + \tG(Y)dB
\end{align}
with
\begin{align}\label{eq:limproc}
\tF(x) &:= \lim_{\ep \to 0}a_\ep b_\ep F(x_\star + x/a_\ep,\lbd_\ep) \ \text{and} \nonumber \\
\tG(x) &:= \lim_{\ep \to 0}\ep^2 a_\ep^2 b_\ep G(x_\star + x/a_\ep,\lbd_\ep).
\end{align}
The multipliers $a_\ep b_\ep$ and $\ep^2 a_\ep^2 b_\ep$ arise from the fact that drift scales linearly in space while diffusion scales quadratically, and both scale linearly in time.\\

To simplify the discussion, assume $(x_\star,\lbd_\star)=(0,0)$ which can be achieved by translation, and focus on the rectangle $(x,\lbd)\in[0,1]^2$ in the first quadrant; behaviour on other quadrants follows in the same way.\\

\nid\tbf{Requirements for a limit scale. } $(a_\ep),(b_\ep),(\lbd_\ep)$ is a limit scale iff 
\enumrom
\item \tbf{shape:} The shape of $F(\cdot/a_\ep,\lbd_\ep)$ and $G(\cdot/a_\ep,\lbd_\ep)$ is stable as $\ep \to 0$, i.e., there are sequences $(f_\ep),(g_\ep)$ such that as $\ep \to 0$, we have locally uniform convergence in $x$ of the functions
$$\frac{1}{f_\ep}\,F(x/a_\ep,\lbd_\ep)\quad \text{and} \quad \frac{1}{g_\ep}\, G(x/a_\ep,\lbd_\ep),$$
\item \tbf{ratio:} the rescaled drift to diffusion ratio
\begin{align}\label{eq:rat}
\frac{F(x/a_\ep,\lbd_\ep)}{\ep^2 a_\ep G(x/a_\ep,\lbd_\ep)}
\end{align}
converges either to $0$, a non-trivial function of $x$, or $\infty$, and\\
\item \tbf{time scale:} $(b_\ep)$ is chosen just large enough that one or both of $\tF$, $\tG$ is not identically zero.
\enumend

For sequences $(c_n),(d_n)$, say that $c_n \ll d_n$, $c_n \asymp d_n$ or $c_n \gg d_n$ if $\lim_{n\to\infty} c_n/d_n$ exists and is equal to $0$, a number in $(0,\infty)$, or $\infty$, respectively. Say that the ratio of $(c_n)$ to $(d_n)$ is stable if one of $c_n\ll d_n$, $c_n\asymp d_n$ or $c_n\gg d_n$ holds. Say that $(c_n)$ is stable if the ratio of $(c_n)$ to 1 is stable, i.e., if $\lim_{n\to\infty}c_n$ exists in $[0,\infty]$.\\

\nid\tbf{Resolution of requirements. }Together, (i)-(ii) determine $(a_\ep,\lbd_\ep)$, while (iii) implies that $(b_\ep)$ is uniquely determined, modulo $\asymp$, from $(a_\ep,\lbd_\ep)$. (i)-(ii) are resolved as follows.

\enumrom
\item As explained in Section \ref{sec:domterms}, there are finite sets $M_F,M_G \subset (0,\infty)\cap \Q$ such that on each quadrant of $\R^2$, the shape of $F(\cdot/a_\ep,\lbd_\ep)$ (respectively, $G(\cdot/a_\ep,\lbd_\ep$)) is stable as $\ep \to 0$ iff the ratio of $1/a_\ep$ to $\lbd_\ep^m$ is stable for every $m\in M_F$ ($m\in M_G$). If $(1/a_\ep,\lbd_\ep)$ are viewed as points in the $(x,\lbd)$ plane, regions of stability correspond to conditions such as, for example, $\lbd \ll  x \ll \sqrt{\lbd}$, or $x\asymp \lbd$.\\

\item In Section \ref{sec:const-equil}, \eqref{eq:dd-fcn} we define a drift-diffusion ratio function $r(x,\lbd)$ that is easy to write down and morally satisfies $r(x,\lbd) \asymp |xF(x,\lbd)|/|G(x,\lbd)|$ as $(x,\lbd)\to (0,0)$. More precisely, on regions of the form $\lbd^{m} \le x \le \lbd^{m'}$ for consecutive $m,m'\in M_F\cup M_G$, where $F(x,\lbd)\approx x^{\lph_1}\lbd^{\lph_2}$ and $G(x,\lbd)\approx x^{\beta_1}\lbd^{\beta_2}$ for some $\lph,\beta \in \N^2$, $r(x,\lbd)$ is defined as $x^{1+\lph_1-\beta_1}\lbd^{\lph_2-\beta_2}$.
\enumend
One of our main results, expressed in various contexts by Theorems \ref{thm:iso-limits}, \ref{thm:prmtrzd-limits} and \ref{thm:eq-branch}, is that if $(b_\ep)$ is chosen correctly relative to $(a_\ep)$ and $(\lbd_\ep)$ then \eqref{eq:limproc} holds for some
\enumalph
\item $\tF=0$ and $\tG\ne 0$, if the shape of $F(\cdot/a_\ep,\lbd_\ep)$ is stable and $r(1/a_\ep,\lbd_\ep)\ll \ep^2$,
\item $\tF\ne 0$ and $\tG\ne 0$, if the shape of $F(\cdot/a_\ep,\lbd_\ep)$ and $G(\cdot/a_\ep,\lbd_\ep)$ is stable and $r(1/a_\ep,\lbd_\ep) \asymp \ep^2$,
\item $\tF \ne 0$ and $\tG=0$, if the shape of $G(\cdot/a_\ep,\lbd_\ep)$ is stable and $r(1/a_\ep,\lbd_\ep) \gg \ep^2$.
\enumend
Cases (a) and (c) are respectively called the \emph{pure diffusive range} and the \emph{deterministic range}, while the intermediate region, case (b), is called the \emph{drift-diffusion scale}, or dd scale for short.\\

\nid\tbf{Summary of limit scales. }To summarize, $(a_\ep),(b_\ep),(\lbd_\ep)$ is a limit scale if 
\itemgo
\item the ratio of $r(1/a_\ep,\lbd_\ep)$ to $\ep^2$ is stable,
\item the ratio of $1/a_\ep$ to $\lbd_\ep^m$ is stable for each relevant $m$, i.e.,
\enumalph
\item for every $m\in M_F$ if $\lim_{\ep \to 0}\ep^{-2} r(1/a_\ep,\lbd_\ep)>0$ and
\item for every $m\in M_G$ if $\lim_{\ep \to 0}\ep^{-2} r(1/a_\ep,\lbd_\ep)<\infty$, and
\enumend
\item $(b_\ep)$ is chosen correctly relative to $(a_\ep)$ and $(\lbd_\ep)$.
\itemend
\nid\tbf{Partition into limit classes. }The above characterization partitions limit scales according to whether each of the relevant ratios (those of $r$ to $\ep^2$ and of $1/a_\ep$ to $\lbd_\ep^m$) is $\ll 1$, $\asymp 1$ or $\gg 1$. Moreover, the partition elements, that we call \emph{limit classes}, are also the equivalence classes for the relation $(\tF_1,\tG_1)\sim (\tF_2,\tG_2)$ if $\tF_1=c_1\tF_2$ and $\tG_1=c_2\tG_2$ for positive constants $c_1,c_2$. This is easily verified from the facts that
\itemgo
\item $\tF=0$ iff $r(1/a_\ep,\lbd_\ep) \ll \ep^2$ and $\tG=0$ iff $r(1/a_\ep,\lbd_\ep) \gg \ep^2$, and
\item from the above iff statement relating stable shape of $F,G$ to stable ratios.
\itemend
\nid\tbf{Graphical depiction of limit classes. }We can depict the above partition nicely in the $(x,\lbd)$ plane, by replacing $\ll,\asymp,\gg$ with $\le,=,\ge$. First, define the \emph{transition curves} and the \emph{drift-diffusion curve} as follows:
\itemgo
\item For $m\in(0,\infty)$ let $\Gma_m=\{(x,\lbd)\colon x=\lbd^m\}$. The transition curves are then $\{\Gma_m \colon m\in M_F\cup M_G\}$.
\item The drift-diffusion curve is the sequence of curves $(\Phi_\ep)_\ep$ defined by
$$\Phi_\ep = \{(x,\lbd)\colon r(x,\lbd)=\ep^2\}.$$
\itemend

To depict the partition, draw the following lines.
\enumar
\item Draw the dd curve.
\item For each $m\in M_F$, add $\Gma_m \cap \{(x,\lbd)\colon r(x,\lbd)>1\}$.
\item Then, for each $m\in M_G$, add $\Gma_m \cap \{(x,\lbd)\colon r(x,\lbd)<1\}$.
\enumend
The result (for each $\ep$) is a connected union of curves $\Psi_\ep$, consisting of the dd curve with some or all of each transition curve branching off from it. For each $\ep$, define a partition of $[0,1]^2$ by taking as its elements
\enumar
\item the branch points, i.e., each singleton set $p(m):=\Gma_m \cap \Phi_\ep$, for $m\in M_F\cup M_G$,
\item the edges, i.e., each connected component of $\Psi_\ep \setminus \bigcup_m p(m)$, and
\item the regions, i.e., each connected component of $[0,1]^2 \setminus \Psi_\ep$.
\enumend
The set of partition elements is then in $1:1$ correspondence with the limit classes, restricted to $(1/a_\ep,\lbd_\ep)\in[0,1]^2$, in the manner described above. An example is shown in Figure \ref{fig:classex} for $F(x,\lbd)=\lbd x-x^2$ and $G(x,\lbd)=x$, corresponding to a transcritical bifurcation. The dd curve cuts transversely across the transition curve, which in this case coincides with the positive equilibrium branch of $F$, and separates the pure diffusive range from the deterministic range, and as long as $F=O(G)$, the pure diffusive range is on the side nearest to $(0,0)$. In the above example the dd curve is the graph of a family of functions $\lbd\mapsto\phi_\ep(\lbd)$; as described in Theorem \ref{thm:dd-curve}, in some cases it can develop folds, which is shown in Figure \ref{fig:uprex2}.\\

\begin{figure}
\center
\includegraphics[width=2.5in]{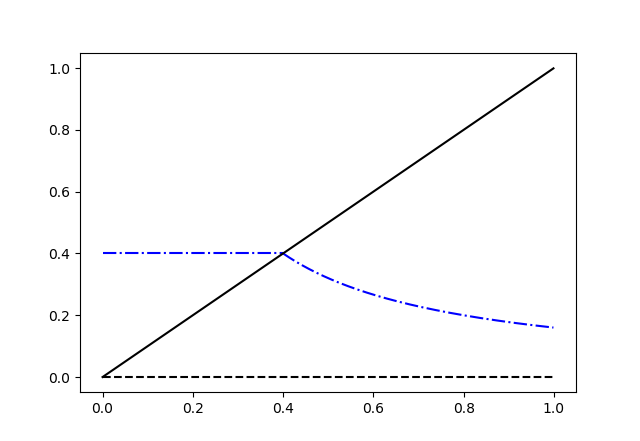}
\includegraphics[width=2.5in]{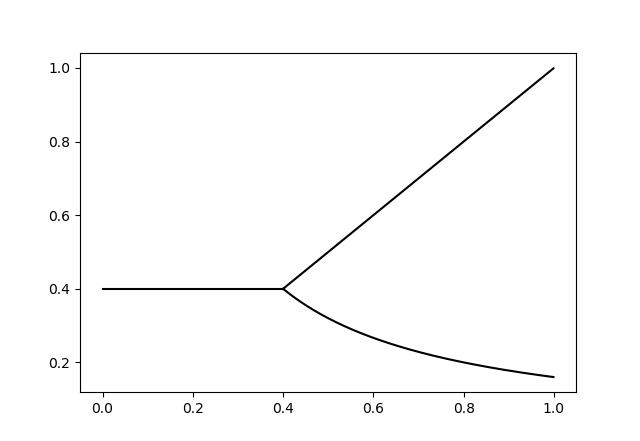}
\caption{\emph{Left:} Plot of equilibria (black) with transition curve in solid black, and drift-diffusion curve (blue) with $\ep=0.4$ on $[0,1]^2$ with $\lbd$ on horizontal axis, for the case $F(x,\lbd)=\lbd x-x^2$ and $G(x)=x$. \emph{Right:} the union of curves $\Phi_\ep$, for the same $F,G$ and $\epsilon$.}
\label{fig:classex}
\end{figure}

\nid\tbf{Equilibrium branch. }In the example of Figure \ref{fig:classex}, it turns out there are fluctuations around the non-constant equilibrium branch that aren't captured by the above analysis, as the latter takes $x=0$ as its point of reference. To capture these, we replace $x_\star$ with $x_\star(\lbd_\ep)$ in \eqref{eq:Yresc} and identify the limit scales, using the same method as above. For simplicity we assume that $F$ has a simple equilibrium branch, i.e., there is a function $x_\star(\lbd)$ such that $F(x_\star(\lbd),\lbd)=0$ and $x_\star(\lbd)$ is a simple root of $F(\cdot,\lbd)$ for each $\lbd$, which allows us to work with the linearization of $F$. The approximation is useful when $\lbd \gg \lbd_\star(\ep)$, the intersection point of the dd curve $\Phi_\ep$ with the equilibrium branch, which is also the range of $\lbd$ values for which the analysis around $x=0$ loses information about fluctuations around $x_\star(\lbd)$. As in the previous analysis, there is a pure diffusive range centered around $x_\star(\lbd)$ and a deterministic range further out, separated by a drift-diffusion scale whose width we denote by $\phi_\ep^\star(\lbd)$. A hybrid plot, showing the dd curves around both $0$ and $x_\star(\lbd)$, is given in Figure \ref{fig:transcrit}. This is complemented by numerical simulations of the logistic Markov chain
\begin{align}\label{eq:logisMC}
X \to \begin{cases} X+1 & \text{at rate} \ (1+\lbd)X, \\
X-1 & \text{at rate} \ X + X^2/N.\end{cases}\end{align}
to demonstrate the limits observed across the diagram. For DDMCs \cite{ddmc}, the functions $F,G$ are given by
\begin{align*}
F(x)=\sum_\Delta \Delta q_\Delta(x) \ \ \text{and} \ \ G(x)=\sum_\Delta \Delta \Delta^{\top}q_\Delta(x)
\end{align*}
where $q_\Delta(x)=\lim_{N\to\infty} q_N(\lfloor Nx \rfloor ,\lfloor Nx\rfloor+\Delta)/N$ and $q_N(X,Y)$ is the transition rate from $X$ to $Y$ in the Markov chain, for a given value of $N$. The Markov chain \eqref{eq:logisMC} has $q_1(x)=(1+\lbd)x$ and $q_{-1}(x) = x + x^2$ so $F(x)=\lbd x - x^2$ and $G(x)=(2+\lbd)x + x^2$ which is asymptotic to $2x$ as $(x,\lbd)\to (0,0)$ and thus compatible with Figure \ref{fig:transcrit}, modulo $\asymp$.\\

\begin{figure}
\begin{center}
\includegraphics[width=3.5in]{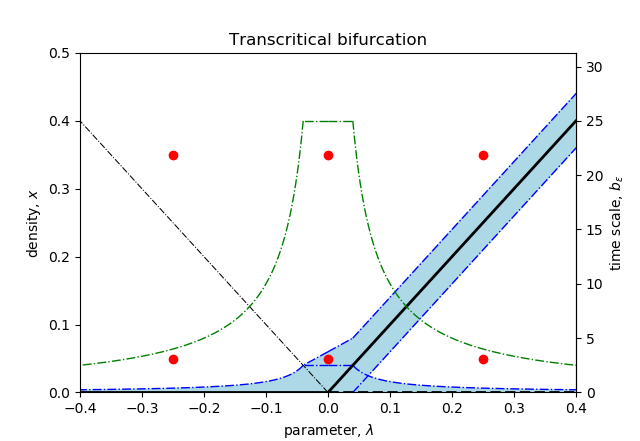}\\
\end{center}
\caption{Bifurcation diagram for $F(x)=\lbd x- x^2$ and $G(x)=x$ including dd time scale, depicted with $\ep=0.04$. Equilibria (thick) and non-equilibrium transition curve (thin) outlined in black. Drift-diffusion curves $\phi_\ep$ and $x_\star \pm \phi_\ep^\star$ in blue, with pure diffusive regions shaded in blue. Time scale at dd curve in green (it is the same for both curves in this example) with right-hand axis indicating values. Red dots denote values of $\lbd$ and $x_\ep(0)$ used in Figure \ref{fig:sampaths}.}
\label{fig:transcrit}
\end{figure}

\begin{figure}
\centering
\includegraphics[width=4in]{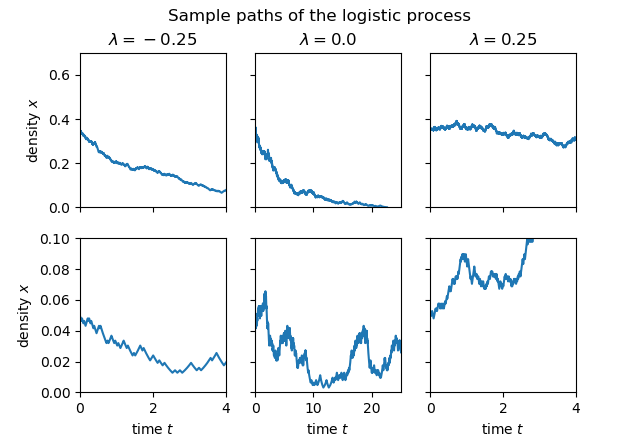}
\caption{Sample paths of the logistic Markov chain \eqref{eq:logisMC}, which corresponds to the example from Figure \ref{fig:transcrit}. At each value of $\lbd$, time window is set to $b_\ep(\lbd)$, the time scale of fluctuations at the dd scale for that value of $\lbd$.}
\label{fig:sampaths}
\end{figure}

\nid\tbf{Bifurcations. }Combining these analyses, we study three types of bifurcations in one dimension: saddle-node, transcritical and pitchfork. We find it is convenient to focus on the dd space and time scales. As long as the dd curve does not fold these can be described using functions of $\lbd$, respectively $\phi_\ep$ and $b_\ep$, around $x=0$ and $\phi_\ep^\star$, $b_\ep^\star$ around $x_\star(\lbd)$. The details are given in Section \ref{sec:bifurc}; some general findings are that (i) $\phi_\ep$ is constant for $\lbd$ near $0$, and may increase or decrease as $\lbd$ increases, and (ii) $b_\ep \gg 1$ for $\lbd$ near $0$, corresponding to slow fluctuations, and decreases (though not always monotonically) to $1$ as $\lbd \uparrow 1$ when $F(0,\lbd)=0$ for $\lbd$ near $0$ as in the transcritical and pitchfork case, and to $0$ when $F(0,\lbd)\ne 0$ for $\lbd \ne 0$, as in the saddle-node case. $b_\ep^\star$ is similar but always decreases to $1$, which relies on the assumption that $F(x_\star(\lbd),\lbd)=0$.\\

\nid\tbf{Scope of the article and later work.} For simplicity, in this article we only study bifurcations in one dimension, and we define QDPs in such a way that the required error estimates for convergence to a limit are satisfied. This suggests two directions for generalization. The first is to study other bifurcations; for example, the Hopf bifurcation can likely be tackled by combining the methods of this article with an averaging result. The second is to make the results applicable to density-dependent Markov chains (DDMCs). It appears that DDMCs whose transition rates have chemical mass-action form (see for example Chapter VII, Section 2 in \cite{vk}; briefly, the reaction rate is proportional to the product of the concentrations of the reactants, or to the same with combinatorial corrections), which also includes most well-mixed population and infection spread models, naturally satisfy the required error estimates, at least at the drift-diffusion scale, so I intend to discuss this in a companion paper. Towards this goal, in the present article we take care to specify precise conditions for convergence, formulated in the language of semimartingales. The reader who is uninterested in these technical details can effectively skip Section \ref{sec:def} and assume that \eqref{eq:QDP} is meant literally, i.e., that $(x_\ep)$ is a family of diffusions with small noise parameter.\\

\nid\tbf{Layout. }The article is organized as follows. In Section \ref{sec:def} we define precisely what it means for a family of processes $(x_\ep)$ to resemble solutions to \eqref{eq:QDP}, and give conditions for \eqref{eq:Yresc} to converge to a diffusion limit. In Section \ref{sec:iso} we treat the unparametrized case, identifying limit scales and computing limit processes. This simpler case acts as a warm-up for Section \ref{sec:prmtzd}. In Section \ref{sec:prmtzd} we treat the parametrized case, establishing the results described in this section. A non-trivial effort is required in Sections \ref{sec:env}-\ref{sec:domterms} to understand the anatomy of bivariate Taylor expansions. In Section \ref{sec:bifurc} we treat the three types of bifurcations mentioned above. We obtain formulae for the space and time scale of fluctuations, paying special attention to the drift-diffusion scale. Since, corresponding to a given $F$ there are potentially several $G$ such that $F=O(G)$, I'll finish this section by arguing that a particular choice of $G$ is generic and show that in this case, the bifurcation diagrams are particularly easy to describe.

\section{Definitions and basic limit theorems}\label{sec:def}

To understand the suitable class of processes, we work backwards from the desired limit processes, which are diffusions, i.e., solutions of a martingale problem associated to a stochastic differential equation. Once we have formulated the martingale problem in suitable generality to allow for explosion, we move to quasi-diffusions, which are sequences of semimartingales that converge to diffusions. Finally we define quasi-diffusive perturbations, which are generalizations of \eqref{eq:QDP} satisfying sufficiently strong estimates that, upon a rescaling of the form \eqref{eq:Yresc}, will be quasi-diffusions if the functions describing drift and diffusion (that we refer to as characteristics) converge.

\subsection{Diffusions}
Our definition of diffusion will be via the martingale problem, which is arguably the most flexible approach. For the sake of the unacquainted reader, we take a short detour to arrive at that formulation. The classical definition of a diffusion is a strong Markov process with continuous sample paths. Underlying this definition is the idea that a diffusion solves a stochastic differential equation (SDE), such as an initial value problem of the form
\begin{align}\label{eq:SDE}
X(0) = x, \quad dX = F(X)dt + \sigma(X)dB,
\end{align}
where $B$ is a $d$-dimensional standard Brownian motion, $F:\R^d\to\R^d$ is a vector field, and $\sigma:\R^d \to M_d(\R)$ is a $d\times d$ matrix-valued function. The formal interpretation of \eqref{eq:SDE} is via the corresponding integral equation
\begin{align}\label{eq:intSDE}
X(t) = x + \int_0^t F(X(s))ds + \int_0^t \sigma(X(s))dB(s),
\end{align}
which leads to the notion of a strong solution: a \emph{strong solution} of \eqref{eq:intSDE} is an $\R^d$-valued process $X$, defined on the filtered probability space $(\Omega,\F,P)$ of a $d$-dimensional Brownian motion $B$, where $\F=(\F(t))$ is the completion of the natural filtration of $B$, such that $X$ is adapted to $\F$ and \eqref{eq:intSDE} is satisfied $P$-almost surely, for all $t\ge 0$. The spirit of this definition is that one begins with a Brownian motion, and then constructs the solution $X$ directly from $B$.\\

A related notion is that of a \emph{weak solution}, which is a probability space $(\Omega,\F,P)$ and a pair of $\F$-adapted and continuous processes $X$ and $B$, such that $B$ is a $d$-dimensional standard $\F$-Brownian motion (see Chapter 5, Section 1 of \cite{ethktz} for a precise definition of $\F$-Brownian motion), and \eqref{eq:intSDE} is satisfied $P$-almost surely for all $t\ge 0$. The spirit of this definition is that both $X$ and $B$ are constructed upon a probability space $(\Omega,\F,P)$ which may be freely chosen.\\

The raison d'{\^e}tre of the weak solution notion is that in order to find solve \eqref{eq:intSDE}, it should not be strictly necessary to start from a Brownian motion. Taking this a step further, if there is a characterization of $X$ in a weak solution to \eqref{eq:intSDE}, then we can forget about $B$ and focus on the distribution of $X$. The desired characterization is called the martingale problem. Define $G=\sigma \sigma^{\top}$ and the operator
\begin{align}\label{eq:mg-op}
Lf = \frac{1}{2}\sum_{i,j=1}^d G_{ij} \frac{\ptl^2}{\ptl x_i \ptl x_j} f + \sum_i F_i \, \frac{\ptl}{\ptl x_i} f
\end{align}
on $C^2$ functions $f:\R^d\to\R$. A solution to the martingale problem for \eqref{eq:SDE} is a probability measure on the space $C([0,\infty),\R^d)$ such that for all $f \in C_0^2(\R^d)$, the corresponding random variable $X$ satisfies $X(0)=x$ and has the property that
$$f(X(t)) - \int_0^t Lf(X(s))ds$$
is a martingale with respect to the natural filtration $\F(t) := \sigma(\{X(s) \colon s\le t\}$ of $X$. Using the fact that $[B](t) = t$, together with some basic properties of the stochastic integral as well as It{\^o}'s equation (Chapter 1, Theorem 4.57 in \cite{jacod}), it is not hard to show that the distribution of $X$ in a weak solution of \eqref{eq:intSDE} solves the martingale problem. It is also possible to construct a weak solution from a solution to the martingale problem (Chapter 5, Theorem 3.3 in \cite{ethktz}). From now on we shall focus on the martingale problem, so we won't further discuss these connections rigorously.\\

Since population values are typically non-negative, it will be useful to formulate a version of the martingale problem that allows both the domain of $F,G$, and the state space of the process, to be restricted to an open, connected set $U\subset \R^d$. Fortunately, this has already been mostly treated; we record the definitions and the existence and uniqueness results below, following $\S$1.12-1.13 of \cite{pinsky}. Let $M_+(\R,d)$ denote the set of $d\times d$ positive semidefinite matrices with values in $\R$. For $A\subset \R^d$ let $\hat A$ denote the one-point compactification of $A$, and let $\c$ denote the one point such that $\hat A = A \cup \{\c\}$. If $A=\R^d$ then the resulting topology on $\hat A=\hat \R^d$ is generated by the metric $\rho$ given by identifying $\hat \R^d$ with the sphere $S^d$ and using the standard Riemannian metric $\rho$ on $S^d$. If $A\ne \R^d$ we can equip $\hat A$ with the metric $\rho_A$ defined by
$$\rho_A(x,y) = \begin{cases}\inf_{z \in \partial A \cup \{\c\}} \rho(x,z) \ \text{if} \ y=\c,\\
\min(\rho(x,y),\rho_A(x,\c) + \rho_A(y,\c)) \ \text{if} \ x,y \ne \c.\end{cases}$$
The resulting topology on $\hat A$ is such that $x_n \to_{\rho_A} x \in \hat A$ if (i) $x \in A$ and $d(x_n,x) \to 0$ or (ii) $x = \c$ and one of (a) $|x_n|\to\infty$ or (b) $d(x_n,\ptl A) \to 0$ holds, where $d$ is the Euclidean distance. Using $\rho_A$ to define a compatible metric in the manner of equation (8.1) in Ch.~1, \cite{pinsky}, the space $C([0,\infty),\hat A)$, with the topology of uniform $\rho_A$-convergence on bounded intervals, is shown to be a Polish space. Below, a \emph{domain} refers to an open and connected set, and $A\subset\subset B$ if $A$ is bounded and $\ol A\subset B$ where $\ol A$ denotes the closure of $A$.

\begin{definition}[Martingale problem and generalized martingale problem, Chapter 1, Sections 12-13 \cite{pinsky}]\label{def:mp}
Let $U\subset \R^d$ be open and connected and let $F:U \to \R^d$ and $G:U\to M_+(\R,d)$ be measurable. For $X \in C([0,\infty),\hat U)$ and $V\subset U$ define $\tau(V) = \inf\{t \colon X(t) \notin V\}$, and let
$$\hat\Omega_U = \{X \in C([0,\infty),\hat U)\colon \tau(U)=\infty \ \text{or} \ X(\tau(U)+t)=\c \ \text{for all} \ t>0\}.$$
Define the filtration $\hat\F_U(t)$ on $\hat \Omega_U$ by $\hat \F_U(t) = \sigma(X(s), 0 \le s \le t)$ and let $\hat \F_U=\hat \F_U(\infty)$.\\

Let $(P_x)_{x \in \hat U}$ on $(\hat \Omega_U,\hat \F_U)$ be a family of probability measures with $P_x(X(0)=x)=1$ for each $x$. $(P_x)$ solves the generalized martingale problem for $F,G$ if there is a sequence of domains $(D_n)$, with $D_n \subset\subset D_{n+1}$ for each $n$ and $\bigcup_n D_n=U$, such that for each $\smash x \in \hat U$, $f \in C^2(U)$ and $n>0$,
\begin{align}\label{eq:mp-mg}
f(X(t\wedge \tau(D_n)) - \int_0^{t \wedge \tau(D_n)}(Lf)(X(s))ds\end{align}
is an $\hat\F_U$-martingale with respect to $P_x$. $(P_x)$ solves the martingale problem for $F,G$ if $P_x(\tau(U)=\infty)=1$ for each $x \in U$, and for each $f\in C_0^2(U)$, \eqref{eq:mp-mg} is a martingale with $t$ in place of $t \wedge \tau(D_n)$.
\end{definition}

\begin{lemma}[\cite{pinsky}]\label{lem:SDExist}
Suppose $F:U\to \R^d$ and $G:U\to M_+(\R,d)$ are bounded on compact $K\subset U$, and that $G$ is continuous and invertible on $U$. Then there is a unique solution $(P_x)_{x \in \hat U}$ to the generalized martingale problem for $F,G$. Moreover, $(P_x)_{x \in U}$ has the Feller property and $(P_x)_{x \in \hat U}$ has the strong Markov property. If, in addition, $F$ and $G$ are bounded on bounded subsets of $U$, then for each $x \in U$, $P_x$-a.s., if $\tau(U)<\infty$ then as $t \to \tau(U)^-$ either $|X(t)| \to \infty$ or $d(X(t),z)\to 0$ for some $z \in \ptl U$, where $d$ is Euclidean distance.
\end{lemma}
\begin{proof}
Everything except the last statement belongs to Theorem 13.1 in Chapter 1 of \cite{pinsky}; the last statement is proved in the Appendix.
\end{proof}

\subsection{Quasi-diffusions}
Next we define quasi-diffusion (QD), for which the following fact provides motivation: namely, that the solution to the generalized martingale problem given in Lemma \ref{lem:SDExist} is uniquely characterized by the following properties:
\enumrom
\item for each $x \in U$, $P_x(X(0)=x)=1$,
\item $t\mapsto X(t)$ is a.s.~continuous with respect to Euclidean distance for all $t\in [0,\tau(U))$, and
\item the following processes are local martingales:
\begin{align*}
& X^m(t) := X(t) - X(0) - \int_0^t F(X(s))ds \ \ \text{and} \\
& X^m(t)(X^m)^{\top}(t) - \int_0^t G(X(s))ds.\end{align*}
\enumend
In (iii), ``local'' refers to a localizing sequence that increases to $\tau(U)$, as for example $(\tau(D_n))$ in Definition \ref{def:mp}. The forward implication -- that the solutions described by Lemma \ref{lem:SDExist} have these properties -- follows by choosing $f(x)=x_i$, $i=1,\dots, d$ for the first expression in (iii) and $f(x)=x_ix_j$, $i,j=1,\dots,d$ for the second. The interested reader may deduce the reverse implication with the help of It{\^o}'s formula (Chapter 1, Theorem 4.57 in \cite{jacod}), although we will not need the rigorous result. Using this idea, roughly speaking, a family of processes $(x_\ep)$ should be a QD with coefficients $F$ and $G$ if, for small $\ep>0$, sample paths are nearly continuous, i.e., have only small jumps, and the processes $x_\ep^m(t):= x_\ep(t)-\int_0^t F(x_\ep(s))ds$ and $\smash x_\ep^m(t)(x_\ep^m)(t))^{\top} - \int_0^t G(x_\ep(s))ds$ are nearly martingales. For a QDP the definition will be similar, except with $\ep^2 \int_0^t G(x_\ep(s))ds$ in place of $\int_0^t G(X_\ep(s))ds$.\\

In order to be precise about ``nearly martingales'', a QD and QDP should have first- and second-order martingales similar in appearance to those in property (iii) above. As is tradition in the theory of stochastic processes, we shall generally assume processes are defined on a filtered probability space  satisfying the usual conditions, and adapted to the filtration. Processes are defined on a time interval $[0,\zeta)$, where $\zeta \in (0,\infty]$ is a predictable time that we call the terminal time, and may be finite. For a particular $X$, $\zeta(X)$ denotes the terminal time of $X$. To describe QDPs, we will work with the class of semimartingales. A semimartingale (SM) is a \cdlg~(right continuous with left limits) process $X$ that can be written
$$X=X(0) + M+A,$$
where $M,A$ are \cdlg, $M$ is a local martingale and $A$ has locally finite variation. For example, the solutions described by Lemma \ref{lem:SDExist} are semimartingales, based on the decomposition $X(t) = X(0) + X^m(t) + \int_0^t F(X(s))ds$ given by (iii) above. A semimartingale is said to be special if it has a decomposition of the above type for which $A$ is predictable. In this case, the decomposition is unique (see \cite[I.4.22]{jacod}) and we denote $M$ and $A$ by $X^m$ and $X^p$, and refer to them as the martingale part and compensator, respectively, in agreement with the standard definition of compensator. For a locally square-integrable martingale $M$, we denote by $[M]$ and $\lng M \rng$ the quadratic variation (qv) and predictable quadratic variation (pqv). We will say that a SM is locally $L^2$ if it is special and if $X^m$ is locally square-integrable.\\

For a \cdlg~process $X$ and $c>0$, define the operators $J_c$ and $J^c$ by
$$J^cX = \sum_{s \le t} \Delta X(s)\1(|X(s)|>c)$$
and $J_cX = X - J^cX$. Since a \cdlg~function has only finitely many jumps of a given minimum size on any finite time interval, $J^cX$ makes sense. Since $J^cX$ is \cdlg~and has finite variation, if $X$ is a SM then so are $J^cX$ and $J_cX$, moreover $|\Delta J_cX| \le c$. As shown in \cite[I.4.24]{jacod}, is $Y$ is a SM with $|\Delta Y| \le c$ then $Y$ is special and $|Y^p| \le c$, $|Y^m| \le 2c$. So, $J_cX$ is special and $|(J_cX)^p| \le c$, $|(J_cX)^m| \le 2c$. As noted in \cite[I.4.1]{jacod}, martingales with bounded jumps are locally $L^2$, so $J_cX$ is locally $L^2$.\\

We are now ready to define QD and QDP, and to relate the notions to each other and to diffusions. We begin with the definition of a QD, and a result that shows a QD converges to a diffusion.

\begin{definition}[Quasi-diffusion]\label{def:qd}
Let $U \subset \R^d$ be an open set and let $F:U\to\R^n$ and $G:U \to M_+(\R,d)$ be measurable. Let $\E\subset \R_+$ be a countable set that accumulates at $0$. A family of semimartingales $(X_\ep)_{\ep \in \E}$ is a quasi-diffusion on $U$ with characteristics $F,G$ if for each domain $D\subset\subset U$, fixed $c,T>0$ and $\tau(D,\ep) := \inf\{t \colon X_\ep(t) \notin D \ \text{or} \ X_\ep(t^-) \notin D\}$, $\zeta(X_\ep) > \tau(D,\ep)$ and
\begin{enumerate}[noitemsep,label={(\roman*)}]
\item $P(J^c X_\ep(t)=0 \ \text{for all} \ t \le \tau(D,\ep)\wedge T) \to 1$ as $\ep \to 0$,
\item $\sup_{t \le \tau(D,\ep)\wedge T} \left| \, (J_cX_\ep)^p(t)  - \int_0^t F(X_\ep(s))ds \, \right| \cp 0 \ \ \text{as} \ \ \ep \to 0$, and
\item $\sup_{t \le \tau(D,\ep) \wedge T} \left| \lng (J_cX_\ep)^m \rng(t) - \int_0^t G(X_\ep(s))ds \, \right| \cp 0 \ \ \text{as} \ \ \ep \to 0$,
\end{enumerate}
with $\cp$ meaning convergence in probability.
\end{definition}

$F$ is called the drift and $G$, the diffusion. As shown in the next result, a quasi-diffusion with characteristics $F,G$ converges to the corresponding diffusion process. Since we allow for processes with explosion, the mode of convergence we'll use is a localized version of convergence in distribution.

\begin{definition}[Local convergence in distribution]\label{def:lcd}

Let $U\subset \R^d$ be an open and connected set, and suppose $X_n, n=1,2,\dots$ and $X$ are \cdlg~processes taking values in $\R^d$. For a bounded domain $D\subset \subset U$ define $\tau(D,n) = \inf\{t \colon X_n(t^-) \notin D \ \text{or} \ X_n(t) \notin D\}$, similarly define $\tau(D)$ for $X$, and assume that $\zeta(X_n) > \tau_n(D)$ and $\zeta(X)>\tau(D)$ for all $n$ and $D\subset\subset U$, where $\zeta(X_n),\zeta(X)$ are the terminal times. Say that $(X_n)$ converges locally in distribution to $X$ on $U$ as $n\to\infty$, writing $X_n \smash \lcd X$ on $U$, if there is a sequence of domains $D_1 \subset\subset D_2 \subset\subset \dots$ increasing to $U$ such that
$$\smash X_n(\cdot \wedge \tau(D_i,n)) \cd X(\cdot \wedge \tau(D_i))$$
for each $i$, where $\smash\cd$ is convergence in distribution with respect to the Skorohod topology.
\end{definition}

In the above, we assume the Euclidean metric on $\R^n$ is used for the Skorohod topology. Note that the solution $X$ described by Lemma \ref{lem:SDExist} takes values in $\hat U = U\cup\{\c\}$. Defining $\zeta(X)=\inf\{t \colon X(t)=\c\}$, the restriction of $X$ to the time interval $[0,\zeta(X))$ fits the context of Definition \ref{def:lcd}. 

\begin{lemma}[Quasi-diffusions converge to diffusions]\label{lem:diff-limit}
Let $U,F,G$ satisfy the assumptions of Lemma \ref{lem:SDExist}, and given $x \in U$ let $X$ denote the corresponding diffusion with $X(0)=x$. If $(X_\ep)$ is a QD with characteristics $F,G$ and $X_\ep(0) \cp x$ as $\ep \to 0$ then $X_\ep \smash \lcd X$ as $\ep \to 0$.
\end{lemma}

\begin{proof}
This is done in the Appendix.
\end{proof}

\subsection{Quasi-diffusive perturbations}

Now we give the definition of a QDP, and a result giving conditions for a rescaled QDP to be a QD. The definition of a QDP is made deliberately so that if $(x_\ep)$ is a QDP then provided that the rescaled characteristics convergence (see Lemma \ref{lem:QDP-QD}), rescaled processes of the form
\begin{align}\label{eq:rescale}
X_\ep(t) :=a_\ep(x_\ep(b_\ep t)-x_\star),
\end{align}
with $x_\star \in \R^d$ and $(a_\ep),(b_\ep)$ sequences of positive numbers, will be QDs.


\begin{definition}[Quasi-diffusive perturbation, isotropic case]\label{def:QDP}
Let $U \subset \R^d$ be an open set and let $F:U\to\R^n$ and $G:U \to M_+(\R,d)$ be measurable. Let $\E\subset \R_+$ be a countable set that accumulates at $0$ and let $(a_\ep)_{\ep\in \E}$, $(b_\ep)_{\ep\in \E}$, be sets of positive real numbers, and let $(D_\ep)_{\ep \in \E}$ be a collection of domains with $D_\ep \subset \subset U$ for each $\ep$. A family of semimartingales $(x_\ep)_{\ep \in \E}$ is a quasi-diffusive perturbation (QDP) on $(D_\ep)$ to scale $(a_\ep),(b_\ep)$ with characteristics $F,G$, if for fixed $c,T>0$ and $\tau_\ep := \inf\{t \colon x_\ep(t) \notin D_\ep \ \text{or} \ x_\ep(t^-) \notin D_\ep\}$, $\zeta(x_\ep) > \tau_\ep$ and
\begin{enumerate}[noitemsep,label={(\roman*)}]
\item $P(J^{c/a_\ep}x_\ep(t)=0 \ \text{for all} \ t \le \tau_\ep\wedge b_\ep T) \to 1$ as $\ep \to 0$,
\item $a_\ep \sup_{t \le \tau_\ep\wedge b_\ep T} \left| \, (J_{c/a_\ep}x_\ep)^p(t)  - \int_0^t F(x_\ep(s))ds \, \right|  \cp 0$, and
\item $ a_\ep^2 \sup_{t \le \tau_\ep \wedge b_\ep T} \left| \, \lng (J_{c/a_\ep}x_\ep)^m \rng(t) - \ep^2\int_0^t G(x_\ep(s))ds \, \right| \cp 0$,
\end{enumerate}
where $\cp$ denotes convergence in probability.\\
$(x_\ep)$ is a QDP on $U$ if $(D_\ep)$ can be chosen such that $\lim_{\ep \to 0} \bigcup_{\ep' \le \ep}D_{\ep'}=U$.
\end{definition}

As with QD, $F$ is referred to as drift and $G$, as diffusion. The first condition says that larger than $c/a_\ep$ jumps contribute nothing on the time interval $[0,b_\ep T]$, while conditions (ii) and (iii) say that the compensator and pqv are well approximated, to order $a_\ep$ and $a_\ep^2$, by the pathwise integrals of $F(x_\ep)$ and $ \ep^2 G(x_\ep)$, respectively -- the reason for this choice of error bounds becomes clear in Lemma \ref{lem:QDP-QD}. We record a few notes on the definition.
\begin{itemize}[noitemsep]
\item Restricting $D_\ep$ or decreasing $b_\ep$ can lead to better estimates on drift and diffusion error.
\item By definition, $J_{c/a_\ep}x_\ep$ has bounded jumps, so is locally $L^2$, so we need not assume the same of $x_\ep$.
\item If $F,G$ are locally bounded, then (i) implies that $x_\ep$ can be replaced with $J_{\dlt c_\ep} x_\ep$ in (ii)-(iii).
\item Even if the processes $(x_\ep)$ are locally $L^2$, (i) does not imply that $J_{\dlt a_\ep}x_\ep$ can be replaced with $x_\ep$ in (ii)-(iii), if $x_\ep$ has increasingly large and infrequent jumps as $\ep \to 0$.
\item For given $F,G$ satisfying the conditions of Lemma \ref{lem:SDExist}, the solution to the generalized martingale problem, trivially parametrized by $\ep>0$, is a QDP on $U$ to any order $(a_\ep)$, on any time scale $(b_\ep)$.
\end{itemize}

Now we give a result stating conditions for a rescaled QDP to be a QD. It is clear from this result that the definition of QDP is tailored to minimize the requirements for convergence.

\begin{lemma}\label{lem:QDP-QD}
Let $(x_\ep)$ be a family of semimartingales, and let $x_\star \in \R^d$, and $\tld U \subset \R^d$ be an open set whose closure contains the origin. Suppose that for every domain $D\subset\subset \tld U$, $(x_\ep)$ is a QDP on $(D_\ep):=(\{x_\star + x/a_\ep\colon x \in D\})$ to scale $(a_\ep),(b_\ep)$ with characteristics $F,G$, and that the following limits exist for every $x \in \tld U$:
\begin{align}\label{eq:QD-coeff}
\tF(x) &:= \lim_{\ep \to 0}a_\ep b_\ep F(x_\star  + x/a_\ep) \ \text{and} \nonumber \\
\tG(x) &:= \lim_{\ep \to 0}\ep^2 a_\ep^2 b_\ep G(x_\star + x/a_\ep), 
\end{align}
with uniform convergence on compact subsets of $\tld U$. Then $(X_\ep)$ defined by \eqref{eq:rescale} is a QD on $\tld U$ with characteristics $\tF,\tG$ given by \eqref{eq:QD-coeff}.
\end{lemma}
\begin{proof}
For $c>0$ we have
$$J^c X_\ep(t) = a_\ep J^{c/a_\ep} x_\ep(b_\ep t),$$
so condition (i) of Definition \ref{def:qd} follows from condition (i) of Definition \ref{def:QDP}. Next, the map $X\mapsto X^p$ on special semimartingales is linear, and the map $X\mapsto \lng X^m\rng$ on locally $L^2$ semimartingales is homogeneous of degree 2. Thus, it follows from \eqref{eq:rescale} that
$$(J_cX_\ep)^p(t) = a_\ep \,(J_{c/a_\ep}x_\ep)^p(b_\ep t) \quad \text{and} \quad \lng (J_cX_\ep)^m\rng(t) = a_\ep^2 \lng (J_{c/a_\ep}x_\ep)^m \rng(b_\ep t).$$
Moreover,
\begin{align*}
&\int_0^{b_\ep t}F(x_\ep(s))ds = \int_0^t F(x_\ep(b_\ep s))d(b_\ep s) = b_\ep \int_0^t F(x_\star + X_\ep(s)/a_\ep)ds, \\ 
& \text{similarly} \ \ \int_0^{b_\ep t}G(x_\ep(s))ds = b_\ep \int_0^t G(x_\star + X_\ep(s)/a_\ep)ds.
\end{align*}
Therefore,
\begin{align}\label{eq:rescale-error}
\left | \, (J_{c/a_\ep}x_\ep)^p(b_\ep t)  - \int_0^{b_\ep t} F(x_\ep(s))ds \, \right| = \frac{1}{a_\ep}\left | \, (J_cX_\ep)^p(t)  - a_\ep b_\ep \int_0^t F(x_\star + X_\ep(s)/a_\ep)ds \, \right|.
\end{align}
If convergence in \eqref{eq:QD-coeff} is uniform on compact sets, then with $\tau(D,\ep)$ as in Definition \eqref{def:qd},
$$\sup_{t \le \tau(D,\ep) \wedge T}\left |\int_0^t \tld F(X_\ep(s))ds - a_\ep b_\ep \int_0^t F(x_\star+X_\ep(s)/a_\ep)ds \right| \cp 0$$
as $\ep \to 0$, so using \eqref{eq:rescale-error}, condition (ii) in Definition \ref{def:qd} follows from condition (ii) in Definition \ref{def:QDP}. Next, observe that
$$\left| \, \lng (J_{c/a_\ep}x_\ep)^m \rng(b_\ep t) - \ep^2\int_0^{b_\ep t} G(x_\ep(s))ds \, \right| = \frac{1}{a_\ep^2} \left| \, \lng (J_c X_\ep)^m \rng (t) - \ep^2 \, a_\ep^2 b_\ep \int_0^t G(x_\star + X_\ep(s)/a_\ep)ds \right|.$$
Arguing as before, condition (iii) in Definition \ref{def:qd} follows from condition (iii) in Definition \ref{def:QDP}.
\end{proof}

\section{Isotropic limit scales}\label{sec:iso}

In this section we describe the possible limits of $Y_\ep = a_\ep(x(b_\ep t)-x_\star)$ where $(a_\ep),(b_\ep)$ are scalars, which we call isotropic scaling. We'll use the following asymptotic notation: for positive functions $f,g$, use $f(x) \asymp g(x)$ as $x\to a$ to mean that $\lim_{x \to a}f(x)/g(x)$ exists and takes its value in $(0,\infty)$, and $f(x) \ll g(x)$ to mean the same thing as $f=o(g)$, i.e., that $\lim_{x \to a} f(x)/g(x)=0$, and $f\gg g$ if $g\ll f$.\\

Lemma \ref{lem:QDP-QD} gives conditions under which a QDP, rescaled around a point $x_\star$, is a QD, and thus has a diffusion limit. The main requirement is the convergence condition \eqref{eq:QD-coeff}, which we recall: for some open set $\tld U$, uniformly over $x$ in any compact $K\subset \tld U$, the following limits exist:
\begin{align*}
\tF(x) &:= \lim_{\ep \to 0}a_\ep b_\ep F(x_\star  + x/a_\ep) \ \text{and} \\
\tG(x) &:= \lim_{\ep \to 0}\ep^2 a_\ep^2 b_\ep G(x_\star + x/a_\ep).\nonumber 
\end{align*}
These conditions are a bit opaque, so let's come up with a more intuitive and equivalent description. Convergence can be tidily broken down into three parts: shape, drift to diffusion ratio, and time scale. We begin with a brief description, then work back from \eqref{eq:QD-coeff} to flesh it out.
\enumrom
\item The functions $x\mapsto F(x_\star+x/a_\ep)$ and $x\mapsto G(x_\star+x/a_\ep)$ have a limiting shape as $\ep \to 0$.

\item The rescaled drift to diffusion ratio converges to $0$, $\infty$, or to a non-zero function of $x$.
\item The time scale is just long enough for either drift or diffusion to be non-vanishing.
\enumend

In particular, we ignore the trivial case $\tF \equiv 0$ and $\tG \equiv 0$. Let's now use \eqref{eq:QD-coeff} to frame these concepts.

\enumrom

\item \textit{Shape.} Working back from \eqref{eq:QD-coeff} and defining $f_\ep=1/(a_\ep b_\ep)$ and $g_\ep = 1/(\ep^2 a_\ep^2 b_\ep)$, for each $x$, as $\ep \to 0$,
\begin{align}\label{eq:FG-shape}
F(x_\star+x/a_\ep) &\sim f_\ep \,\tF(x) \quad \text{and} \nonumber \\
G(x_\star+x/a_\ep) &\sim g_\ep \,\tG(x).
\end{align}
In other words, $F$ and $G$ are asymptotically a product of a function of $\ep$, with a function of $x$. This is true if $F$ and $G$ are locally homogeneous around $x_\star$, i.e., if there exist $\lph,\beta \ge 0$ and homogeneous functions $Q,V$ of degrees $\lph,\beta$ respectively (i.e., such that $Q(rx)=r^\lph Q(x)$ and $V(rx)=r^\beta V(x)$ for real $r$), such that as $x\to 0$,
\begin{align}\label{eq:FG-hom}
F(x_\star+x) \sim Q(x) \quad \text{and} \quad G(x_\star+x) \sim V(x).
\end{align}
\eqref{eq:FG-hom} holds, for example, if $F,G$ have a Taylor expansion around $x_\star$, in which case $\lph,\beta$ are integers. When \eqref{eq:FG-hom} holds, if $a_\ep \to \infty$ as $\ep \to 0$ then for each $x$, as $\ep \to 0$,
\begin{align}\label{eq:FG-hom-shape}
F(x_\star+x/a_\ep) \sim (1/a_\ep)^\lph Q(x) \quad \text{and} \quad G(x_\star+x/a_\ep)\sim (1/a_\ep)^\beta V(x).
\end{align}
Let $h_\ep=a_\ep^{1-\lph}b_\ep$ and $\ell_\ep=\ep^2 a_\ep^{2-\beta} b_\ep$. Combining \eqref{eq:FG-shape} and \eqref{eq:FG-hom-shape}, for each $x$, as $\ep \to 0$,
\begin{align}\label{eq:tF-tG-ratio}
\tF(x) \sim h_\ep Q(x) \quad \text{and} \quad \tG(x) \sim \ell_\ep V(x).
\end{align}
If $Q,V$ are not identically zero, then each expression gives a dichotomy.
\itemgo
\item Either $h_\ep \to 0$, in which case $\tF \equiv 0$, or $h_\ep \asymp 1$, in which case $\tF \propto Q$.
\item Either $\ell_\ep\to 0$, in which case $\tG \equiv 0$, or $\ell_\ep \asymp 1$, in which case $\tG \propto V$.
\itemend
The converse is also true: if \eqref{eq:FG-hom} holds and $(h_\ep),(\ell_\ep)$ each satisfy one of the two above conditions then \eqref{eq:QD-coeff} holds with $\tF,\tG$ as described, and convergence is locally uniform.\\

\item \textit{Drift to diffusion ratio.} The ratio of the terms in \eqref{eq:QD-coeff} is
\begin{align}\label{eq:rsc-dd-ratio}
\frac{|a_\ep b_\ep F(x_\star+x/a_\ep)|}{|\ep^2 a_\ep^2 b_\ep G(x_\star + x/a_\ep)|} = \frac{|F(x_\star+x/a_\ep)|}{\ep^2 a_\ep |G(x_\star + x/a_\ep)|}.
\end{align}
In particular, the ratio depends on $a_\ep$ but not on $b_\ep$.  It can also be written $\ep^{-2}\, \th(1/a_\ep,x)$, where the drift : diffusion ratio function $\theta$ is defined by
\begin{align}\label{eq:dd-th}
\theta(u,x) = \frac{ u \, |F(x_\star + u\,x)|}{|G(x_\star+u\,x)|}.
\end{align}
If \eqref{eq:FG-hom} holds and $x$ is such that $Q(x)\ne 0$ and $V(x) \ne 0$, then as $u\to 0$
$$\th(u,x) \sim \frac{u |Q(ux)|}{|V(ux)|} = \frac{u^{1+\lph}|Q(x)|}{u^\beta |V(x)|} = u^{1+\lph-\beta} \frac{|Q(x)|}{|V(x)|}.$$
In particular, $\th(1/a_\ep,x) \asymp (1/a_\ep)^{1+\lph-\beta}$ for such $x$. Using this as inspiration and referring to (i), we see that $\ep^{-2}(1/a_\ep)^{1+\lph-\beta} = h_\ep / \ell_\ep$, so if both $h_\ep \to 0$ or $h_\ep \asymp 1$, and $\ell_\ep \to 0$ or $\ell_\ep \asymp 1$, then ignoring the trivial case where both tend to $0$, there are three relevant cases, that we refer to collectively as limit scales for the drift to diffusion ratio.
\itemgo
\item If $(1/a_\ep)^{1+\lph-\beta}\ll \ep^2$ then $h_\ep \ll \ell_\ep$ (diffusion dominates).
\item If $(1/a_\ep)^{1+\lph-\beta} \gg \ep^2$ then $\ell_\ep \ll h_\ep$  (drift dominates).
\item If $(1/a_\ep)^{1+\lph-\beta} \asymp \ep^2$ then $\ell_\ep \asymp h_\ep$ (drift matches diffusion).
\itemend 

\item \textit{Time scale.} Given $(a_\ep)$ satisfying one of the above three cases, $(b_\ep)$ should be chosen just large enough that at least one of $\tF,\tG$ is non-zero. If \eqref{eq:QD-coeff} and \eqref{eq:FG-hom} hold then from \eqref{eq:tF-tG-ratio}, $h_\ep \asymp 1 \Lra \tF\propto Q$ and $\ell_\ep \asymp 1 \Lra \tG\propto R$. Recalling that $h_\ep=a_\ep^{1-\lph}b_\ep$ and $\ell_\ep = \ep^2 a^{2-\beta} b_\ep$, 
\itemgo
\item $h_\ep \asymp 1 \Lra b_\ep \asymp a_\ep^{\lph-1}$ and
\item $\ell_\ep \asymp 1 \Lra b_\ep \asymp \ep^{-2}a_\ep^{\beta-2}$.
\itemend
We refer to this choice of $(b_\ep)$ in each case as the visible time scale corresponding to $(a_\ep)$, since it is the time scale at which changes are visible (not too slow, and not too fast).

\enumend
Combining the observations of (i)-(iii), if \eqref{eq:FG-hom} holds, $(1/a_\ep)$ is a limit scale for the drift to diffusion ratio, and $(b_\ep)$ is the visible time scale corresponding to $(a_\ep)$, then \eqref{eq:QD-coeff} holds with $\tF,\tG$ determined as follows.
\itemgo
\item If $(1/a_\ep)^{1+\lph-\beta } \ll \ep^2$ and $b_\ep \asymp \ep^{-2} a_\ep^{\beta-2}$ then $\tld F=0$ and $\tld G \propto V$.
\item If $(1/a_\ep)^{1+\lph-\beta } \gg \ep^2$ and $b_\ep \asymp a_\ep^{\lph-1}$ then $\tld F \propto Q$ and $\tld G=0$.
\item If $(1/a_\ep)^{1+\lph-\beta } \asymp \ep^2$ and $b_\ep \asymp a_\ep^{\lph-1}$ (which is then $\asymp \ep^{-2}a_\ep^{\beta-2}$), then $\tld F \propto Q$ and $\tld G \propto V$.
\itemend

To resolve the inequality on $(1/a_\ep)$, we need an assumption on the sign of $1+\lph-\beta $. To study the effect of the sign, we'll consider $(1/a_\ep)^{1+\lph-\beta} = O(\ep^2)$ which corresponds to non-zero diffusion.
\itemgo
\item If $1+\lph-\beta >0$, $(1/a_\ep)^{1+\lph-\beta }=O(\ep^2)$ iff $a_\ep= \ep^{-2/(1+\lph-\beta )} \gg 1$.
\item If $1+\lph-\beta =0$, $(1/a_\ep)^{1+\lph-\beta }=1$ is not $O(\ep^2)$ as $\ep \to 0$.
\item If $1+\lph-\beta <0$, $(1/a_\ep)^{1+\lph-\beta }=O(\ep^2)$ iff $a_\ep=\ep^{2/(\beta-\lph-1)} \ll 1$.
\itemend
In particular, if $1+\lph-\beta \le 0$ and $a_\ep \to \infty$ as $\ep \to 0$ then diffusion does not occur, i.e., the only limit possible for $\tld G$ is $0$. This case can be excluded if $F=O(G)$, i.e., $|F(x)| \le C|G(x)|$ for some $C>0$ and all $x$, since then $\lph \ge \beta$ and $1+\lph-\beta$ has positive sign. Let us give this property a name.
\begin{definition}\label{def:ssQDP}
A QDP is \emph{strongly stochastic} if its characteristics $F,G$ satisfy $F=O(G)$, i.e., if there exists $C>0$ such that $|F(x)| \le C\, |G(x)|$ for every $x \in U$.
\end{definition}

From the above discussion and Lemma \ref{lem:QDP-QD} we obtain the following result.

\begin{theorem}[Limit processes for QDPs under isotropic scaling]\label{thm:iso-limits}
Let $U,F,G$ be as in Definition \ref{def:QDP}. Let $(x_\ep)$ be a family of semimartingales, and let $x_\star \in \R^d$. Let $\tld U$ be a non-empty open convex set whose closure contains $0$. Suppose that for each domain $D\subset\subset \tld U$, $(x_\ep)$ is a strongly stochastic QDP on $(D_\ep):=(\{x_\star + x/a_\ep\colon x \in D\})$ to scale $(a_\ep),(b_\ep)$, with coefficients $F,G$ that are locally homogeneous around $x_\star$, i.e., that satisfy \eqref{eq:FG-hom}.\\

Let $Y_\ep(t) = a_\ep(x_\ep(b_\ep t)-x_\star)$ and suppose that $a_\epsilon \to \infty$. Suppose $(1/a_\ep)$ is a limit scale for the drift to diffusion ratio, and that $(b_\ep)$ is the visible time scale corresponding to $(a_\ep)$, as described above. Then one of the three cases below holds, and $(Y_\ep)$ is a QD with characteristics $\tF,\tG$ as given.

\enumrom
\item If $1/a_\ep \ll \ep^{2/(1+\lph-\beta )}$ and $b_\ep \asymp \ep^{-2}a_\ep^{\beta-2}$, then $\tF \propto 0$ and $\tG \propto V$.
\item If $1/a_\ep \gg \ep^{2/(1+\lph-\beta )}$ and $b_\ep \asymp a_\ep^{\lph-1}$, then $\tF \propto Q$ and $\tG \propto 0$.
\item If $1/a_\ep \asymp \ep^{2/(1+\lph-\beta )}$ and $b_\ep \asymp a_\ep^{\lph-1} \asymp \ep^{-2} a_\ep^{\beta-2}$, then $\tF \propto Q$ and $\tG \propto V$.

\enumend
\end{theorem}

We take a moment to name the different cases in Theorem \ref{thm:iso-limits}.

\begin{definition}[limit ranges]\label{def:lim-scales}
In the context of Theorem \ref{thm:iso-limits}, say that $(a_\ep),(b_\ep)$ is a \emph{QD limit scale} for $(x_\ep)$ if it satisfies one of the three sets of conditions above. Define the three \emph{limit ranges} as follows:
\enumrom
\item $1/a_\ep \ll \ep^{2/(1+\lph-\beta)}$ is the \emph{pure diffusive range},
\item $1/a_\ep \gg \ep^{2/(1+\lph-\beta)}$ is the \emph{deterministic range}, and
\item $1/a_\ep \asymp \ep^{2/(1+\lph-\beta)}$ is the \emph{drift-diffusion (dd) scale}.
\enumend

\end{definition}

Theorem \ref{thm:iso-limits} tells a nice story. When we assume strong stochasticity, characteristics scale in such a way that diffusion dominates at smaller scales, while drift dominates at larger scales. The dd scale separates the two regimes and is also the largest scale at which, from the fixed vantage point $x_\star$, random fluctuations are not drowned out by the deterministic flow. Moreover, once the spatial scale $(1/a_\ep)$ has been chosen, there is a unique choice of time scale $(b_\ep)$, modulo multiplication by a non-zero constant, that leads to a non-trivial limit process; on any faster (shorter) time scale, no change would be observed, while on any slower (longer) time scale the change would be instantaneous, as either the drift or diffusion coefficient would be divergent.

\section{Limit scales for parametrized QDPs}\label{sec:prmtzd}

In this section we extend the analysis of Section \ref{sec:iso} to the case where $F,G$ also depend on a parameter $\lbd\in \R$. In other words, we consider a family of processes $(x_\ep(t;\lbd))$ indexed by $\ep$ and $\lbd$, with corresponding characteristics $F(\cdot,\lbd),G(\cdot,\lbd)$, that we wish to study in the neighbourhood of a point $(x_\star,\lbd_\star)$. We will study two kinds of limits:
\enumrom
\item \textit{Constant branch:} $Y_\ep(t;\lbd_\ep):= a_\ep( x_\ep(b_\ep t;\lbd_\ep)-x_\star)$, and
\item \textit{Non-constant branch:} $Y_\ep(t;\lbd_\ep):= a_\ep( x_\ep(b_\ep t;\lbd_\ep)-x_\star(\lbd_\ep))$, \\ for some function $x_\star(\lbd)$ satisfying $x_\star(\lbd_\star)=x_\star$.
\enumend
In both cases the goal is to find $(a_\ep),(b_\ep)$ and $(\lbd_\ep)$ with $a_\ep \to \infty$ and $\lbd_\ep \to \lbd_\star$ such that the rescaled process is a quasi-diffusion. For simplicity, we consider only a single spatial dimension, i.e., $x \in \R$. Instead of the homogeneous assumption \eqref{eq:FG-hom} on $F,G$, we'll assume existence of a Taylor expansion in $x$ and $\lbd$ around $(x_\star,\lbd_\star)$, then partition the $o(1)$ neighbourhood of $(x_\star,\lbd_\star)$ into regions where $F,G$ have a limiting shape with respect to $x$. We begin by precisely defining a parametrized QDP, and translate Lemma 3.6 to fit that context. 

\begin{definition}[Parametrized QDP]\label{def:prmtrzdQDP}
Let $\lbd$ be a parameter taking values in an interval $I\subset \R$. For each $\lbd \in I$, let $U,F(\,\cdot\,,\lbd),G(\,\cdot\,,\lbd),\E,(a_\ep),(b_\ep),(D_\ep)$ be as in Definition \ref{def:QDP}, let $(x_\ep(t;\lbd))_{\ep \in \E,\,\lbd \in I,\,t\ge 0}$ be a collection of semimartingales, and fix a sequence $(\lbd_\ep)$. Then $(x_\ep(t,\lbd_\ep))$ is a parametrized QDP on $(D_\ep)$ to order $(a_\ep)$ on time scale $(b_\ep)$, with characteristics $F,G$, if the estimates of Definition \ref{def:QDP} hold with $x_\ep(t;\lbd_\ep)$, $F(\,\cdot,\,\lbd_\ep)$ and $G(\,\cdot,\,\lbd_\ep)$ in place of $x_\ep$, $F$ and $G$. 
\end{definition}

The following trivial corollary of Lemma \ref{lem:QDP-QD} gives conditions for a parametrized QDP to be a QD.

\begin{corollary}\label{cor:QDP-QD}
Let $x_\star,\tld U$ be as in Lemma \ref{lem:QDP-QD} and suppose that for each domain $D\subset\subset \tld U$, $(x_\ep(t;\lbd_\ep))$ is a parametrized QDP on $(D_\ep):=(\{x_\star+x/a_\ep\colon x \in D\})$ to scale $(a_\ep),(b_\ep)$, with characteristics $F,G$, and that the following limits exist for every $x \in \tld U$:
\begin{align}\label{eq:FG-par-lim}
\tF(x) &:= \lim_{\ep \to 0}a_\ep b_\ep F(x_\star + x/a_\ep,\lbd_\ep) \ \text{and} \nonumber \\
\tG(x) &:= \lim_{\ep \to 0}\ep^2 a_\ep^2 b_\ep G(x_\star + x/a_\ep,\lbd_\ep),
\end{align}
with uniform convergence on compact subsets of $\tld U$. Then the process $(X_\ep)$ defined below is a QD with characteristics $\tld F,\tld G$:
\begin{align*}
X_\ep(t) =a_\ep(x_\ep(b_\ep t;\lbd_\ep)-x_\star).
\end{align*}
\end{corollary}

\begin{proof}
The proof is identical to the proof of Lemma \ref{lem:QDP-QD}, except with $x_\ep(\cdot;\lbd_\ep),F(\cdot,\lbd_\ep)$ and $G(\cdot,\lbd_\ep)$ in place of $x_\ep,F$ and $G$.
\end{proof}

To find $(a_\ep),(b_\ep),(\lbd_\ep)$ such that $F,G$ can satisfy \eqref{eq:FG-par-lim}, as in Section \ref{sec:iso} we use the three-step philosophy of shape, drift to diffusion ratio, and time scale. We begin with case (i), the constant branch, then treat case (ii) in Section \ref{sec:float-equil}. Since $F,G$ now depend on two variables, the step that will occupy the most effort is shape. We shall begin the discussion in the same manner as Section \ref{sec:iso}, then take a theoretical detour to address shape before returning to address the remaining two steps. Since the notation gets more complex, we'll mostly use $(x_\star,\lbd_\star)=(0,0)$ to formulate results.\\

\nid\textit{Shape.} Working back from \eqref{eq:FG-par-lim} with $(x_\star,\lbd_\star)=(0,0)$ and defining $f_\ep=1/(a_\ep b_\ep)$ and $g_\ep = 1/(\ep^2 a_\ep^2 b_\ep)$,
\begin{align}\label{eq:FG-flip-lim}
F(x/a_\ep, \lbd_\ep) &\sim f_\ep \,\tF(x) \quad \text{and} \nonumber \\
G(x/a_\ep, \lbd_\ep) &\sim g_\ep \,\tG(x).
\end{align}
In other words, the space and parameter scales $(1/a_\ep,\lbd_\ep)$ are such that on that scale, $F,G$ are asymptotically the product of a function of $\ep$ with a function of $x$. Since we now have a free variable $\lbd$, it is no longer necessary that $F,G$ be locally homogeneous. Instead, the correct generalization to this context is that $(1/a_\ep,\lbd_\ep)$ be a limit scale for $F$ and $G$, in the following sense.

\begin{definition}[limit scale]\label{def:homscale}
Let $f$ be defined in a neighbourhood of $(0,0)$. Say that a sequence $(x_n,\lbd_n)$ with limit $(0,0)$ is a limit scale for $f$ if there are functions $v,w$, called the scale functions, such that $f(ux,\lbd)\sim v(x,\lbd)w(u)$ as $n\to\infty$, locally uniformly with respect to $u \in (0,\infty)$.
\end{definition}
If $(1/a_\ep,\lbd_\ep)$ is a limit scale for both $F$ and $G$, then letting $v_F,w_F$ and $v_G,w_G$ denote the respective scale functions and rearranging, \eqref{eq:FG-flip-lim} gives
\begin{align}\label{eq:tF-tG-ratio2}
\tF(x) \sim h_\ep Q(x) \quad \text{and} \quad \tG(x) \sim \ell_\ep V(x),
\end{align}
where, using $f_\ep=1/(a_\ep b_\ep)$ and $g_\ep=1/(\ep^2 a_\ep b_\ep)$,
\begin{align}\label{eq:h-ell}
&Q(x)=w_F(x), \quad V(x)=w_G(x), \nonumber \\
&h_\ep = a_\ep b_\ep v_F(1/a_\ep,\lbd_\ep) \ \ \text{and} \ \ \ell_\ep = \ep^2 a_\ep b_\ep v_G(1/a_\ep,\lbd_\ep).
\end{align}
We will return to these expressions after laying out a method for identifying limit scales and scale functions. To get warmed up, suppose $f$ has a Taylor expansion around $(0,0)$ in the sense that
\begin{align}\label{eq:loTaylor}
f(x,\lbd) \sim \sum_{\lph \in A}c_\lph x^{\lph_1}\lbd^{\lph_2} \ \ \text{as} \ \ |x|+|\lbd| \to 0
\end{align}
for some finite set of powers $A\subset \N^2$ and non-zero coefficients $(c_\lph)_{\lph \in A}$. Then any sequence $(x_n,\lbd_n)$ tending to $(0,0)$ along a curve of the form $x=\lbd^m$ for some $m>0$ is a limit scale for $f$, since if $x=\lbd^m$ then letting $s(\lph,m) = m \lph_1+\lph_2$, $s(A,m) = \min \{m\lph_1+\lph\colon \lph \in A\}$ and $A(m) = \{\lph \in A\colon s(\lph,m) = s(A,m)$, for $u\in (0,\infty)$, as $x\to 0$,
$$f(ux,\lbd) \sim \sum_{\lph \in A}c_\lph u^{\lph_1}\lbd^{m \lph_1 + \lph_2} \sim \lbd^{s(A,m)}\sum_{\lph \in A(m)}c_\lph u^{\lph_1},$$
with uniform convergence over $u$ in compact subsets of $(0,\infty)$. By studying the structure of the sets $(A(m))_{m\in (0,\infty)}$ we can identify all the limit scales for $f$ satisfying \eqref{eq:loTaylor}. We begin with some simple theory of partially ordered sets.

\subsection{Envelope of a partially ordered subset of $\N^2$}\label{sec:env}

Let $(S,\le)$ be a partially ordered set. Elements $a,b$ are incomparable if neither $a\le b$ nor $b\le a$ is true, or equivalently, if $a \ne b$ and neither $a<b$ nor $b<a$ is true. A subset $A\subset S$ is an anti-chain (I prefer the term disordered) if $a,b \in A$ and $a\le b$ implies $a=b$, or equivalently, if all distinct pairs of elements in $A$ are incomparable. If $x \in A$ then $x$ is minimal in $A$ if $y<x$ implies $y\notin A$. For any $A$, the set $\min(A):= \{x\in A\colon x \ \text{is minimal in} \ A \}$ is disordered. If $\min(A)\subset B\subset A$ then $\min(B)=\min(A)$. If $A$ is disordered then $A=\min(A)$. If $A$ is disordered and and $B\subset A$ then $B$ is disordered.\\

Define the partially ordered set $(\N^2,\le)$ by $(i,j) \le (k,\ell)$ if $i\le k$ and $j \le \ell$.

\begin{lemma}
If $A\subset \N^2$ is disordered, then it is finite.
\end{lemma}

\begin{proof}
Suppose $(j,k) \in A$. If $(\ell,m)\in A$ and $(\ell,m) \ne (j,k)$ then either $\ell<j$ or $m<k$. Moreover, if $(\ell,m) \in A$ then $(\ell,m') \notin A$ for all $m' \ne m$ and $(\ell',m) \notin A$ for all $\ell' \ne \ell$, so for each $\ell \in \{0,\dots,j-1\}$ there is at most one $m$ such that $(\ell,m) \in A$, and for each $m \in \{0,\dots,k-1\}$ there is at most one $\ell$ such that $(\ell,m) \in A$. In other words, in addition to $(j,k)$, $A$ contains at most $j+k$ other elements.
\end{proof}

If $A \subset \N^2$ is disordered then since $(\ell,m) \in A \Rightarrow (\ell',m),(\ell,m') \notin A$ for any $\ell'\ne \ell$ or $m' \ne m$, so $A$ can be arranged in increasing or decreasing order of either the first or second coordinate. If $\lph,\lph'$ are incomparable then $\sgn(\lph_1-\lph_1')=-\sgn(\lph_2-\lph_2')$, so if $A$ is arranged as $\lph^{(1)},\dots,\lph^{(n)}$ with $\lph_1^{(1)} < \dots < \lph_1^{(n)}$, then $\lph_2^{(1)} > \dots > \lph_2^{(n)}$.\\

For $(x,y)\in \R^2$ and $m\in (0,\infty)$ let $s((x,y),m) = mx + y$. A set $A\subset \N^2$ is an envelope if, for each $\lph \in A$, there is $m\in (0,\infty)$ such that $s(\lph,m) \le s(\lph',m)$ for all $\lph' \in A$. Envelopes are disordered, since if $(k,\ell) \le (i,j)$ then letting $m$ correspond to $(i,j)$, $mi+j\le mk+\ell$ which together imply that $(k,\ell)=(i,j)$. Not all disordered sets are envelopes. For example, take $\{(3,0),(2,2),(0,3)\}$, which is disordered. Then $s((3,0),m)<s((2,2),m)$ for $m<2$ while $s((0,3),m)<s((2,2),m)$ for $m>1/2$.\\

For any $A\subset \N^2$ and $m\in (0,\infty)$ let $s(A,m) = \min\{s(\lph,m)\colon \lph \in A\}$ and let $A(m) = \{\lph\in A \colon s(\lph,m)=s(A,m)\}$. Then $\env(A) := \bigcup_{m\in(0,\infty)} A(m) =\{\lph \in A\colon s(\lph,m)=s(A,m) \ \text{for some} \ m \in (0,\infty)\}$ is an envelope, called the envelope of $A$. If $A$ is an envelope then $A=\env(A)$. If $\env(A) \subset B \subset A$ then $\env(B)=\env(A)$. If $\lph' < \lph$ then $s(\lph',m) < s(\lph,m)$, so if $s(\lph,m) = s(A,m)$ then $\lph \in \min(A)$. In particular, $\env(A) \subset \min(A)$. Since $\min(A) \subset A$ it follows that $\env(\min(A)) = \env(A)$. A visual example of a set $A$ together with $\min(A)$ and $\env(A)$ is given in Figure \ref{fig:set-min-env}.\\
.
\begin{figure}
\centering
\includegraphics[width=1.5in]{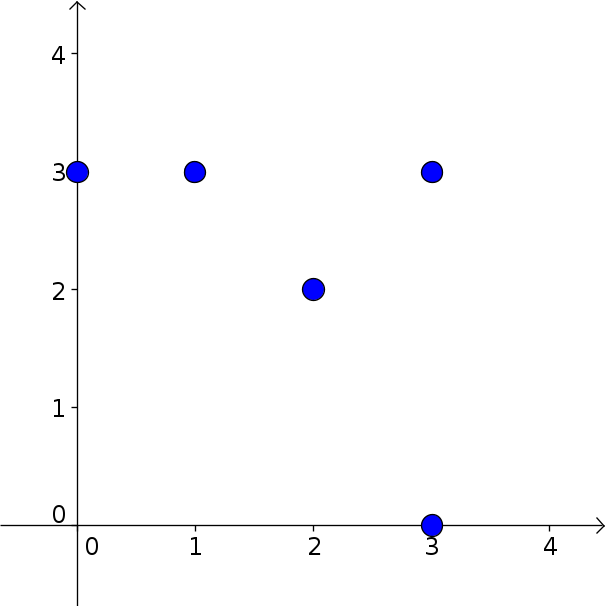}
\hspace{.4in}
\includegraphics[width=1.5in]{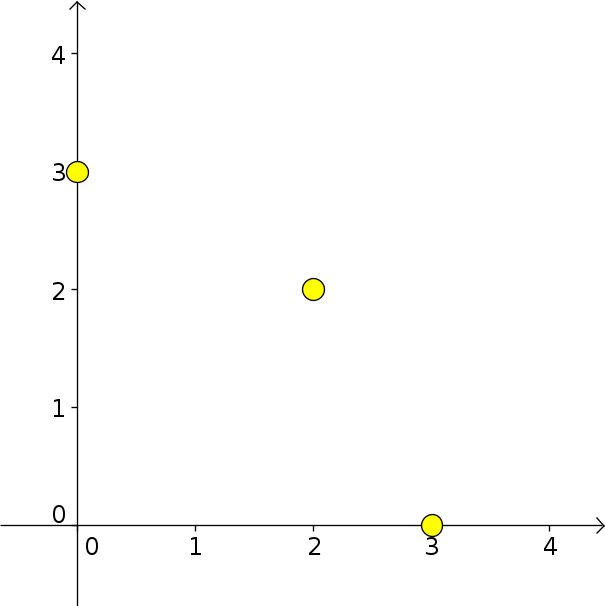}
\hspace{.4in}
\includegraphics[width=1.5in]{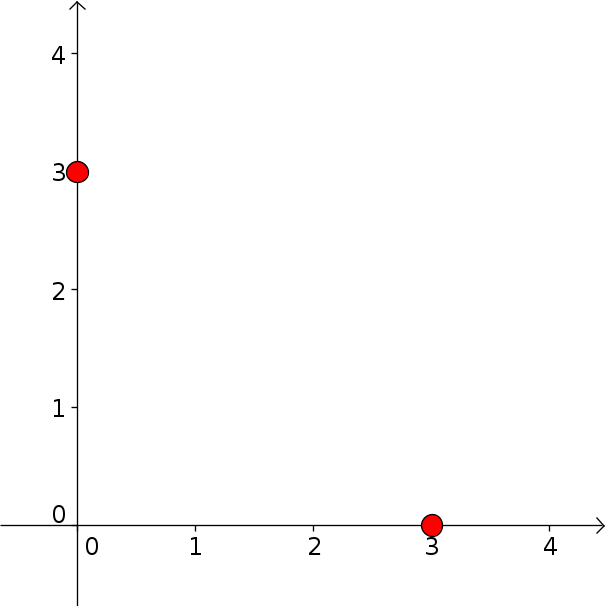}
\caption{The set $A=\{(3,0),(2,2),(0,3),(1,3),(3,3)\}$ in blue, with $\min(A)$ in yellow and $\env(A)$ in red.}
\label{fig:set-min-env}
\end{figure}

To understand $\env(A)$ we examine the sets $(A(m))_{m\in (0,\infty)}$, which are explained by the following lemma.

\begin{lemma}\label{lem:envelope}
Let $A\subset \N^2$ and let $(A(m))_{m\in (0,\infty)}$ be as above. Define the pivot slopes and pivots of $A$ by
$$M(A):= \{m \in (0,\infty) \colon |A(m)| \ge 2\}$$ and
$$\piv(A):= \{\lph \in A\colon A(m)=\{\lph\} \ \text{for  some} \ m \in (0,\infty)\}.$$
Then $M(A)$ is a finite set $0<m_1<\dots<m_n<\infty$, and  letting $m_0:=0$ and $m_{n+1}:=\infty$, $\piv(A) = \{\lph(i)\}_{i=0}^n$, where $A(m) = \lph(i)$ for all $m\in (m_i,m_{i+1})$, and $\lph(i-1),\lph(i) \in A(m_i)$ for each $i$. In addition,
\itemgo
\item $i\mapsto \lph_1(i)$ is decreasing,
\item $\lph_1(i) \le \lph_1 \le \lph_1(i-1)$ for $\lph \in A(m_i)$,
\item $\lph(0) = \arg\max\{\lph_1\colon (\lph_1,\lph_2) \in \min(A)\}$, and
\item $\lph(n) = \arg\min\{\lph_1\colon (\lph_1,\lph_2) \in \min(A)\}$.
\itemend

Finally, for each $m\in (0,\infty)$ there is $\lph \in \piv(A)$ such that $s(A,m)=s(\lph,m)$.
\end{lemma}

\begin{proof}
If $\lph,\lph'$ are incomparable then $m(\lph,\lph'):=(\lph_1-\lph_1')/(\lph_2'-\lph_2) \in (0,\infty)$, so 
\enumrom
\item $s(\lph,m)=s(\lph',m) \Lra m(\lph_1-\lph_1') = \lph_2'-\lph_2 \Lra m=m(\lph,\lph')$, and
\item if $\lph_1<\lph_1'$ then $s(\lph,m)<s(\lph',m) \Lra m>m(\lph,\lph')$.
\enumend
Using (ii), $M(A) \subset \{m(\lph,\lph') \colon \lph,\lph' \in A, \ \lph \ne \lph'\}$, so $M(A)$ is finite. Since $m\mapsto s(\lph,m)$ is continuous for each $\lph$ and $\env(A)$ is finite, using (i),
$$\{m\in (0,\infty)\colon A(m)=\{\lph\}\} = \bigcap_{\lph' \in \env(A)\setminus \lph}\{m\in (0,\infty)\colon s(\lph,m)<s(\lph',m)\}$$
is an open interval and similarly, using (i)-(ii), $\{m\in (0,\infty)\colon \lph \in A(m)\}$ is a relatively closed interval in $(0,\infty)$. Let $\lph$ be such that $A(m)=\lph$ for some $m$, and let $(\ul m, \ol m)=\{m\colon A(m)=\lph\}$. Then
\itemgo
\item either $\ul m=0$, or $\lph \in A(\ul m)$ and $\ul m \in M(A)$, and
\item either $\ol m=\infty$, or $\lph \in A(\ol m)$ and $\ol m \in M(A)$.
\itemend
In all cases $(\ul m, \ol m ) = (m_i,m_{i+1})$ for some $i \in \{0,\dots,n\}$. In particular, $A(m)$ is constant on each set $(m_i,m_{i+1})$, and letting $\lph(i)$ be its unique element, $\lph(i-1),\lph(i) \in A(m_i)$. If $\lph,\lph' \in A(m_i)$ and $\lph \ne \lph'$ then $m(\lph,\lph')=m_i$. If moreover $\lph_1<\lph_1'$ then from (ii) above,
\itemgo
\item $s(\lph,m)<s(\lph',m) \Lra m>m_i$ so $\lph' \ne \lph(i)$, and
\item $s(\lph',m)<s(\lph,m) \Lra m<m_i$ so $\lph \ne \lph(i-1)$.
\itemend
Since $\lph(i-1),\lph(i) \in A(m_i)$, it follows that
\itemgo
\item $\lph(i-1) = \arg\max\{\lph_1\colon (\lph_1,\lph_2)\in A(m_i)\}$ and
\item $\lph(i) = \arg\min\{\lph_1\colon (\lph_1,\lph_2) \in A(m_i)\}$.
\itemend
In particular, $\lph(i-1) \ne \lph(i)$, and $\lph_1(i-1)>\lph_1(i)$. Let $\lph = \arg\max\{\lph_1\colon (\lph_1,\lph_2\in \min(A)\}$. Then for $\lph' \in \min(A)\setminus \{\lph\}$ and small $m$, $s(\lph,m)<s(\lph',m)$, so $s(\lph,m)=s(A,m)$ for small $m$ which means that $\lph=\lph(0)$. Similarly one can show that $\lph(n)=\arg\min\{\lph_1\colon(\lph_1,\lph_2)\in\min(A)\}$. To prove the last statement -- that for each $m$ there is $\lph\in \piv(A)$ such that $s(\lph,m)=s(A,m)$ -- simply note that for each $m$, $\piv(A)\cap A(m) \ne \emptyset$.
\end{proof}

There is a nice visual interpretation of Lemma \ref{lem:envelope}. Define the set of lines $(L(A,m))_{m\in [0,\infty]}$ as follows:
\itemgo
\item $L(A,0) := \{(x,y) \in \R^2 \colon y = \min\{\lph_2\colon(\lph_1,\lph_2) \in A\}\}$,
\item $L(A,\infty) := \{(x,y) \in \R^2 \colon x = \min\{\lph_1\colon(\lph_1,\lph_2) \in A\}\}$, and
\item for $m\in (0,\infty)$, $L(A,m):= \{(x,y) \in \R^2\colon s((x,y),m) = s(A,m)\}$.
\itemend
Then each line $L(A,m)$ is a lower bound and a supporting line of $A$, in the sense that
\itemgo 
\item $L(A,m) \cap A$ is non-empty,
\item if $\lph \in A$ then $\lph \ge \lph'$ for some $\lph' \in L(A,m)$, and
\item if $\lph \in L(A,m)$ and $\lph'<\lph$ then $\lph' \notin A$.
\itemend
By definition, $A(m)= A \cap L(A,m)$ for $m\in (0,\infty)$.  As a function of $m$, $L(A,m)$ begins at $m=0$ as a horizontal line through the lowest points of $A$, finishes at $m=\infty$ as a vertical line through the leftmost points of $A$, and as $m$ increases from $0$ to $\infty$, rotates clockwise around the pivots of $A$, changing pivots when $m$ is equal to a pivot slope. The contour followed by the lines $(L(A,m))$ is the set $L(A)$ defined by
\begin{align}\label{eq:LofA}
L_+(A,m) &= \{(x,y) \in \R^2\colon s((x,y),m) \ge s(A,m)\} \ \text{and} \nonumber \\
L(A) &= \bigcap_{m\in (0,\infty)} L_+(A,m)\cap \bigcup_{m\in [0,\infty]}L(A,m).
\end{align}
Geometrically, $L(A)$ is made up of $n$ line segments, connecting $\lph(i-1)$ to $\lph(i)$ for $i=0,\dots,n-1$, together with the portion of $L(A,0)$ to the right of $\lph(0)$, and the portion of $L(A,\infty)$ above $\lph(n)$. It is the boundary, as well as a kind of set-wise greatest lower bound, of the set $S(A)$ defined by
\begin{align}\label{eq:SofA}
S(A) &= \bigcap_{m\in (0,\infty)}L_+(A,m).
\end{align} 
An example is shown in Figure \ref{fig:set-slopes-contour}.

\begin{figure}
\centering
\includegraphics[width=1.5in]{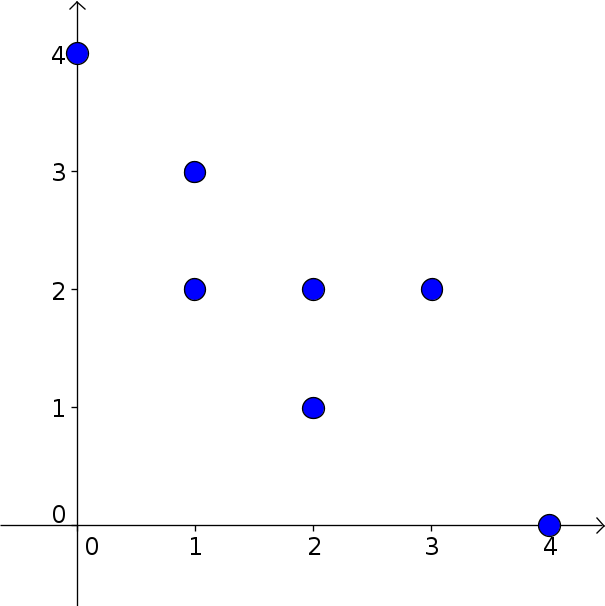}
\hspace{.4in}
\includegraphics[width=1.5in]{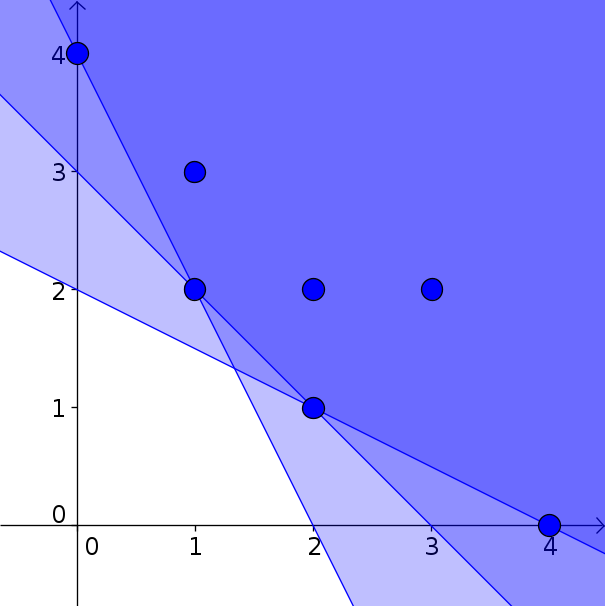}
\hspace{.4in}
\includegraphics[width=1.5in]{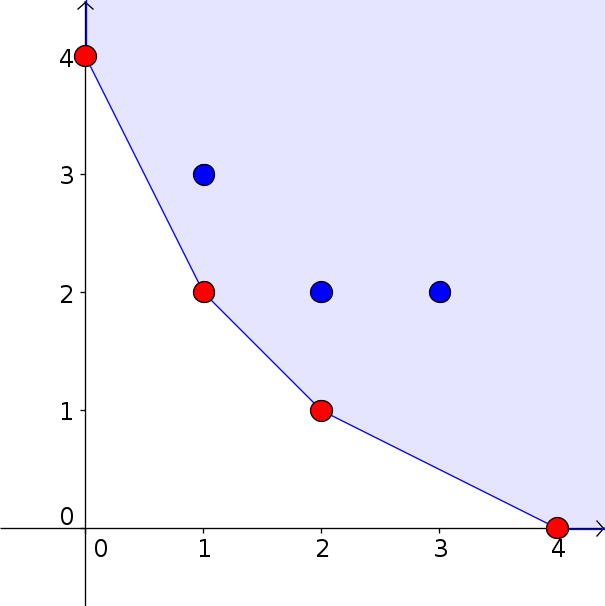}
\caption{The set $A=\{(4,0),(2,1),(1,2),(2,2),(3,3),(1,3),(0,4)\}$ in blue (left), with the lines $L(A,m_i)$ and shaded regions $L_+(A,m_i)$ evaluated at the pivot slopes $m_1=1/2, m_2=1$ and $m_3=2$ (center) and the contour $L(A)$ with the shaded region $S(A)$ and $\env(A)$ higlighted in red (right).}
\label{fig:set-slopes-contour}
\end{figure}

\subsection{Dominant terms and limit scales for $f(x,\lbd)$}\label{sec:domterms}
Next, we show that if $f$ has a Taylor expansion in the sense of \eqref{eq:loTaylor} then we can sum over $\env(A)$ instead of $A$, and we identify the limit scales for $f$ in the sense of Definition \ref{def:lim-scales}, and the form of the product expansion $v(x,\lbd)w(u)$. Note that if $f$ is analytic at $(0,0)$ then letting $(c_\lph)_{\lph \in \N^2}$ be the coefficients in its power series, it satisfies \eqref{eq:loTaylor} with $A=\min\{\lph\colon c_\lph \ne 0\}$. First we show only the terms in $\env(A)$ are important. 
\begin{lemma}\label{lem:envelope-sum}
Suppose $f$ has a Taylor expansion around $(0,0)$ as in \eqref{eq:loTaylor}, with powers $A$ and non-zero coefficients $(c_{\lph})_{\lph \in A}$. Then
\begin{align}\label{eq:envelope-sum}
f(x,\lambda) \sim \sum_{\alpha \in \env(A)}c_\alpha x^{\alpha_1}\lambda^{\alpha_2} \ \ \text{as} \ \ |x|+|\lambda| \to 0.
\end{align}
\end{lemma}

\begin{proof}
The result is clear if $A=\emptyset$. We will assume $x,\lbd \ge 0$, as the other cases follow in the same way with a change of sign. Since $\env(\min(A)) = \env(A)$ and $\min(A)$ is disordered, it suffices to show that
\itemgo
\item \eqref{eq:loTaylor} holds with $\min(A)$ in place of $A$, and
\item \eqref{eq:envelope-sum} holds assuming $A$ is disordered.
\itemend
The first step is easy: since $A$ is finite, for any $\lph \in A$ with $c_\lph \ne 0$,
$$\sum_{\{\lph' \in A \colon \lph' \ge \lph\}} c_{\lph'}x^{\lph_1'}\lbd^{\lph_2'} = (c_\lph+o(1))x^{\lph_1}\lbd^{\lph_2},$$
and since $A \subset \bigcup_{\lph \in \min(A)}\{\lph'\in \N^2 \colon \lph' \ge \lph\}$, it follows that $f(x,\lbd) \sim \sum_{\lph \in \min(A)}c_\lph x^{\lph_1} \lbd^{\lph_2}$. Now suppose $A$ is disordered. Since $A$ is finite, it suffices to show that for $\alpha=(\alpha_1,\lph_2) \in A\setminus \env(A)$,
\begin{align}\label{eq:env-sum}
x^{\lph_1}\lbd^{\lph_2} \ll \sum_{\alpha'\in \env(A)} x^{\lph_1'}\lbd^{\lph_2'}.
\end{align}
Let $\{\lph(i)\}_{i=0}^n \subset \env(A)$ denote the pivots from Lemma \ref{lem:envelope}. Since $A$ is disordered, $A=\min(A)$, so $\lph(0)=\arg\max\{\lph_1\colon \lph \in A\}$ and $\lph(n) =\arg\min\{\lph_1\colon \lph \in A\}$. Thus if $\lph \in A\setminus\env(A)$ then for some $i$, $\lph_1(i-1)<\lph_1 < \lph_1(i)$, and $s(\lph(i-1),m_i) = s(\lph(i),m_i)=s(A,m_i)<s(\lph,m_i)$.
\itemgo
\item If $(x,\lbd) \in [0,1]^2$ and $x = z \lbd^m$ with $z\le 1$ then $z^{\lph_1} \le z^{\lph_1(i-1)}$, so as $\lbd \to 0$,
$$x^{\lph_1}\lbd^{\lph_2} = z^{\lph_1}\lbd^{m \lph_1 + \lph_2} \ll z^{\lph_1(i-1)}\lbd^{m\lph_1(i-1) + \lph_2(i-1)} = x^{\lph_1(i-1)}\lbd^{\lph_2(i-1)}.$$

\item If $(x,\lbd) \in [0,1]^2$ and $x= z\lbd^m$ with $z \ge 1$, then $z^{\lph_1} \le z^{\lph_1(i)}$ so as $\lbd \to 0$,
$$x^{\lph_1}\lbd^{\lph_2} = z^{\lph_1}\lbd^{m\lph_1 + \lph_2} \ll z^{\lph_1(i)}\lbd^{m\lph_1(i)+\lph_2(i)} = x^{\lph_1(i)}\lbd^{\lph_2(i)}.$$
\itemend
It follows that $x^{\lph_1}\lbd^{\lph_2} \ll x^{\lph_1(i-1)}\lbd^{\lph_2(i-1)} + x^{\lph_1(i)}\lbd^{\lph_2(i)}$, which implies \eqref{eq:env-sum}.
\end{proof}

The next step is to refine Lemma \ref{lem:envelope-sum} by partitioning the vicinity of $(0,0)$ into regions where specific terms in the sum \eqref{eq:envelope-sum} are dominant. Conveniently, these regions are also the limit scales for $f$ in the sense of Definition \ref{def:lim-scales}. It suffices to work on $[0,1]^2$, as the results map to other quadrants by mapping that quadrant into $[0,1]^2$ by a change of sign. Just a note that I am picturing $[0,1]^2$ in $\R^2$ with $\lbd$ on the horizontal axis and $x$ on the vertical axis, as is tradition for bifurcation diagrams.\\

The usual notion of partition is a bit too strong for sets of sequences, so we replace it with the following partition-like notion that we call a subpartition.
\begin{definition}[subpartition]\label{def:subpart}
Let $S$ be a set of sequences, let $\mathcal{P}$ be a collection of subsets of $S$ and define the domain of $\cP$ by $\dom(\cP) := \bigcup_{A \in \cP}A$. Then $\cP$ is a subpartition of $S$ if
\itemgo
\item each $A\in \cP$ is closed under the taking of subsequences,
\item $A,B \in \cP$ and $A \ne B$ implies $A\cap B=\emptyset$, and
\item every element of $S$ has a subsequence belonging to $\dom(\cP)$.
\itemend
\end{definition}

A compelling example is the following: let $K$ be a compact set and $S$ the set of sequences with values in $K$, then let $\cP = \{A_x\}_{x \in K}$ where $A_x$ is the set of sequences in $S$ with limit $x$. Another example, which is the one we'll use for limit scales, is given by the following definition.

\begin{definition}[subpartition induced by a finite set]\label{def:subpartM}
Let $M=\{m_i\}_{i=1}^n \subset (0,\infty)$ be a non-empty finite set with $0<m_1<\dots <m_n<\infty$. Abusing notation somewhat, let $(x,\lbd)$ denote a sequence in $[0,1]^2$ that converges to $(0,0)$ and define $(R_i)_{i=0}^{2n}$ by
\begin{align}\label{eq:asym-regions}
& R_0 = \{(x,\lbd) \colon x \gg \lbd^{m_1}\}, \nonumber \\
& R_{2i-1} = \{x \asymp \lbd^{m_i}\}, \ i\in\{1,\dots,n\}, \nonumber \\
& R_{2i} = \{(x,\lbd) \colon \lbd^{m_i} \gg x \gg \lbd^{m_{i+1}}\}, \ i \in \{1,\dots,n-1\}\nonumber \\
& R_{2n} = \{(x,\lbd)\colon \lbd^{m_n} \gg x\}. 
\end{align}
The collection $\{R_i\}_{i=0}^{2n}$ is called the subpartition induced by $M$ and denoted $\sub(M)$.
\end{definition}

The following is not hard to show, so we'll omit the proof.

\begin{lemma}\label{lem:subpartM}
Let $M$ and $\sub(M)$ be as in Definition \ref{def:subpartM}.
\itemgo
\item $\sub(M)$ is a subpartition of sequences in $[0,1]^2$ that converge to $(0,0)$.
\item If $M\supset M'$ and $(x,\lbd)\in \dom(\sub(M))$, then $(x,\lbd) \in \dom(\sub(M'))$.
\itemend
\end{lemma}

In the next result we show that if $f$ has a Taylor expansion around $(0,0)$ with powers $A$ and $M\supset M(A)$ then on each element of the subpartition induced by $M$, $f$ is asymptotic to the sum of the dominant terms on that element. The reason we formulate it in terms of $M\supset M(A)$ and not just $M(A)$ is because, once we deal with both $F$ and $G$, we need to join the corresponding subpartitions, which refines each one.

\begin{lemma}\label{lem:dom-terms}
Suppose $f$ has a Taylor expansion around $(0,0)$ with powers $A$ as in \eqref{eq:loTaylor}, and that $A \ne \emptyset$, and let $(A(m))_{m\in(0,\infty)}$ and $M(A)$ be as in Lemma \ref{lem:envelope}. Suppose $M = \{m_i\}_{i=1}^n\subset \N^2$ is a finite set with $0<m_1<\dots<m_n<\infty$, and that $M(A) \subset M$. Let $m_0=0$ and $m_{n=1}=\infty$ and for $i\in \{0,\dots,n\}$ let $\lph(i)$ be such that $A(m)=\{\lph(i)\}$ for $m\in (m_i,m_{i+1})$. Let $\{R_i\}_{i=0}^{2n}$ denote the subpartition induced by $M$, as in Definition \ref{def:subpartM}.
\enumrom
\item If $(x,\lbd) \in R_{2i}$ for $i\in \{0,\dots,n\}$ then $f(x,\lbd) \sim c_{\alpha(i)} x^{\lph_1(i)}\lbd^{\lph_2(i)}$.
\item If $(x,\lbd) \in R_{2i-1}$ for $i \in \{1,\dots,n\}$ then
$$f(x,\lbd) \sim \sum_{\lph \in A(m_i)} c_\lph x^{\lph_1}\lbd^{\lph_2}.$$
\enumend
Letting $A_{2i} = A(m_i)$ and $A_{2i-1} = \{\lph(i)\}$, for $i\in \{0,\dots,2n\}$ and $(x,\lbd) \in R_i$, 
\enumrom
\item[(iii)] $f(x,\lbd) \sim \sum_{\lph \in A_i}c_\lph x^{\lph_1}\lbd^{\lph_2}$, and
\item[(iv)] for any $\lph \in A_i$, $f(x,\lbd) = O(x^{\lph_1}\lbd^{\lph_2})$.
\enumend
\end{lemma}
For the sake of disambiguation, it should be noted that the definition of $\{\lph(i)\}_{i=0}^n$ given here agrees with the one in Lemma \ref{lem:envelope} iff $M=M(A)$. 
\begin{proof}
We first prove (iii)-(iv). (iii) is simply a summary of (i)-(ii). (iv) is trivial if $i$ is even, since $A_i$ is a singleton. For odd $i$, if $(x,\lbd)\in R_i$ then $x\sim \lbd^{m_i}z$ for some constant $z\in(0,\infty)$, and if $\lph \in A_i$ then $s(\lph,m_i)=s(A,m_i)$, so
$$x^{\lph_1}\lbd^{\lph_2} = z^{\lph_1 m_i}\lbd^{\lph_1 m_i + \lph_2} = z^{\lph_1 m_i} \lbd^{s(\lph,m_i)} \asymp \lbd^{s(A,m_i)}.$$
Since $A_i$ is finite, $f(x,\lbd) = O(\lbd^{s(A,m_i)})$, and (iv) follows.\\

We now prove (i)-(ii). Referring to Lemma \ref{lem:envelope}, since $M(A) \subset M$ and $M$ is arranged in increasing order,
\itemgo
\item $\{\lph(i)\}_{i=0}^n = \piv(A)$,
\item $\lph(0)=\arg\max\{\lph_1\colon \lph \in A\}$, 
\item $\lph(n)=\arg\min\{\lph_1\colon \lph \in A\}$,
\item $\lph_1(i-1)<\lph_1(i)$ and $\lph(i-1),\lph(i) \in A(m_i)$ if $m_i \in M(A)$, and
\item $\lph_1(i-1)=\lph(i)$ and $A(m_i)=\{\lph(i-1)\}=\{\lph(i)\}$ if $m_i \in M\setminus M(A)$.
\itemend

Suppose $(x,\lbd) \in R_{2i}$ and $\lph \in \env(A)\setminus \{\lph(i)\}$. If $\lph_1 \le \lph_1(i)$, then since $\env(A)$ is disordered, $\lph_1<\lph_1(i)$ and since $\lph(n)=\arg\min\{\lph_1\colon \lph\in A\}$, $i<n$ and $m_{i+1}<\infty$. Since $\lph(i)\in A(m_{i+1})$, $s(\lph(i),m_{i+1}) \le s(\lph,m_{i+1})$, so if $z$ satisfies $x=z\lbd^{m_{i+1}}$ then since $(x,\lbd) \in R_{2i}$, $\lbd \ll 1$ and $z\gg 1$, so
$$x^{\lph_1}\lbd^{\lph_2} = z^{\lph_1}\lbd^{s(\lph,m_{i+1})} \ll z^{\lph_1(i)}\lbd^{s(\lph(i),m_{i+1})} = x^{\lph_1(i)}\lbd^{\lph_2(i)}.$$
If instead $\lph_1 > \lph(i)$ then since $\lph(0) = \arg\max\{\lph_1\colon \lph\in A\}$, $i>0$ so $m_i>0$. Since $\lph(i) \in A(m_i)$, $s(\lph(i),m_i) \le s(\lph,m_i)$, so if $z$ satisfies $x=z\lbd^{m_i}$, then since $(x,\lbd) \in R_{2i}$, $\lbd \ll 1$ and $z \ll 1$, so
$$x^{\lph_1}\lbd^{\lph_2} =  z^{\lph_1}\lbd^{s(\lph,m_i)} \ll z^{\lph_1(i)} \lbd^{s(\lph(i),m_i)} = x^{\lph_1(i)}\lbd^{\lph_2(i)},$$
which implies (i). If $(x,\lbd) \in R_{2i-1}$ and $x=z\lbd^{m_i}$ then $z \asymp 1$, so for any $\lph,\lph' \in A$,
$$x^{\lph_1'}\lbd^{\lph_2'} = z^{\lph_1'} \lbd^{s(\lph',m_i)} \asymp \lbd^{s(\lph',m_i)} \begin{cases} = \lbd^{s(\lph,m_i)} & \text{if} \ \ s(\lph',m_i)=s(\lph,m_i) \\
\ll \lbd^{s(\lph,m_i)} & \text{if} \ \ s(\lph',m_i) > s(\lph,m_i),\end{cases}$$
which implies (ii).
\end{proof}

Using Lemma \ref{lem:dom-terms} and a bit of extra work, we show the limit scales for $f$ satisfying \eqref{eq:loTaylor} are exactly the sequences in $\dom(\sub(M(A)))$, and identify the correpsonding scale functions on each element of $\sub(M(A))$.

\begin{theorem}\label{thm:limscale}
Suppose $f$ has a Taylor expansion around $(0,0)$ with non-empty set of powers $A$ as in \eqref{eq:loTaylor}, and let $M(A)$ be as in Lemma \ref{lem:envelope}, $\sub(\dots)$ as in Definition \ref{def:subpartM} and $\dom(\dots)$ as in Definition \ref{def:subpart}. Then $(x,\lbd)$ is a limit scale for $f$ if and only if $(x,\lbd) \in \dom(\sub(M(A)))$.\\

In particular, if $M\supset M(A)$ and $(x,\lbd) \in \dom(\sub(M))$ then $(x,\lbd)$ is a limit scale for $f$, and defining $(R_i)_{i=0}^{2n}$ using $M$ as in Definition \ref{def:subpartM}, in the notation of Lemma \ref{lem:dom-terms}, the corresponding scale functions are given by
\itemgo
\item $v_{2i}(x,\lbd) := c_{\lph(i)}x^{\lph_1(i)}\lbd^{\lph_2(i)}$ and $w_{2i}(u) := u^{\lph_1(i)}$ on $(R_{2i})_{i=0}^n$, and by
\item $v_{2i-1}(x,\lbd) := \lbd^{s(A,m_i)}$ and $w_{2i-1}(u) := \sum_{\lph \in A(m_i)} c_\lph z^{\lph_1} u^{\lph_1}$ on $(R_{2i-1})_{i=1}^n$,
\itemend
where for each $(x,\lbd) \in R_{2i-1}$, $z$ is the unique value such that $x \sim z \lbd^{m_i}$.
\end{theorem}

\begin{proof}
We first show that if $M\supset M(A)$ and $(x,\lbd) \in \dom(\sub(M))$ then $(x,\lbd)$ is a limit scale for $f$, and identify the given scale functions. We then show that if $(x,\lbd) \notin \dom(\sub(M(A)))$ then $(x,\lbd)$ is not a limit scale for $f$. If $M\subset (0,\infty)$ is any finite set, then elements of $\sub(M)$ are closed under positive scaling of either variable. That is, defining $(m_i)$, $(R_i)$ as in Definition \ref{def:subpartM}, if $(x,\lbd) \in R_i$ and $u\in (0,\infty)$ then $(ux,\lbd),(x,u\lbd)\in R_i$, for $i\in \{0,\dots,2n\}$. Suppose $M\supset M(A)$ and $(x,\lbd) \in \dom(\sub(M))$, so that $(x,\lbd) \in R_i$ for some $i$. If $(x,\lbd) \in R_{2i}$ then in the notation of Lemma \ref{lem:dom-terms},
$$f(ux,\lbd) \sim c_{\lph(i)}(ux)^{\lph_1(i)}\lbd^{\lph_2(i)}$$
which has the desired form with $v_{2i}(x,\lbd)=c_{\lph(i)}x^{\lph_1(i)}\lbd^{\lph_2(i)}$ and $w_{2i}(u) = u^{\lph_1(i)}$. If $(x,\lbd) \in R_{2i-1}$, letting $z$ be the constant such that $x\sim z\lbd^{m_i}$, then using Lemma \ref{lem:dom-terms} and recalling that $A(m_i):= \{\lph \in A\colon m_i\lph_1+\lph_2=s(A,m_i)\}$,
$$f(ux,\lbd) \sim \sum_{\lph \in A(m_i)}c_\lph (ux^{\lph_1})\lbd^{\lph_2} = \lbd^{s(A,m_i)}\sum_{\lph \in A(m_i)}c_\lph z^{\lph_1}u^{\lph_1},$$
which has the desired form with $v_{2i-1}(x,\lbd) = \lbd^{s(A,m_i)}$ and $w_{2i-1}(u) = \sum_{\lph \in A(m_i)} c_\lph z^{\lph_1}u^{\lph_1}$. In both cases, it is clear, by examining the proof of Lemma \ref{lem:dom-terms}, that for each $R_i$ and any fixed sequence $(x,\lbd)\in R_i$, convergence of $f(ux,\lbd)/(v_i(x,\lbd)w_i(u))$ to $1$ as $|x|+|\lbd|\to 0$ is uniform over $u$ in compact subsets of $(0,\infty)$. It remains to show that if $(x,\lbd) \notin \dom(\sub(M(A)))$ then $(x,\lbd)$ is not a limit scale for $f$. If $(x,\lbd) \notin \dom(\sub(M(A)))$ then it has subsequences $(x',\lbd')$ and $(x''',\lbd''')$ belonging to distinct elements of $\sub(M(A))$. This is verified as follows.
\itemgo
\item By Lemma \ref{lem:subpartM}, $(x,\lbd)$ has a subsequence $(x',\lbd') \in \A$ for some $\A \in \sub(M(A))$, and
\item Since $(x,\lbd) \notin \A$, by definition of $(R_i)$, $(x,\lbd)$ has a subsequence $(x'',\lbd'')$ that has no subsequence in $\A$; for example, if $(x_n,\lbd_n) \notin R_{2i-1}$, take $(n_k)$ so that $x_{n_k}/\lbd^{m_i}_{n_k}$ tends to $0$ or $\infty$. Using Lemma \ref{lem:subpartM} again, $(x'',\lbd'')$ has a subsequence $(x''',\lbd''')$ in $\A'$ for some $\A' \in \sub(M(A))$ with $\A' \ne \A$.
\itemend
Let $i \ne j$ be such that $\A=R_i$ and $\A'=R_j$ and let $v_i,w_i$ and $v_j,w_j$ be the limit scale functions given above. Then $w_i,w_j$ are polynomials and with $(A_i)$ as in Lemma \ref{lem:dom-terms}, since $A_i \ne A_j$ for $i\ne j$, $w_i$ is not a multiple of $w_j$. Suppose $(x,\lbd)$ is a limit scale and let $k,w$ be the corresponding functions. Then
$$f(ux',\lbd') \sim v_i(x',\lbd')w_i(u) \sim v(x',\lbd')w(u)$$
and
$$f(ux''',\lbd''') \sim v_j(x''',\lbd''')w_j(u) \sim v(x''',\lbd''')w(u).$$
If $u$ is such that $w_i(u),w_j(u),w(u) \ne 0$ then
$$w_i(u)/w(u) \sim v(x',\lbd')/v_i(x',\lbd') \ \text{and} \ w_j(u)/w(u) \sim v(x''',\lbd''')/v_j(x''',\lbd''').$$
In particular, on $\{u\in(0,\infty)\colon w_i(u)w_j(u)w(u) \ne 0\}$, $w_i(u)/w(u)$ and $w_j(u)/w(u)$ are constant, so $w_i(u)/w_j(u)$ is constant, i.e., $w_i$ is a multiple $w_j$, which it is not.
\end{proof}

\subsection{Limit scales around a constant branch}\label{sec:const-equil}

We now return to the discussion of shape that was interrupted above Section \ref{sec:env}. Recall that $F(x,\lbd),G(x,\lbd)$ are the characteristics of a parametrized QDP $(x_\ep(t;\lbd_\ep))_{t\ge 0}$ in the sense of Definition \eqref{def:prmtrzdQDP}, and we want $(1/a_\ep,\lbd_\ep)$ to be a limit scale for both $F$ and $G$. Assume that $F,G$ have a Taylor expansion around $(0,0)$ with respective powers $A,B$ and non-zero coefficients $(c_\lph)_{\lph \in A}$, $(d_\beta)_{\beta \in B}$, in the sense of \eqref{eq:loTaylor}. Then by Theorem \ref{thm:limscale}, $(1/a_\ep,\lbd_\ep)$ is a limit scale for both $F$ and $G$ iff it is in the domain of the subpartitions defined by both $A$ and $B$. This can be expressed concisely as follows. Let $A(m),B(m)$ and $M(A),M(B)$ be as in Lemma \ref{lem:envelope}, and let $M=M(A)\cup M(B)$. It is easy to check that
$$\sub(M) = \sub(M(A)) \vee \sub(M(B)) = \{S \cap S'\colon S\in \sub(M(A)), \ S' \in \sub(M(B))\}.$$
It follows that $(1/a_\ep,\lbd_\ep)$ is a limit scale for both $F$ and $G$ iff it is in $\dom(\sub(M))$. Label $\sub(M)$ by $(R_i)_{i=0}^{2n}$ as in Definition \ref{def:subpartM}, and let $(\lph(i))_{i=0}^n$ and $(\beta(i))_{i=0}^n$ be as in Lemma \ref{lem:dom-terms}, corresponding to $A,B$ respectively. Referring back to \eqref{eq:tF-tG-ratio2}-\eqref{eq:h-ell} and using the scale functions of Theorem \ref{thm:limscale}, if $(1/a_\ep,\lbd_\ep) \in R_i$ then $Q,V,h_\ep,\ell_\ep$ are as given in Table \ref{tab:QVhl}. Note that in all cases,
\begin{align}\label{eq:h-ell-Gma}
h_\ep &\asymp a_\ep^{1-\lph_1(i)}b_\ep \lbd_\ep^{\lph_2(i)} \quad \text{and} \nonumber \\
\ell_\ep &\asymp \ep^2 a_\ep^{2-\beta_1(i)}b_\ep \lbd_\ep^{\beta_2(i)}.
\end{align}
On $R_{2i}$, this is immediate; on $R_{2i-1}$, use $1/a_\ep\sim z\lbd^{m_i}$ and the fact that $s(\lph,m_i) = s(A,m_i)$ for all $\lph \in A(m_i)$ and $\lph(i) \in A(m_i)$, similarly for $B(m_i)$, to obtain \eqref{eq:h-ell-Gma}.\\
\begin{table}
\bgroup
\def\arraystretch{2}
\begin{tabular}{c | c | c | c | c}
$(1/a_\ep,\lbd_\ep) \in \bullet$ & $Q$ & $V$ & $h_\ep$ & $\ell_\ep$ \\ \hline  \hline 
$R_{2i}$ & $c_{\lph(i)}x^{\lph_1(i)}$ & $c_{\beta(i)}x^{\beta_1(i)}$ & $a_\ep^{1-\lph_1(i)} b_\ep \lbd_\ep^{\lph_2(i)}$ & $\ep^2 a_\ep^{2-\beta_1(i)} b_\ep \lbd_\ep^{\beta_2(i)}$ \\ \hline
$R_{2i-1}$ & $\sum_{\lph \in A(m_i)}c_\lph (zx)^{\lph_1}$ & $\sum_{\beta \in B(m_i)} c_\beta (zx)^{\beta_1}$ & $a_\ep b_\ep \lbd_\ep^{s(A,m_i)}$ & $\ep^2 a_\ep b_\ep \lbd_	\ep^{s(B,m_i)}$
\end{tabular} 
\egroup
\caption{Shape and scale of characteristics, in the notation of Lemma \ref{lem:dom-terms}.
\\ On $R_{2i-1}$, $z$ is the constant such that $1/a_\ep \sim z \lbd_\ep^{m_i}$.}
\label{tab:QVhl}
\end{table}

\nid\textit{Drift to diffusion ratio.} 
Let $\dlt(i) = \lph(i)-\beta(i)$ and define $r$ on $[0,1]^2$ piecewise by
\begin{align}\label{eq:dd-fcn}
r(x,\lbd) = x^{1+\dlt_1(i)}\lbd^{\dlt_2(i)} \ \ \text{for} \ \ \lbd^{m_{i+1}} \le x \le \lbd^{m_i},
\end{align}
with $m_0:=0$ and $m_{n=1}:=\infty$ (so $\lbd^{m_0}=1$ and $\lbd^{m_{n+1}}=0$). Since $\lph(i)$ and $\lph(i+1)$ both minimize $s(\lph, m_{i+1})$ over $\lph \in A$, and similarly for $\beta$, $r$ is well-defined on the curves $x=\lbd^{m_{i+1}}$ and continuous on $[0,1]^2$. The function $r$ gives the approximate ratio of drift to diffusion terms: if $(1/a_\ep,\lbd_\ep) \in R_i$ for some $i$ then it follows from \eqref{eq:h-ell-Gma} that
$$\ep^{-2} r(1/a_\ep,\lbd_\ep) \asymp h_\ep/\ell_\ep.$$
As in Section \ref{sec:iso}, there are three pertinent cases, that we call limit scales for the drift to diffusion ratio.

\itemgo
\item If $r(1/a_\ep,\lbd_\ep) \ll \ep^2$ then $h_\ep \ll \ell_\ep$(diffusion dominates).
\item If $r(1/a_\ep,\lbd_\ep) \gg \ep^2$ then $\ell_\ep \ll h_\ep$ (drift dominates).
\item If $r(1/a_\ep,\lbd_\ep) \asymp \ep^2$ then $\ell_\ep \asymp h_\ep$ (drift matches diffusion).
\itemend

\nid\textit{Time scale.} As in Section \ref{sec:iso}, $(b_\ep)$ should be chosen just large enough that at least one of $\tF,\tG$ is non-zero. Referring back to \eqref{eq:tF-tG-ratio2}, $h_\ep \asymp 1 \Lra \tF\propto Q$ and $\ell_\ep \asymp 1 \Lra \tG\propto R$. Using \eqref{eq:h-ell-Gma},
\itemgo
\item $h_\ep \asymp 1 \Lra b_\ep \asymp a_\ep^{\lph_1(i)-1}\lbd_\ep^{-\lph_2(i)}$ and
\item $\ell_\ep \asymp 1 \Lra b_\ep \asymp \ep^{-2}a_\ep^{\beta_1(i)-2}\lbd_\ep^{-\beta_2(i)}$.
\itemend

As in Section \ref{sec:iso}, in each case we say that $(b_\ep)$ is the visible time scale  corresponding to $(1/a_\ep,\lbd_\ep)$. We now combine these observations to state a limit theorem for QDPs. As in Section \ref{sec:iso}, we'll find that when $F=O(G)$, diffusion dominates when $(1/a_\ep,\lbd_\ep)$ is small, and drift dominates when $(1/a_\ep,\lbd_\ep)$ is large, but since the details are more complex, we'll first state the result just in terms of $r$ itself, then afterward study the geometry of the separatrix $\{(x,\lbd) \colon r(x,\lbd)=\ep^2\}$.

\begin{theorem}[Limit scales for parametrized QDPs]\label{thm:prmtrzd-limits}
Let $\tld U$ be a non-empty open convex set whose closure contains $0$, and let $\lbd$ be a parameter taking values in an interval $I$ that contains $0$. Suppose $F=F(x,\lbd)$ and $G=G(x,\lbd)$ have a Taylor expansion around $(0,0)$ with respective powers $A,B$, as in \eqref{eq:loTaylor}. Let $(x_\ep(t;\lbd)\colon  \ep \in \E,\,\lbd \in I,\,t\ge 0)$ be a collection of semimartingales, and given $(a_\ep),(b_\ep),(\lbd_\ep)$, suppose that for each domain $D\subset\subset \tld U$, $(x_\ep(t ;\lbd_\ep))_{t\ge 0}$ is a parametrized QDP on $(D_\ep):=(\{x/a_\ep\colon x \in D\})$ to scale $(a_\ep),(b_\ep)$ with characteristics $F,G$. Let $M(A),M(B)$ be as in Lemma \ref{lem:envelope} and let $M=M(A)\cup M(B)$, and let $\sub(M)$ be the subpartition induced by $M$, as in Definition \ref{def:subpartM}. Label $\sub(M)$ by $(R_i)$ as in Lemma \ref{lem:dom-terms}.\\

Let $Y_\ep(t) = a_\ep x_\ep(b_\ep t;\lbd_\ep)$ and suppose that $a_\epsilon \to \infty$. Suppose that $(1/a_\ep,\lbd_\ep)$ is a limit scale for both $F$ and $G$, and for the drift to diffusion ratio, and that $(b_\ep)$ is the corresponding visible time scale. Then $(1/a_\ep,\lbd_\ep) \in R_i$ for some $i$, $(a_\ep),(b_\ep)$ and $(\lbd_\ep)$ satisfy one of the sets of conditions below, and $(Y_\ep)$ is a QD with characteristics $\tF,\tG$ as described below, with $Q,V$ as given in Table \ref{tab:QVhl}.
\itemgo
\item If $r(1/a_\ep,\lbd_\ep) \ll \ep^2$ and $b_\ep \asymp \ep^{-2}a_\ep^{\beta_1(i)-2}\lbd_\ep^{-\beta_2(i)}$ then $\tld F=0$ and $\tld G \propto V$.
\item If $r(1/a_\ep,\lbd_\ep) \gg \ep^2$ and $b_\ep \asymp a_\ep^{\lph_1(i)-1}\lbd_\ep^{-\lph_2(i)}$ then $\tld F \propto Q$ and $\tld G=0$.
\item If $r(1/a_\ep,\lbd_\ep) \asymp \ep^2$ and $b_\ep \asymp a_\ep^{\lph_1(i)-1}\lbd_\ep^{-\lph_2(i)} \asymp \ep^{-2}a_\ep^{\beta_1(i)-2}\lbd_\ep^{-\beta_2(i)}$, then $\tld F \propto Q$ and $\tld G \propto V$.
\itemend
\end{theorem}

\begin{proof}
The result follows in each case from locally uniform convergence in \eqref{eq:FG-par-lim}, which in turn follows easily from the discussion above.
\end{proof}

We now study the regions defined by the three conditions on $r$. Before doing so, we take a moment to relate the condition $F=O(G)$ to the Taylor series of $F$ and $G$. Condition (iii) below is included for its visual appeal.

\begin{lemma}\label{lem:FOGenv}
Suppose $F,G$ has a Taylor expansion to leading order around $(0,0)$ as in \eqref{eq:loTaylor} with powers $A ,B\subset \N^2$ respectively, and let $\piv(A),\piv(B)$ be as in Lemma \ref{lem:envelope-sum} and $S(A),S(B)$ as in \eqref{eq:SofA}. Among the following statements, (i) $\Rightarrow$ (ii), (ii) $\Leftrightarrow$ (iii), and if $\piv(B)=\env(B)$ and $G\ge 0$ then (ii) $\Rightarrow$ (i).
\enumrom
\item $F(x,\lbd)=O(G(x,\lbd))$ as $|x|+|\lbd| \to 0$.
\item $s(B,m)\le s(A,m)$ for all $m\in(0,\infty)$.
\item $\piv(A)\subset S(B)$.
\enumend
Using $M=M(A)\cup M(B)$ in Lemma \ref{lem:dom-terms}, let $(\lph(i))$, $(\beta(i))$ correspond to $A,B$, respectively. Then any of (i)-(iii) above implies that
\itemgo
\item $\lph_2(0)\ge\beta_2(0)$ and if $\lph_2(0)=\beta_2(0)$ then $\lph_1(0) \ge \beta_1(0)$, and
\item $\lph_1(n)\ge \beta_1(n)$ and if $\lph_1(n) =\beta_1(n)$ then $\lph_2(n)\ge \beta_2(n)$.
\itemend
\end{lemma}
Before giving the proof, let us note a counterexample that illustrates why the implication (ii) $\Rightarrow$ (i) needs to be qualified. Let $F(x,\lbd) = \lbd^2+x^2$ and $G(x,\lbd)=(\lbd-x)^2$, which even has $G\ge 0$. Then $A=\env(A)=\{(0,2),(2,0)\}$ and $B=\env(B)=\{(0,2),(1,1),(2,0)\}$, so $M(A)=M(B)=\{1\}$ and $\piv(A)=\piv(B)=\{(2,0),(0,2)\}$. Since $s(A,m)=s(\piv(A),m)$, and similarly for $B$, $s(A,m)=s(B,m)$ for each $m\in(0,\infty)$. On the other hand, along the line $x=\lbd$ we have $F(\lbd,\lbd) = 2\lbd^2$ while $G(\lbd,\lbd) = 0$. Note that, in this example, $\env(B)\setminus \piv(B)=\{(1,1)\}\ne\emptyset$.

\begin{proof}[Proof of Lemma \ref{lem:FOGenv}]
We proceed in the following order:
\enumar
\item (ii) implies the last two statements.
\item (ii) $\Leftrightarrow$ (iii).
\item (i) $\Rightarrow$ (ii).
\item If $\piv(B)=\env(B)$ and $G\ge 0$ then (ii) $\Rightarrow$ (i).
\enumend

\nid\textit{Last two statements.} If $m\in (0,m_1)$ then using (ii), $s(\lph(0),m)=s(A,m)\ge s(B,m) = s(\beta(0),m)$. Letting $m\to 0$, $\lph_2(0)\ge \beta_2(0)$, and if $\lph_2(0)=\beta_2(0)$ then $\lph_1(0)\ge \beta_1(0)$. Using a similar argument for $m\in (m_n,\infty)$ gives the last statement.\\

\nid\textit{(ii) $\Rightarrow$ (iii).} Since $\piv(A)\subset A$ it's enough to show that $A\subset S(B)$. By definition, $S(B) = \{(x,y) \in \R^2\colon s((x,y),m) \ge s(B,m) \ \text{for all} \ m \in (0,\infty)\}$. If $\lph \in A$ then by definition $s(\lph,m) \le s(A,m)$ for all $m\in (0,\infty)$. Using (ii), $s(\lph,m) \le s(B,m)$ for all $m\in (0,\infty)$, so $\lph \in S(B)$.\\

\nid\textit{(iii) $\Rightarrow$ (ii).} For $m \in (0,\infty)$, from Lemma \ref{lem:envelope} there is $\lph \in \piv(A)$ with $s(\lph,m)=s(A,m)$. If (iii) holds then $\lph \in S(B)$ so in particular, $s(\lph,m) \ge s(B,m)$. Combining gives (ii).\\ 

\nid\textit{(i) $\Rightarrow$ (ii).} 
Denote the coefficients of $F,G$ as in \eqref{eq:loTaylor} by $(c_\lph)$ and $(d_\beta)$ respectively. If $x = z\lbd^{m_i}$ with constant $z>0$ and $\lbd \to0$ then $(x,\lbd) \in R_{2i-1}$ and
$$F(x,\lbd)=\sum_{\lph \in A(m_i)}(c_\lph + o(1))x^{\lph_1}\lbd^{\lph_2} = \lbd^{s(A,m_i)}\sum_{\lph \in A(m_i)} (c_\lph+o(1))z^{\lph_1},$$ 
where $c_\lph \ne 0$ for every $\lph \in A(m_i)$. Since $\lph(i)$ minimizes $\lph_1$ over $(\lph_1,\lph_2)\in A(m_i)$, if $z>0$ is small then
$$\left |\sum_{\lph \in A(m_i)}(c_\lph+o(1))z^{\lph_1}\right| \ge \frac{1}{2}\, c_{\lph(i)}z^{\lph_1(i)}>0$$
and $F(z\lbd^{m_i},\lbd) \asymp \lbd^{s(A,m_i)}$. Similarly $G(z\lbd^{m_i},\lbd)\asymp \lbd^{s(B,m_i)}$ for small $z>0$. If $F=O(G)$ it follows that $s(A,m_i) \ge s(B,m_i)$.
If $m\in (m_i,m_{i+1})$ then $(\lbd^m,\lbd)\in R_{2i}$ and $F(\lbd^m,\lbd)\sim c_{\lph(i)}\lbd^{s(\lph(i),m)}$, similarly, $G(x,\lbd)\sim d_{\lph(i)}\lbd^{s(\beta(i),m)}$. Since $s(\lph(i),m)=s(A,m)$ and $s(\beta(i),m) = s(B,m)$, if $F=O(G)$ then $s(A,m)\ge s(B,m)$.\\ 

\nid\textit{If $\piv(B)=\env(B)$ and $G\ge 0$ then (ii) $\Rightarrow$ (i).} We show by contradiction that if (i) is false then (ii) cannot be true. With $M(A),M(B)$ as in Lemma \ref{lem:envelope}, let $M=M(A)\cup M(B)$ and label $\sub(M)$ by $(R_i)_{i=0}^{2n}$ as in Lemma \ref{def:subpartM}. If (i) is false, there is a sequence $(x,\lbd)$ with limit $(0,0)$ such that $|G(x,\lbd)| / |F(x,\lbd)| \to 0$. By Lemma \ref{lem:subpartM} we can assume $(x,\lbd) \in \dom(\sub(M))=\bigcup_{i=0}^{2n} R_i$. There are two cases -- note that the assumptions $\piv(B)=\env(B)$ and $G\ge 0$ will only be needed in Case 2.\\

\nid\textit{Case 1: $(x,\lbd) \in R_{2i}$ for some $i$.} By Lemma \ref{lem:dom-terms}, $F(x,\lbd) \sim x^{\lph_1(i)}\lbd^{\lph_2(i)}$ and $G(x,\lbd)\sim x^{\beta_1(i)}\lbd^{\beta_2(i)}$, so
\begin{align}\label{eq:GoverF}
|G(x,\lbd)|/|F(x,\lbd)| \sim x^{\beta_1(i)-\lph_1(i)}\lbd^{\beta_2(i)-\lph_2(i)} \to 0,
\end{align}
moreover $s(\lph(i),m) = s(A,m)$ and $s(\beta(i),m) = s(B,m)$ for $m\in (m_i,m_{i+1})$, so if (ii) holds then $s(\lph(i),m) \ge s(\beta(i),m)$ for $m\in (m_i,m_{i+1})$. If $i<n$ then $m_{i+1}<\infty$, so by definition of $(x,\lbd) \in R_{2i}$, $x \gg \lbd^{m_{i+1}}$, and by continuity of $m\mapsto s((x,y),m)$, $s(\lph(i),m_{i+1}) \ge s(\beta(i),m_{i+1})$. Plugging $x\gg \lbd^{m_{i+1}}$ into \eqref{eq:GoverF},
$$\lbd^{m_{i+1}(\beta_1(i)-\lph_1(i))} \ll \lbd^{\lph_2(i)-\beta_2(i)}$$
and rearranging gives $$\lbd^{s(\beta(i),m_{i+1})} \ll \lbd^{s(\lph(i),m_{i+1})}$$
which, since $\lbd \to 0$, contradicts $s(\lph(i),m_{i+1})\ge s(\beta(i),m_{i+1})$. The same argument is applicable if $i=n$, as long as $x \gg \lbd^m$ for some $m\in (m_n,\infty)$. If instead $i=n$ and $x/\lbd^m \to 0$ for all $m<\infty$ then supposing again that (ii) holds, we showed above that $\lph_1(n) \ge \beta_1(n)$ and that if $\lph_1(n)=\beta_1(n)$ then $\lph_2(n) \ge \beta_2(n)$. The latter is ruled out, since plugging into \eqref{eq:GoverF} gives $\lbd^a \to 0$ with $a=\beta_2(i)-\lph_2(i) \le 0$. The remaining case is $\lph_1(n)>\beta_1(n)$ which in \eqref{eq:GoverF} gives $x^a \lbd^b$ for some $a,b$ with $a<0$. By assumption, $\lbd^m/x \to \infty$ for every $m$, so taking $m=b/|a|$, $x^a \lbd^b  = (\lbd^{b/|a|}/x)^{|a|}\to \infty$, contradicting \eqref{eq:GoverF}.\\

\nid\textit{Case 2: $(x,\lbd) \in R_{2i-1}$ for some $i$.} We first prove the result assuming $d_\beta>0$ for all $\beta \in B(m_i)$, then show this assumption is implied by $G\ge 0$ and $\piv(B)=\env(B)$. Let $z \in (0,\infty)$ be such that $x\sim \lbd^{m_i}z$. Then
$$G(x,\lbd) \sim \sum_{\beta \in B(m_i)} d_\beta x^{\beta_1}\lbd^{\beta_2} \sim \lbd^{s(B,m_i)}\sum_{\beta \in B(m_i)}d_\beta z^{\beta_1}.$$
Using (iv) of Lemma \ref{lem:dom-terms} and noting $x^{\lph_1}\lbd^{\lph_2} \sim z^{\lph_1}\lbd^{s(A,m_i)}$ for $\lph \in A(m_i)$,
$$F(x,\lbd) = O(\lbd^{s(A,m_i)}).$$
If $d_\beta>0$ for every $\beta \in B(m_i)$ then $G(x,\lbd) \asymp \lbd^{s(B,m_i)}$, so if $|G(x,\lbd)|/|F(x,\lbd)| \to 0$ then $\lbd^{s(B,m_i)} \ll \lbd^{s(A,m_i)}$, implying $s(B,m_i) > s(A,m_i)$. Lastly we show that if $G\ge 0$ and $\piv(B)=\env(B)$ then $d_\beta>0$ for all $\beta \in B(m_i)$ and each $i \in \{1,\dots,n\}$. Recall that $\piv(B) = \{\beta(i)\}_{i=1}^n$. If $\piv(B)=\env(B)$ then since $B(m_i)\subset \env(B)$, $B(m_i) = B(m_i) \cap \piv(B)$ and, since from Lemma \ref{lem:envelope}, $\beta(i-1),\beta(i) \in B(m_i)$, it follows that $B(m_i) = \{\beta(i-1),\beta(i)\}$, so for fixed $z \in (0,\infty)$,
$$G(z\lbd^{m_i},\lbd) \sim \lbd^{s(B,m_i)} (d_{\beta(i-1)} z^{\beta_1(i-1)} + d_{\beta(i)}z^{\beta_1(i)}).$$
By Lemma \ref{lem:envelope}, $i\mapsto \beta_1(i)$ is decreasing, so $\beta_1(i)<\beta_1(i-1)$. By assumption, $d_{\beta(i-1)}\ne 0$ and $d_{\beta(i)}\ne 0$. If $z$ is small then $|d_{\beta(i)} z^{\beta_1(i)}| \le |d_{\beta(i-1)}z^{\beta_1(i-1)}|/2$, so if $G\ge 0$ then since $z>0$, $d_{\beta(i-1)}$ must be positive. Similarly if $z$ is large then
$ |d_{\beta(i-1)}z^{\beta_1(i-1)}|\le |d_{\beta(i)} z^{\beta_1(i)}|/2$, so if $G\ge 0$ then since $z>0$, $d_{\beta(i)}$ must be positive. 
\end{proof}

Next we define the drift-diffusion separatrix curve $\Phi$, which is actually a family of curves in $[0,1]^2$, parametrized by $\ep>0$. We will define it as a set, then show it is a piecewise smooth curve. Define $\Phi(\ep)$ to be the $\ep^2$ level set of $r$, that is,
\begin{align}\label{eq:dd-set}
\Phi(\ep) = \{(x,\lbd) \in [0,1]^2\colon r(x,\lbd) = \ep^2\}.
\end{align}
As is clear from Theorem \ref{thm:prmtrzd-limits}, $\Phi$ delineates the boundary between stochastic and deterministic limits. In the previous section we found that if $F=O(G)$ then diffusion dominates at small scales while drift dominates at larger scales, with the two regimes separated by the dd scale. The following result generalizes that observation. A Jordan arc is the image of an injective, continuous function with domain $[0,1]$.

\begin{theorem}\label{thm:dd-curve}
Suppose $F,G$ have a Taylor expansion around $(0,0)$ with respective powers $A,B$ as in \eqref{eq:loTaylor}, and that $A$ has an element of the form $(\lph_1,0)$. Suppose in addition that $F=O(G)$ on $[0,1]^2$, and let $\Phi$ be as in \eqref{eq:dd-set}. If $\ep>0$ is small, then $\Phi(\ep)$ is a piecewise smooth Jordan arc that connects $(0,1) \times \{0\}$ to $(0,1) \times \{1\}$ through the set $(0,1)^2$, and $\sup\{x \colon (x,\lbd) \in \Phi(\ep)\} \to 0$ as $\ep \to 0$.\\

$\Phi(\ep)$ is the graph of a function $\lbd \mapsto \phi_\ep(\lbd)$ for small $\ep>0$ iff $\lph_1(i)\ge \beta_1(i)$ for every $i \in \{0,\dots,n\}$. When this holds, for given sequences $(a_\ep),(\lbd_\ep)$, $r(1/a_\ep,\lbd_\ep)$ is $\ll \ep^2, \asymp \ep^2$ or $\gg \ep^2$ iff $1/a_\ep$ is $\ll \phi_\ep(\lbd_\ep), \asymp \phi_\ep(\lbd_\ep)$ or $\gg \phi_\ep(\lbd_\ep)$, respectively.
\end{theorem}

Before proving this result, which takes a bit of effort, we record the definition that it justifies. We reserve the word ``upright'' to refer to the case where $\Phi(\ep)$ is the graph of a function $\lbd \mapsto\phi_\ep(\lbd)$.

\begin{definition}[parametrized limit ranges]\label{def:prmtrzd-lim-scales}
In the context of Theorem \ref{thm:prmtrzd-limits}, and assuming $F=O(G)$ around $(0,0)$, define the following \emph{limit ranges} for $(x_\ep)$:
\enumrom
\item $r(1/a_\ep,\lbd_\ep) \ll \ep^2$ is the \emph{pure diffusive range},
\item $r(1/a_\ep,\lbd_\ep) \gg \ep^2$ is the \emph{deterministic range}, and
\item $r(1/a_\ep,\lbd_\ep) \asymp \ep^2$ is the \emph{drift-diffusion (dd) scale}.
\enumend
The set $\Phi(\ep)$ from \eqref{eq:dd-set} is the \emph{dd curve}. Say that the dd curve is \emph{upright} if $\lph_1(i)\ge\beta_1(i)$ for every $i\in \{0,\dots,n\}$.
\end{definition}

\begin{proof}[Proof of Theorem \ref{thm:dd-curve}]
For $i\in \{1,\dots,n\}$ let $L_i = \{(x,\lbd)\in[0,1]^2 \colon x=\lbd^{m_i}\}$, and let $L_0=[0,1]\times\{0\}$ and $L_{n+1} = [0,1]\times\{1\}$. In addition, for $i\in \{0,\dots,n\}$ let $S_i = \{(x,\lbd)\in [0,1]^2 \colon \lbd^{m_{i+1}} \le x \le \lbd^{m_i}\}$. The steps to the proof are as follows: for each $i$ and $\ep \in (0,1)$, we show that
\enumar
\item $\Phi(\ep) \cap L_i$ is a singleton $p_i=(x_i,\lbd_i)$ with $x_i>0$ and $p_i(\ep) \to 0$ as $\ep \to 0$.
\item $\Phi(\ep)\cap S_i$ is a smooth Jordan arc connecting $p_i(\ep)$ to $p_{i+1}(\ep)$.
\item $\Phi(\ep) \cap S_i$ is contained in the rectangle with corners $p_i(\ep)$ and $p_{i+1}(\ep)$.
\item Moving from $p_i(\ep)$ to $p_{i+1}(\ep)$ along $\Phi(\ep) \cap S_i$, $\lbd$ increases iff $\lph_1(i) \ge \beta_1(i)$.
\enumend
Once these are established, the proof is completed as follows. Connecting the arcs $\Phi(\ep) \cap S_i$ from step 2 in increasing order of $i$, we obtain an injective continuous function $\psi_\ep:[0,1] \to [0,1]^2$ with $\psi_\ep(0)=p_0(\ep)\in (0,1)\times \{0\}$ and $\psi_\ep(1)=p_{n+1}(\ep) \in (0,1)\times \{1\}$. Using the second half of step 1 together with step 3, we find that $\sup\{x\colon (x,\lbd) \in \Phi(\ep)\}  = \sup \{x\colon (x,\lbd) \in \bigcup_i p_i(\ep)\}\to 0$ as $\ep \to 0$. Letting $\pi_2(x,\lbd)=\lbd$ denote projection onto the $\lbd$ coordinate, since $\psi$ is continuous and has the given endpoints it follows that $\pi_2(\Phi(\ep)) = \pi_2(\psi_\ep([0,1])) = [0,1]$, so $\Phi(\ep)$ is the graph of a function $\lbd\mapsto \phi_\ep(\lbd)$ if and only if $u\mapsto \pi_2(\psi_\ep(u))$ is increasing, i.e., iff for each $i$, the value of $\lbd$ increases as we move from $p_i(\ep)$ to $p_{i+1}(\ep)$ along $\Phi(\ep)\cap S_i$, which by step 4 holds iff $\lph_1(i)\ge \beta_1(i)$ for every $i$. Lastly, if $\lph_1(i)\ge \beta_1(i)$ for every $i$ then $\dlt_1(i)\ge 0$ for each $i$, and since $r(x,\lbd)=x^{1+\dlt_1(i)}\lbd^{\dlt_2(i)}$ on each $S_i$ with the expressions matching along $L_i$, it follows that for $x,x'$ and fixed $\lbd$, $r(x',\lbd)/r(x,\lbd) \ge x'/x$. Since $r(\phi_\ep(\lbd_\ep),\lbd_\ep) = \ep^2$ by definition, the last statement concerning $\phi_\ep$ follows directly.\\

\nid\textbf{Step 1.} We begin with $L_1,\dots, L_n$. If $(x,\lbd) \in L_i$ and $i\in\{1,\dots,n\}$ then noting that $s(\lph(i),m_i)=s(A,m_i)$ and $s(\beta(i),m_i)=s(B,m_i)$,
\begin{align}\label{eq:rLi}
r(x,\lbd)=\lbd^{m_i(1+\lph_1(i)-\beta_1(i))}\lbd^{\lph_2(i)-\beta_2(i)}=\lbd^{m_i + s(A,m_i) - s(B,m_i)}.
\end{align}
If $F=O(G)$ then by Lemma \ref{lem:FOGenv}, $s(A,m_i)\ge s(B,m_i)$. Therefore $m_i + s(A,m_i)-s(B,m_i)>0$, so $\lbd \mapsto \gma_i(\lbd) := r(\lbd^{m_i},\lbd)$ is strictly increasing. Since $\gma_i(0)=0$ and $\gma_i(1)=1$, if $\ep \in (0,1)$ then $\Phi(\ep) \cap L_i= \{p_i\}$ where $p_i=(x_i,\lbd_i) = (\lbd_i^{m_i},\lbd_i)$ and $\lbd_i$ is the unique solution to $\gma_i(\lbd) = \ep^2$. Denoting it $p_i(\ep)$, it follows that $p_i(\ep) \in (0,1)^2$ for each $\ep\in(0,1)$ and $p_i(\ep)\to 0$ as $\ep \to 0$.\\

We now treat $L_0$ and $L_{n+1}$. If $(x,\lbd) \in L_0$ then $\lbd=0$ so $x \ge \lbd^m$ for any $m>0$, which means that $r(x,0) = x^{1+\lph_1(0)-\beta_1(0)}0^{\dlt_2(0)}$. By assumption, $A$ is an envelope, so is disordered; since $\lph(0)=\arg\max\{\lph_1 \colon \lph \in A\}$, also $\lph(0)=\arg\min\{\lph_2\colon \lph \in A\}$. Since $A$ has an element of the form $(\lph_1,0)$, $\lph_2(0)=0$. Since $F=O(G)$, by Lemma \ref{lem:FOGenv}, $\lph_2(0) \ge \beta_2(0)$ so $\beta_2(0)=\lph_2(0)=0$, which implies $\lph_1(0)\ge \beta_1(0)$ and $x\mapsto r(x,0)$ is strictly increasing, and as above we obtain $p_0(\ep)$ with the stated properties. If $(x,\lbd) \in L_{n+1}$ then $\lbd=1$ so $x\le \lbd^m$ for any $m<\infty$, which means that $r(x,1) = x^{1+\lph_1(n)-\beta_1(n)}1^{\dlt_2(n)}$. Since $F=O(G)$, by Lemma \ref{lem:FOGenv} $\lph_1(n)\ge \beta_1(n)$, so $x\mapsto r(x,1)$ is increasing and as above we obtain $p_{n+1}(\ep)$ with the stated properties.\\

\nid\textbf{Step 2.} Recall that $\dlt(i)=\lph(i)-\beta(i)$. If $1\le i \le n$, then since $s(\lph(i),m_i)=s(A,m_i)\ge s(B,m_i)=s(\beta(i),m_i)$ from Lemma \ref{lem:FOGenv}, if $\dlt_1(i)\le 0$ then $\dlt_2(i)\ge 0$ and if $\dlt_2(i) \le 0$ then $\dlt_1(i)\ge 0$, and the same is true with $<,>$ in place of $\le,\ge$. Since as noted above, $\lph_1(i)\ge \beta_1(i)$ for $i\in\{0,n+1\}$, $\dlt_1(i)\ge 0$ for $i\in \{0,n+1\}$.\\

First note that $p_i(\ep),p_{i+1}(\ep) \in \Phi(\ep) \cap S_i$. Let $r_i(x,\lbd) = x^{1+\dlt_1(i)}\lbd^{\dlt_2(i)}$ so that $r=r_i$ on $S_i$, and let $\Phi_i(\ep) = \{(x,\lbd)\in [0,1]^2\colon r_i(x,\lbd)=\ep^2\}$, so that $\Phi_i(\ep)\cap S_i = \Phi(\ep) \cap S_i$. In particular, $\Phi_i(\ep)\cap L_i = \{p_i(\ep)\}$ and $\Phi_i(\ep) \cap L_{i+1} = \{p_{i+1}(\ep)\}$. We break into cases according to the sign of $1+\dlt_1(i)$.
\enumrom
\item \textit{Case 1: $1+\dlt_1(i)=0$}. Then, $\dlt_1(i)<0$ so $\dlt_2(i)>0$, and the condition $r_i(x,\lbd)=\ep^2$ gives $\lbd = \ep^{2/\dlt_2(i)}$ which is a vertical line. It follows that $\Phi(\ep)\cap S_i$ is a vertical line segment connecting $p_i(\ep)$ to $p_{i+1}(\ep)$.\\

\item \textit{Case 2: $1+\dlt_1(i) \ne 0$.} The condition $r_i(x,\lbd)=\ep^2$ gives
\begin{align}\label{eq:dd-xvalue}
x = \ep^{2/(1+\dlt_1(i))}\lbd^{-\dlt_2(i)/(1+\dlt_1(i))},
\end{align}
so $\Phi_i(\ep)$ is the graph of a function of $\lbd$. Moreover, for fixed $\lbd$, $x\mapsto r_i(x,\lbd)$ is increasing if $1+\dlt_1(i)>0$, and is decreasing if $1+\dlt_1(i)<0$. We consider each subcase separately.
\enumalph
\item $1+\dlt_1(i)>0$. In this case $x\mapsto r_i(x,\lbd)$ is increasing. As shown above, $r$ (and thus $r_i$) increases with $\lbd$ along each $L_i$, $i \in \{1,\dots,n\}$. Denoting $p_i(\ep)$ by $(x_i(\ep),\lbd_i(\ep))$, it follows that $\lbd_{i+1}(\ep) > \lbd_i(\ep)$ for all $i$; for $i=0$ and $i=n$ this follows since $\lbd_0(\ep)=0$ and $\lbd_{n+1}(\ep)=1$, and for $i\in\{1,\dots,n-1\}$ it's because, by moving vertically down from $p_i(\ep)$ onto $L_{i+1}$, $r_i$ decreases below $\ep^2$, so for $r_i$ to reach $\ep^2$ along $L_{i+1}$, $\lbd$ must be increased. Using monotonicity of $r_i$ on both $S_i$ and $L_i$, if $x\le \lbd^{m_i}$  and $\lbd<\lbd_i(\ep)$ then $r(x,\lbd)<\ep^2$, and if $x\ge \lbd^{m_{i+1}}$ and $\lbd>\lbd_{i+1}(\ep)$ then $r(x,\lbd)>\ep^2$. It follows that $\Phi_i(\ep)\cap S_i$ is contained in the strip $\Lbd_i:=\{(x,\lbd)\colon \lbd_i(\ep) \le \lbd \le \lbd_{i+1}(\ep)\}$. Since $\Phi_i(\ep)\cap L_i = \{p_i(\ep)\}$ and $\Phi_i(\ep) \cap L_{i+1} = \{p_{i+1}(\ep)\}$, it follows from the intermediate value theorem that $\Phi_i(\ep) \cap S_i = \Phi_i(\ep) \cap \Lbd_i$. Since $\Phi_i(\ep) \cap \Lbd_i$ is the graph of the restriction of a smooth function to a closed interval, it is a Jordan arc.
\item $1+\dlt_1(i)<0$. Analogous arguments show that $\lbd_{i+1}(\ep)<\lbd_i(\ep)$, that $\Phi_i(\ep) \cap S_i = \Phi_i(\ep) \cap \Lbd_i$ where $\Lbd_i = \{(x,\lbd)\colon \lbd_{i+1}(\ep) \le \lbd \le \lbd_i(\ep)\}$, and that $\Phi_i(\ep) \cap \Lbd_i$ is a Jordan arc.
\enumend
\enumend
\nid\textbf{Step 3.} Referring to the above three cases, when $1+\dlt_1(i)=0$ the result is trivial. When $1+\dlt_1(i) \ne 0$, the result follows from $\Phi_i(\ep)\cap S_i \subset \Lbd_i$ and the fact that $\Phi_i(\ep)$ is the graph of a monotone function of $\lbd$.\\

\nid\textbf{Step 4.} Referring to the above three cases, moving from $p_i(\ep)$ to $p_{i+1}(\ep)$ along $\Phi(\ep)$,
\itemgo
\item if $1+\dlt_1(i)=0$ then $\lbd$ is constant,
\item if $1+\dlt_1(i)>0$ then $\lbd$ increases, and
\item if $1+\dlt_1(i)<0$ then $\lbd$ decreases.
\itemend
In particular, $\lbd$ increases iff $1+\dlt_1(i) = 1+\lph_1(i)-\beta_1(i)>0$, i.e., $\lph_1(i)> \beta_1(i)-1$, which since both are integer-valued is equivalent to $\lph_1(i)\ge \beta_1(i)$.
\end{proof}

We close this section with a simple example showing the possibility of both upright and non-upright dd curves, in particular the transition from upright, to vertical, to folded dd curve.

\begin{example}\label{ex:uprex}
Consider the case with $A = \{(4,0),(1,2)\}$ and $B=\{(k,0),(0,k)\}$ where $k\in \{1,2,3\}$ is fixed, which has $A=\piv(A)$ and $B=\piv(B)$. One easily verifies from Figure \ref{fig:uprex1} that $\piv(A)\subset S(B)$ as required by Lemma \ref{lem:FOGenv}. Then $M(A)= \{2/3\}$ and $M(B)=\{1\}$, so $M:=M(A)\cup M(B)=\{m_1,m_2\}$, with $m_1=2/3$ and $m_2=1$. We have $\lph(0)=(4,0)$ and $\lph(1)=\lph(2)=(1,2)$, and $\beta(0)=\beta(1)=(k,0)$ and $\beta(2)=(0,k)$, so $\dlt(0) = (4-k,0)$, $\dlt(1) = (1-k,2)$ and $\dlt(2) = (1,2-k)$. In particular, the dd curve is upright, i.e., $\dlt_1(i)\ge 0$ for all $i$, iff $k=1$. The dd curve is depicted for each case in Figure \ref{fig:uprex2}.
\end{example}

\begin{figure}
\begin{center}
\includegraphics[width=2in]{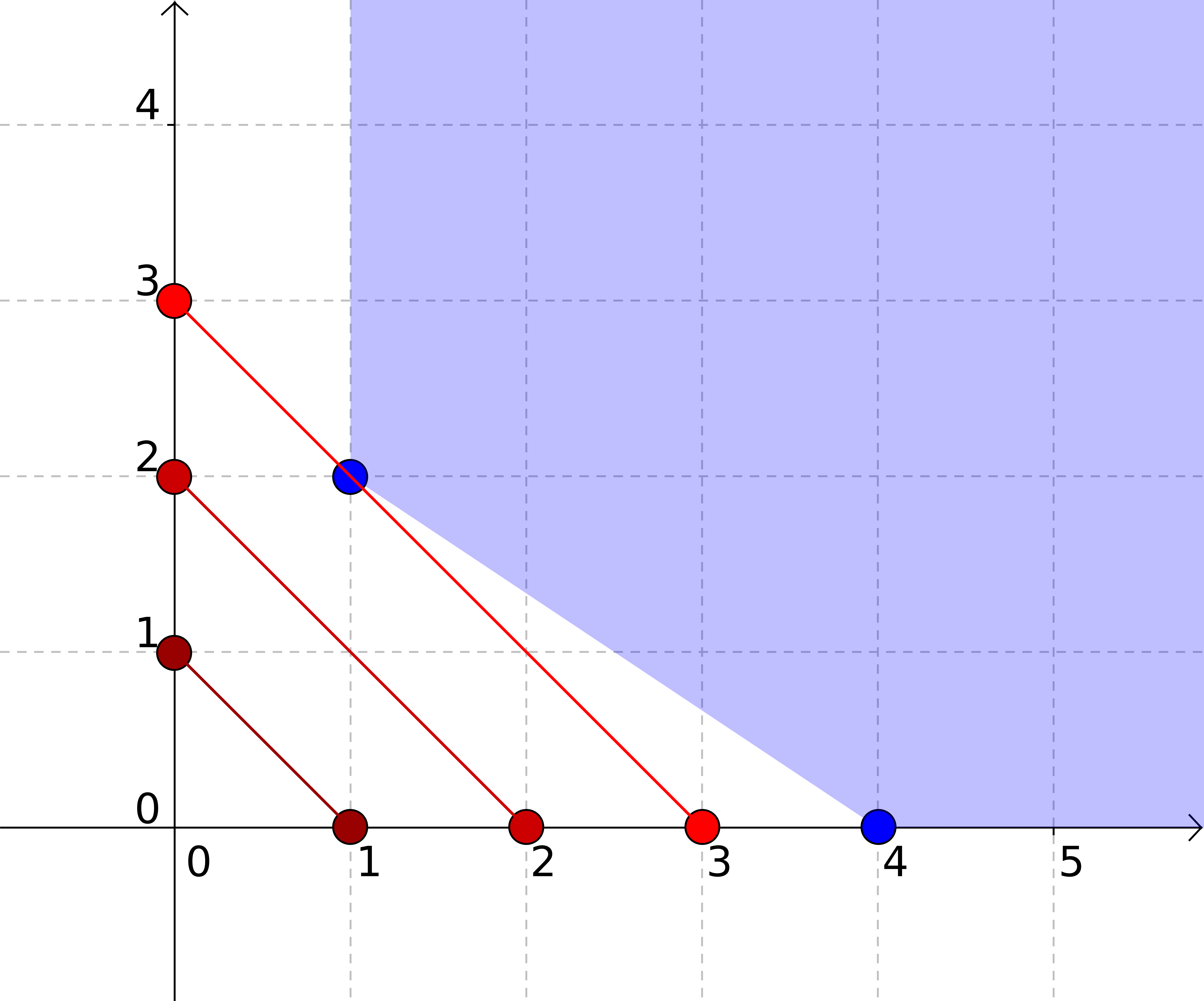}\hspace{.05in}
\end{center}
\caption{a graph depicting $A$ (blue dots), $S(A)$ (blue region) and the three options for $B$ (red dots of darkening shade) in Example \ref{ex:uprex}, with a portion of the contour $L(B)$ (red segments) in each case.}
\label{fig:uprex1}
\end{figure}

\begin{figure}
\begin{center}
\includegraphics[width=1.8in]{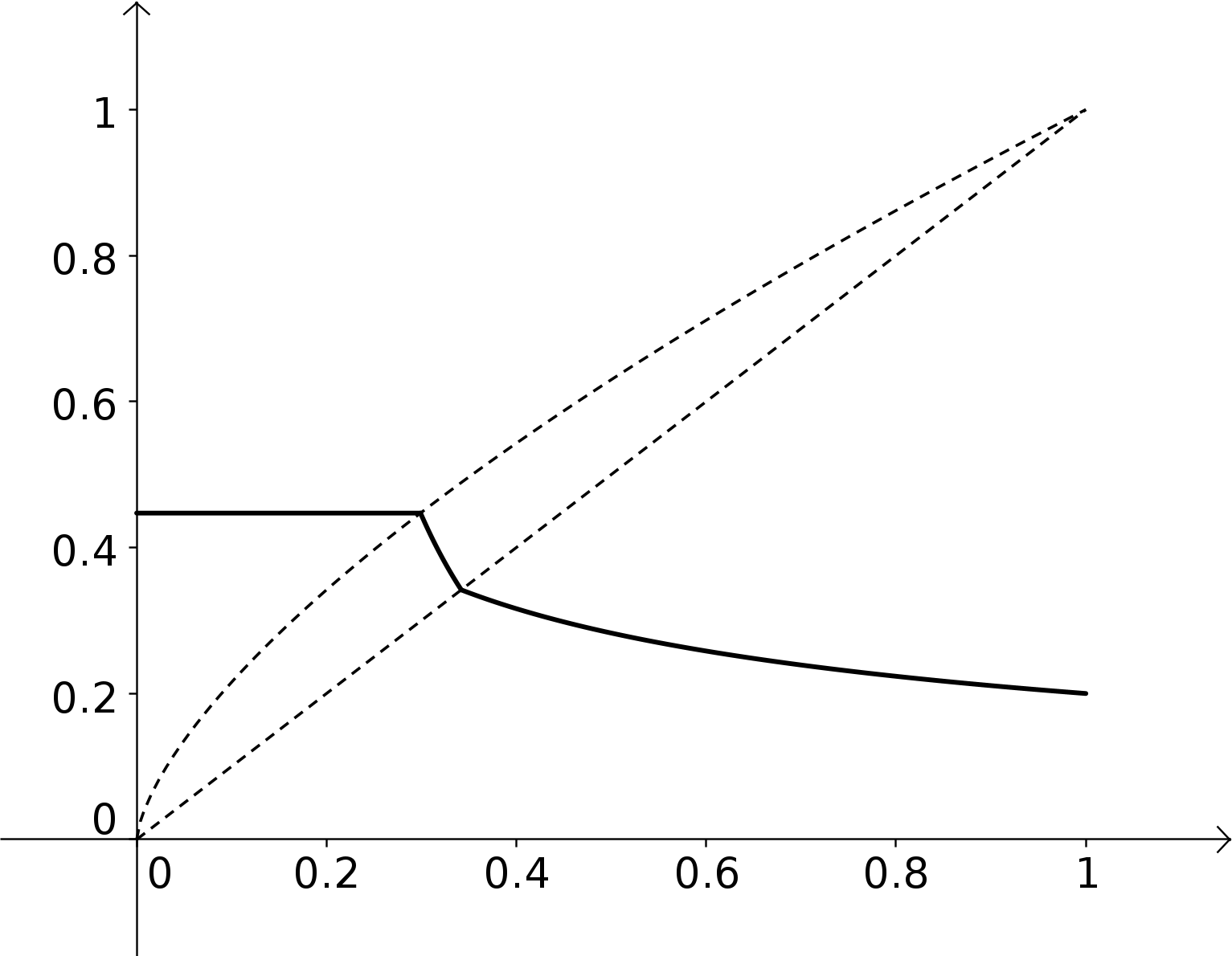}\hspace{.1in}
\includegraphics[width=1.8in]{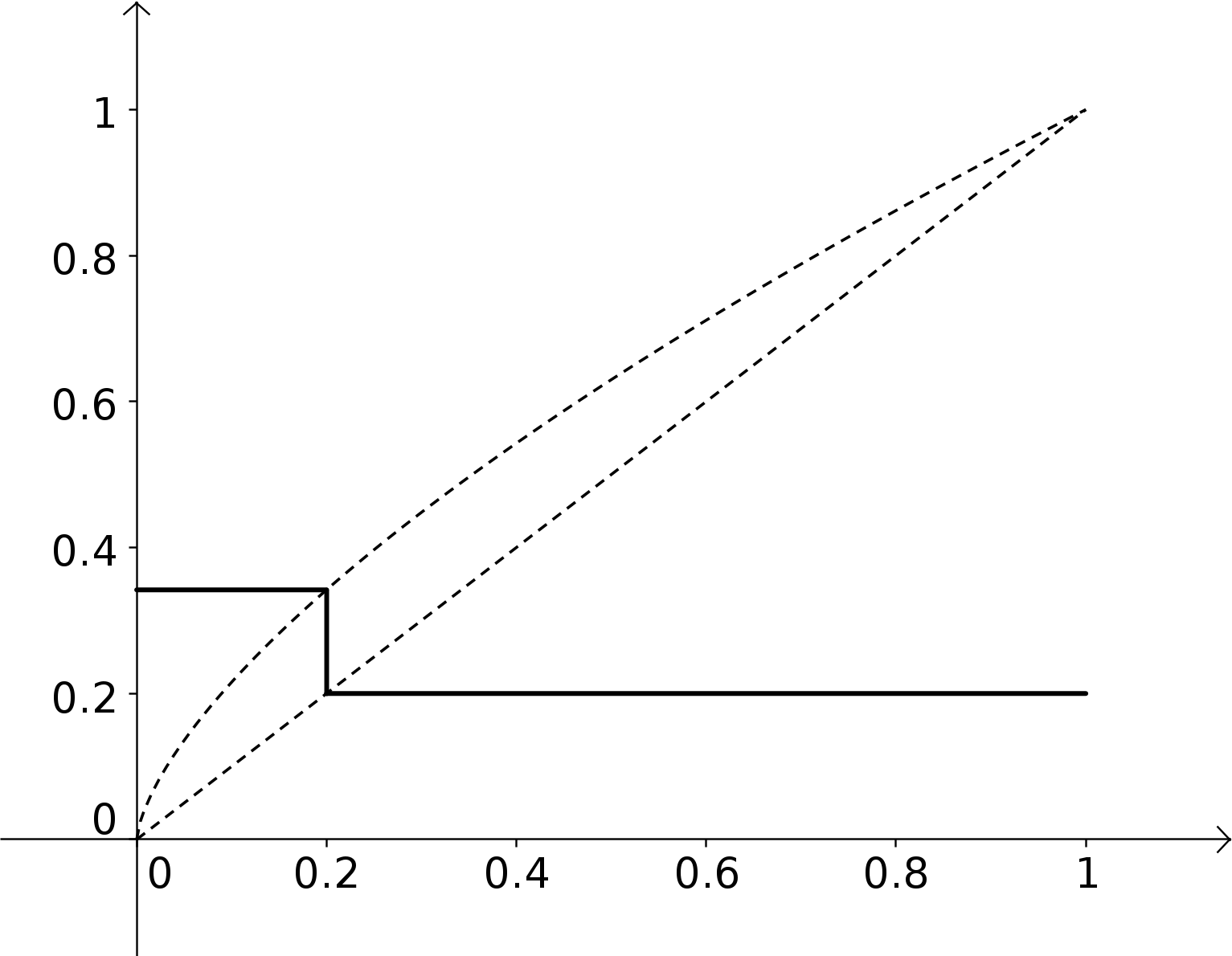}\hspace{.1in}
\includegraphics[width=1.8in]{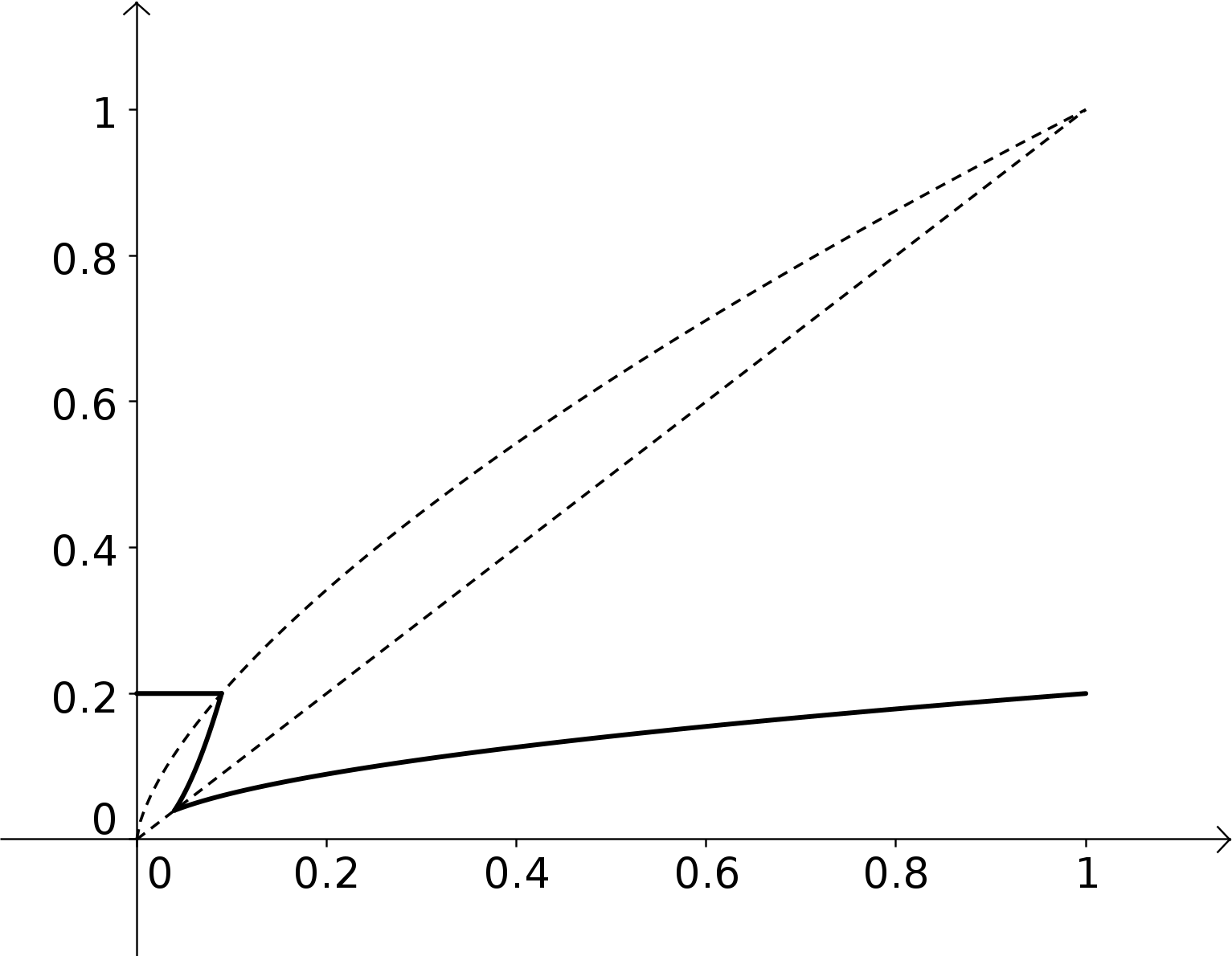}
\end{center}

\caption{The dd curve $\Phi(\ep)$ of \eqref{eq:dd-set} (bold) drawn with $\ep=0.2$ and the lines $L_i=\{(x,\lbd)\colon x=\lbd^{m_i}\}$ (dotted) for Example \ref{ex:uprex} with $k=1,2,3$ from left to right, plotted with vertical $x$ and horizontal $\lambda$.}
\label{fig:uprex2}
\end{figure}

\subsection{Limit scales around a simple equilibrium branch}\label{sec:float-equil}

We now treat case (ii) from the top of Section \ref{sec:prmtzd}: limits around a non-constant branch $x_\star(\lbd)$. Since the goal is to use this for bifurcation theory, we will focus on the case where $x_\star(\lbd)$ is a non-constant simple branch of the zeroset of $F$, i.e., $\lbd \mapsto x_\star(\lbd)$ is non-constant and for each $\lbd$, $x_\star(\lbd)$ is a simple root of $F$. To do so, in the context of Section \ref{sec:const-equil}, suppose there is $\star \in \{1,\dots,n\}$ and $z_\star>0$ such that 
\begin{align}\label{eq:Feqbr}
\sum_{\lph \in A(m_\star)}c_\lph z_\star^{\lph_1}=0 \quad \text{and} \quad \frac{d}{dz}\sum_{\lph \in A(m_\star)}c_\lph z^{\lph_1}\big|_{z=z_\star} \ne 0.
\end{align}
For $\lbd \in (0,1)$ define $ F_\star(z,\lbd) = \lbd^{-s(A,m_\star)}F(z\lbd^{m_\star},\lbd)$, then examining the scale functions for $F$ on $R_{2i-1}$ from Lemma \ref{lem:dom-terms}, we have
$$F_\star(z,\lbd) \sim \sum_{\alpha \in A(m_\star)}c_\lph z^{\lph_1}$$
locally uniformly in $z$, as $\lbd \to 0$. In particular, $F_\star$ can be smoothly extended to $\lbd=0$ by defining $F_\star(z,0) = \sum_{\lph \in A(m_\star)}c_\lph z^{\lph_1}$. By assumption, $F_\star(z_\star,0)=0$ and $\ptl_z  F_\star(z,0) \ne 0$. By the implicit function theorem, there is a function $z_\star(\lbd)$, defined on some interval $[0,\lbd_0)$, such that $F_\star(z_\star(\lbd),\lbd)=0$, and $\ptl_z F_\star(z_\star(\lbd),\lbd) = \ptl_z F_\star(z_\star,0) + o(1) \ne 0$. Letting $x_\star(\lbd) = \lbd^{m_\star}z_\star(\lbd)$, $F(x_\star(\lbd),\lbd)=0$ for $0 \le \lbd <\lbd_0$. \\

We are interested in the scaling of $F,G$ around $x_\star(\lbd)$. Specifically, we seek limits of the form
\begin{align}\label{eq:FG-branch-lim}
\tF(x) &:= \lim_{\ep \to 0}a_\ep b_\ep F(x_\star(\lbd_\ep) + x/a_\ep,\lbd_\ep) \ \text{and} \nonumber \\
\tG(x) &:= \lim_{\ep \to 0}\ep^2 a_\ep^2 b_\ep G(x_\star(\lbd_\ep) + x/a_\ep,\lbd_\ep).
\end{align}
The cases $1/a_\ep \asymp x_\star(\lbd_\ep)$ and $1/a_\ep \gg x_\star(\lbd_\ep)$ are covered by Theorem \ref{thm:prmtrzd-limits}, since if $1/a_\ep \asymp x_\star(\lbd_\ep)$ then \eqref{eq:FG-par-lim} holds iff \eqref{eq:FG-branch-lim} holds, with the argument of $\tF,\tG$ shifted by $\lim_{\ep \to 0}a_\ep x_\star(\lbd_\ep)$, while if $1/a_\ep \gg x_\star(\lbd_\ep)$ then \eqref{eq:FG-par-lim} holds iff \eqref{eq:FG-branch-lim} holds, with the same $\tF,\tG$. Thus, it remains to consider $1/a_\ep \ll x_\star(\lbd_\ep)$. To keep things simple, we shall assume that $F=O(G)$, which constrains the behaviour of $G$ rather nicely. We'll again follow the three steps of Section \ref{sec:iso}: shape, drift to diffusion ratio and time scale.\\

\nid\textit{Shape.} The notion of limit scales is again applicable with the obvious adjustments. Specifically, we seek to decompose the right-hand side of the equations in \eqref{eq:FG-branch-lim} in such a way that $\tF,\tG$ satisfy \eqref{eq:tF-tG-ratio2}, which we recall here for convenience:
$$\tF(x) \sim h_\ep Q(x) \quad \text{and} \quad \tG(x) \sim \ell_\ep V(x)$$
for some $h_\ep,\ell_\ep$ and $Q,V$. By assumption,
$$F_\star(z_\star(\lbd) + z,\lbd) \sim \ptl_zF_\star(z_\star,0)\, z \ \ \text{as} \ \ |z|+|\lbd| \to 0,$$
and by definition,
$$F(x_\star(\lbd) + x,\lbd)=\lbd^{s(A,m_\star)}F_\star(z_\star(\lbd) + x/\lbd^{m_\star},\lbd),$$
so if $x\ll \lbd^{m_\star}$ then
$$F(x_\star(\lbd)+x,\lbd) \sim \lbd^{s(A,m_\star)-m_\star}\ptl_z F_\star(z_\star,0)\, x \ \ \text{as} \ \ \lbd \to 0.$$
Since $x_\star(\lbd) \sim \lbd^{m_\star}z_\star$ as $\lbd \to 0$, if $1/a_\ep \ll x_\star(\lbd_\ep)$ and $x=O(1)$ then $x/a_\ep \ll \lbd^{m_\star}$ and
\begin{align}\label{eq:Fbr-scale}
F(x_\star(\lbd_\ep) + x/a_\ep,\lbd_\ep) \sim (1/a_\ep)\lbd^{s(A,m_\star)-m_\star}\ptl_z F_\star(z_\star,0)\, x.
\end{align}
In particular, $x \mapsto F(x_\star(\lbd_\ep)+x/a_\ep,\lbd_\ep)$ is asymptotically linear as $\ep \to 0$, and in \eqref{eq:tF-tG-ratio2} we can take
$$h_\ep = b_\ep \lbd_\ep^{s(A,m_\star)-m_\star} \ \ \text{and} \ \ Q(x) = \ptl_z F_\star(z_\star,0) \, x.$$
For $G$, define $G_\star(z,\lbd) = \lbd^{-s(B,m_\star)}G(z\lbd^{m_\star},\lbd)$, then similarly to $F_\star$,
$$G_\star(z,\lbd) \sim \sum_{\beta \in B(m_\star)}c_\beta z^{\beta_1},$$
so extend to $\lbd=0$ by letting $G_\star(z,0) = \sum_{\beta \in B(m_\star)}c_\beta z^{\beta_1}$. By assumption, $G$ is non-negative, so $G_\star(z,\lbd) \ge 0$ for all $z,\lbd$, including $\lbd=0$ by continuity. In particular, any zero of the function $z \mapsto G_\star(z,0)$ must have even multiplicity. If $G_\star(z_\star,0)=0$ then since $\ptl_z F_\star(z_\star,0) \ne 0$ and $F=O(G)$ by assumption, $\ptl_z G_\star(z_\star,0) \ne 0$ which implies $G_\star(z,0)<0$ for some $z$ near $z_\star$, a contradiction. It follows that $G_\star(z_\star,0)>0$, and consequently that
$$G_\star(z_\star(\lbd)+z,\lbd) = G_\star(z_\star,0) + o(1) \ \ \text{as} \ \ |z|+|\lbd| \to 0.$$
Thus if $1/a_\ep \ll x_\star(\lbd_\ep)$ and $x=O(1)$ then $x/a_\ep \ll \lbd^{m_\star}$ and
\begin{align}\label{eq:Gbr-scale}
G(x_\star(\lbd_\ep)+x/a_\ep,\lbd_\ep) \sim \lbd_\ep^{s(B,m_\star)} G_\star(z_\star,0),
\end{align}
so in \eqref{eq:tF-tG-ratio2} we can take
$$\ell_\ep = \ep^2 a_\ep^2 b_\ep \lbd_\ep^{s(B,m_\star)} \ \ \text{and} \ \ V(x) = G_\star(z_\star,0).$$

\nid\textit{Drift to diffusion ratio.} Following the recipe of Sections \ref{sec:iso} and \ref{sec:const-equil},
$$\frac{h_\ep}{\ell_\ep} = \frac{\lbd_\ep^{s(A,m_\star)-s(B,m_\star)-m_\star}}{\ep^2 a_\ep^2},$$
so $h_\ep \ll \ell_\ep \Leftrightarrow 1/a_\ep \ll \ep \lbd_\ep^{\gma_\star}$, where $\gma_\star = (m_\star + s(B,m_\star) - s(A,m_\star))/2$, similarly with $\asymp,\gg$ in place of $\ll$.\\

\nid\textit{Time scale.} $(b_\ep)$ is again determined by setting $h_\ep \asymp 1$ or $\ell_\ep \asymp 1$, so we'll just state the result.

\begin{theorem}\label{thm:eq-branch}
Let $x_\star \in \R^d$ and let $\tld U$ be a non-empty open convex set whose closure contains $0$. Let $\lbd$ be a parameter taking values in an interval $I$ containing $0$. Assume that for each domain $D\subset\subset \tld U$, $(x_\ep(t ;\lbd_\ep))_{t\ge 0}$ is a strongly stochastic parametrized QDP on $(D_\ep):=(\{x_\star(\lbd_\ep) + x/a_\ep\colon x \in D\})$ to scale $(a_\ep),(b_\ep)$ with characteristics $F,G$ that have a Taylor expansion around $(0,0)$ as in \eqref{eq:loTaylor}. Suppose in addition that $F$ satisfies \eqref{eq:Feqbr} for some $i$. Let $x_\star(\lbd)$ denote the corresponding branch and let $\gma_\star = (m_\star + s(B,m_\star) - s(A,m_\star))/2$.\\

Let $Y_\ep(t) = a_\ep(x_\ep(b_\ep t;\lbd_\ep)-x_\star(\lbd_\ep))$ and suppose that $a_\epsilon \to \infty$. If $1/a_\ep \ll x_\star(\lbd_\ep)$ and $(a_\ep),(b_\ep),(\lbd_\ep)$ satisfy one of the sets of conditions below, then $(Y_\ep)$ is a QD with characteristics $\tF,\tG$ as described below, where
$$Q(x) = \ptl_z F_\star(z_\star,0)\, x \ \ \text{and} \ \ V(x) = G_\star(z_\star,0)$$
and $F_\star,G_\star$ are as given in the discussion above.
\itemgo
\item If $1/a_\ep \ll \ep \lbd_\ep^{\gma_\star}$ and $b_\ep \asymp \ep^{-2}a_\ep^{-2}\lbd_\ep^{-s(B,m_\star)}$ then $\tld F=0$ and $\tld G \propto V$.
\item If $1/a_\ep \asymp \ep \lbd_\ep^{\gma_\star}$ and $b_\ep \asymp \lbd_\ep^{m_\star-s(A,m_\star)}$ then $\tld F \propto Q$ and $\tld G=0$.
\item If $1/a_\ep \gg \ep \lbd_\ep^{\gma_\star}$ and $b_\ep$ satisfies either condition above then $\tld F \propto Q$ and $\tld G \propto V$.
\itemend
\end{theorem}

\begin{proof}
It is trivial to check that Corollary \ref{cor:QDP-QD} holds with $x_\star(\lbd_\ep)$ in place of $x_\star$. The result then follows from locally uniform convergence in \eqref{eq:FG-branch-lim}, which follows from the discussion.
\end{proof}

As we did before, we define the relevant limit ranges. In this case, it is important that we restrict to $1/a_\ep \ll x_\star(\lbd_\ep)$.

\begin{definition}[limit ranges for a simple equilibrium branch]\label{def:branch-lim-scales}
In the context of Theorem \ref{thm:eq-branch} (in particular, assuming that $1/a_\ep \ll x_\star(\lbd_\ep)$), define the following \emph{limit ranges} for $(x_\ep)$:
\enumrom
\item $1/a_\ep \ll \ep \lbd_\ep^{\gma_\star}$ is the \emph{pure diffusive range},
\item $1/a_\ep \gg \ep \lbd_\ep^{\gma_\star}$ is the \emph{deterministic range}, and
\item $1/a_\ep \asymp \ep \lbd_\ep^{\gma_\star}$ is the \emph{drift-diffusion (dd) scale}.
\enumend
\end{definition}

In cases where both $0$ and $x_\star(\lbd_\ep)$ are equilibria (zeros of $F$), or when there are two branches $\pm x_\star(\lbd_\ep)$ as in the case of a saddle-node, we could say that a bifurcation has occurred, in the sense of the equilibria becoming distinguishable, once the distance between equilibria exceeds the drift-diffusion scale around $x_\star(\lbd)$. Since that distance is of order $x_\star(\lbd_\ep)$, and by assumption, $x_\star(\lbd) \asymp \lbd^{m_\star}$ as $\lbd\to 0$, this occurs when $\ep \lbd^{\gma_\star} \asymp \lbd^{m_\star}$. Setting the two equal to each other,
$$\ep\,\lbd^{(m_\star+s(B,m_\star)-s(A,m_\star))/2} = \lbd^{m_\star},$$
and squaring and re-arranging gives $\ep^2 = \lbd^{m_\star+s(A,m_\star)-s(B,m_\star)}$ or
\begin{align}\label{eq:lbdtranspt}
\lbd = \ep^{\nu_\star} \quad \text{where} \quad \nu_\star = \dfrac{2}{m_\star+s(A,m_\star)-s(B,m_\star)},
\end{align} 
which is also equivalent to $\ep^2 = r(\lbd^{m_\star},\lbd)$, the crossing point of $\Phi(\ep)$ (from Theorem \ref{thm:dd-curve}) with the curve $x=\lbd^{m_\star}$. The condition $1/a_\ep \ll x_\star(\lbd_\ep)$ for validity of the limit scale translates to $\ep\lbd^{\gma_\star} \ll \lbd^{m_\star}$, which corresponds to $\lbd \gg \ep^{\nu_\star}$.

\section{Bifurcations in one dimension}\label{sec:bifurc}

We now apply the theory developed in Section \ref{sec:prmtzd} to study some one-dimensional bifurcations: saddle-node, transcritical and pitchfork. Our goal is to understand how the dd scale changes, in both space and time, as the value of $\lambda$ sweeps across the bifurcation. We'll begin by giving formulae for the objects of study, then conduct the bifurcation analysis.

\subsection{Generalities}\label{sec:gener}

Let's first write down the dd space and time scales, as a function of $\lambda$, around both $x=0$ and $x=x_\star$, assuming in the former case that the dd curve is upright, which will mostly be the case in what follows, and in the latter case that we are in the context of Section \ref{sec:float-equil}.\\

\nid\tbf{Scales around $x=0$. }As in the proof of Theorem \ref{thm:dd-curve}, the dd curve intersects the lines $L_i=\{(x,\lbd)\colon x=\lbd^{m_i}\}$ at the points $p_i(\ep)=(x_i(\ep),\lbd_i(\ep))$ such that $r(p_i(\ep))=\ep^2$, which using \eqref{eq:rLi} , for $i=1,\dots,n$ have
$$\lbd_i(\ep) = \ep^{\nu_i} \ \ \text{where} \ \ \nu_i = \frac{2}{m_i+s(A,m_i)-s(B,m_i)}.$$
If the dd scale is upright then $\lbd_i(\ep)<\lbd_{i+1}(\ep)$ for each $i$. Referring to \eqref{eq:dd-xvalue} and letting $\lbd_0(\ep)=0$ and $\lbd_{n+1}(\ep)=1$, the dd curve around $x=0$ has the form
$$\phi_\ep(\lbd) = \ep^{2/(1+\dlt_1(i))}\lbd^{-\dlt_2(i)/(1+\dlt_1(i))}\quad \text{for} \quad \lbd_i(\ep)\le \lbd \le \lbd_{i+1}(\ep).$$
Since, for a given sequence $(\lbd_\ep)$ and subpartition element $R_i$, the dd scale has $1/a_\ep \asymp \phi_\ep(\lbd_\ep)$, we can sensibly define the dd time scale around $x=0$, modulo $\asymp$, piecewise by the function
$$b_\ep(\lbd) = \phi_\ep(\lbd)^{1-\lph_1(i)}\lbd^{-\lph_2(i)} \quad \text{for} \quad \lbd_i(\ep) \le \lbd \le \lbd_{i+1}(\ep),$$
which is obtained from the form of $b_\ep$ given in Theorem \ref{thm:prmtrzd-limits} by substituting $\phi_\ep$ for $1/a_\ep$.\\

From Theorem \ref{thm:dd-curve} we already know that $\phi_\ep(\lbd)\ll 1$ uniformly over $\lbd$ as $\ep \to 0$, and that $\lbd_i(\ep) \ll 1$ as $\ep \to 0$ for $i\in \{1,\dots,n\}$, since $\phi_\ep(\lbd_i(\ep)) = \lbd_i(\ep)^{m_i}$ and $m_i\in (0,\infty)$. Using the formulae above we can read off a good deal more information. Uprightness is still assumed, so $1+\dlt_1(i) >0$.
\enumrom
\item For each $i$ and $\lbd \in [\lbd_i(\ep),\lbd_{i+1}(\ep)]$,
\enumalph
\item $\phi_\ep$ and $b_\ep$ each have the form $\ep^{q_1}\lbd^{q_2}$ for some $q_1,q_2 \in \Q$, so each is monotone in both $\ep$ and $\lbd$.
\item $\lbd \mapsto \phi_\ep(\lbd)$ is increasing if $\dlt_2(i)<0$, constant if $\dlt_2(i)=0$ and decreasing if $\dlt_2(i)>0$.
\item since $\phi_\ep \ll 1$, if $\lph_1(i)>1$ then $b_\ep(\lbd) \gg 1$ as $\ep \to 0$, uniformly over $\lambda$.
\item if $\lph_1(i)=1$ and $i<n$ then since $\lbd_n(\ep)\ll 1$, $b_\ep(\lbd) \gg 1$ as $\ep \to 0$, uniformly over $\lbd$.
\enumend
\item If $\lph_1(n)=1$ then either 
\enumalph
\item $b_\ep(\lbd)=1$ on $[\lbd_n(\ep),1]$ (if $\lph_2(n)=0$), or
\item $b_\ep(\lbd) \downarrow 1$ as $\lbd \uparrow 1$ (if $\lph_2(n)>0$).
\enumend
\enumend

\nid\tbf{Scales around $x=x_\star$. }Referring to the dd scale in Theorem \ref{thm:eq-branch}, we have $x_\star(\lbd) \asymp \lbd^{m_\star}$, $\gma_\star=m_\star+s(B,m_\star)-s(A,m_\star)$ and $\nu_\star=2/(m_\star+s(A,m_\star)-s(B,m_\star))$, and we can define a dd curve around $x_\star$ with half-width
$$\phi_\ep^\star(\lbd) := \ep \lbd^{\gma_\star} \quad \text{for} \quad \lbd\ge \lbd_\star(\ep):=\ep^{\nu_\star},$$
Similarly as above, referring to Theorem \ref{thm:eq-branch} the dd time scale around $x_\star$ can be defined by
$$b_\ep^\star(\lbd) = \lbd^{m_\star-s(A,m_\star)}\quad \text{for} \quad \lbd\ge \lbd_\star(\ep).$$
From the definition of $\lbd_\star(\ep)$ (see \eqref{eq:lbdtranspt}), it follows that $\phi_\ep^\star(\lbd_\star(\ep)) \asymp \phi_\ep(\lbd_\star(\ep))$ as $\ep \to 0$. Similarly it can be verified from the formulae that $b_\ep^\star(\lbd_\star(\ep)) \asymp b_\ep(\lbd_\star(\ep))$ as $\ep \to 0$. As before we will make some general observations, this time for $\lbd \in [\lbd_\star(\ep),1]$.
\enumalph
\item $\phi_\ep^\star$ and $b_\ep^\star$ each have the form $\ep^{q_1}\lbd^{q_2}$ for some $q_1,q_2 \in \Q$, so each is monotone in $\ep$ and $\lbd$.
\item $\phi_\ep^\star(\lbd) \ll 1$ as $\ep \to 0$, uniformly in $\lbd$. This follows from monotonicity of $\lbd\mapsto \phi_\ep^\star(\lbd)$ together with $\phi_\ep^\star(\lbd_\star(\ep)) \asymp \phi_\ep(\lbd_\star(\ep))\ll 1$ and $\phi_\ep^\star(1) = \ep \ll 1$.
\item $\phi_\ep^\star(\lbd) \to \ep$ and $b_\ep^\star(\lbd) \to 1$ as $\lbd \to 1$, for each $\ep>0$.
\item Since $b_\ep^\star(1)=1$, if $b_\ep(\lbd_\star(\ep))\gg 1$ as $\ep \to 0$ then for small $\ep>0$, $b_\ep^\star(\lbd)\downarrow 1$ as $\lbd \uparrow 1$.
\enumend

\subsection{Bifurcations}

We now study bifurcations. Since, in principle, we can use the above formulae to describe exactly the dd space and time scales, we shall be more concerned with the following qualitative properties:
\enumar
\item The general shape of $\phi_\ep$ and $\phi_\ep^\star$, i.e., the intervals $[\lbd_i(\ep),\lbd_{i+1}(\ep)]$ on which \\each one increases, decreases or remains constant, and
\item The values of $\lbd$ for which $b_\ep(\lbd)$ and $b_\ep^\star(\lbd)$ are $\gg 1$, $\asymp 1$ or $\ll 1$ as $\ep\to 0$, respectively, regions where the diffusion limit is slow, fast, or irrelevant (i.e., too short to observe on the original time scale).
\enumend
The context is a strongly stochastic ($F=O(G)$) parametrized QDP, as in Definitions \ref{def:ssQDP} and \ref{def:prmtrzdQDP}, whose characteristics $F,G$ have a Taylor expansion to leading order around $(0,0)$ as in \eqref{eq:loTaylor}, and for which $F(0,0)=\ptl_x F(0,0)=0$. Letting $A,B$ denote the powers of $F,G$ from \ref{eq:loTaylor}, by Lemma \ref{lem:envelope-sum}, we can, and will, assume in this section that $A,B$ are envelopes.\\

\nid\tbf{Bifurcation types. }The set $A$ determines the bifurcation type, and from the assumption $F=O(G)$ and Lemma \ref{lem:FOGenv}, $B$ is constrained by the condition $\piv(A)\subset S(B)$. The bifurcations that we'll consider correspond to the following choices for $A$ and equilibria:
\enumar
\item Saddle-node: $A=\{(2,0),(0,1)\}$ with equilibria at $\pm \, x_\star(\lbd)$ where $x_\star(\lbd) \asymp \sqrt{\lbd}$.
\item Transcritical: $A=\{(2,0),(1,1)\}$ with equilibria at $0$ and $x_\star(\lbd) \asymp \lbd$.
\item Pitchfork: $A=\{(3,0),(1,1)\}$ with equilibria at $0$ and $\pm \, x_\star(\lbd)$ where $x_\star(\lbd) \asymp \sqrt{\lbd}$.
\enumend

In all three cases, $A$ has the form $\{(j_1,0),(j_2,1)\}$ for some $j_2\le 1<j_1$, which has $\piv(A)=A$ and $M(A)=\{m_A\}$ with $m_A:=1/(j_1-j_2)\}$, and $F$ has a non-constant equilibrium branch $x_\star(\lbd) \asymp \lbd^{m_\star}$ with $m_\star=m_A$. For simplicity we'll assume $B$ has no elements $(\beta_1,\beta_2)$ with $\beta_2 \ge 2$, i.e., $G$ has no relevant $\lbd$ dependence above $\lbd^1$, which is reasonable for most applications. Using this and the constraint $A \subset S(B)$, in each case either
\enumrom
\item $B=\{(k,0)\}$ for some $k\le j_2$ which has $M(B)=\emptyset$, or
\item $B=\{(k_1,0),(k_2,1)\}$ for some $k_2<k_1\le j_1$ with $k_2 \le j_2$ which has $M(B)=\{1/(k_1-k_2)\}$. 
\enumend

\nid\tbf{Uprightness. }For the above bifurcations and any compatible choice of $B$, the dd curve around $x=0$ is upright, with the exception of $A=\{(3,0),(0,1)\}$ and $B=\{(2,0),(0,1)\}$ where it is vertical for $\lbd \le x \le \sqrt{\lbd}$. This is clear if $B=\{(k,0)\}$ since $\lph_1(i)\in \{j_1,j_2\}$ for each $i$ while $\beta_1(i)=k \le \min(j_1,j_2)$. Otherwise, since $k_1\le j_1$ and $k_2\le j_2$, $\alpha_1(i) < \beta_1(i)$ only occurs if for some $i$, $\lph(i) = (j_2,1)$ and $\beta(i)=(k_1,0)$ with $k_1>j_2$. Since $A(m)=\{(j_2,1)\}$ iff $m>m_A$ and $B(m)=\{(k_1,0)\}$ iff $m<m_B:=1/(k_1-k_2)$, this is possible only if $m_A<m_B$, which as confirmed by a quick sketch is only possible in the pitchfork case, and only when $B=\{(2,0),(0,1)\}$.\\

\nid\tbf{Regions. }As usual, we restrict to $(x,\lbd)\in [0,1]^2$. The behaviour in the other three quadrants can be inferred by symmetry: limit scales around $x=0$ are unchanged modulo reflection, and the same goes for scales around $x=x_\star$, when there is an equilibrium branch in the given quadrant, such as quadrant IV for saddle-node and pitchfork, and quadrant III for transcritical. Note that limit scales do not depend on whether the equilibrium is stable or unstable.\\

Using what we learned in Section \ref{sec:gener}, we give a qualitative picture of the dd space and time scales around $x=0$ and $x=x_\star$. We have $M(A)=\{m_A\}$ with $m_A=1/(j_1-j_2)$ and if $M(B) \ne \emptyset$ then $M(B)=\{m_B\}$ with $m_B=1/(k_1-k_2)$. In addition, $m_\star=m_A$. There are four cases to consider:
\enumrom
\item if $B=\{(k,0)\}$ then $M(B)=\emptyset$,
\item if $j_1-j_2=k_1-k_2$ then $m_B=m_A$,
\item if $j_1-j_2>k_1-k_2$ then $m_B<m_A$ and
\item if $j_1-j_2<k_1-k_2$ then $m_B>m_A$.
\enumend
This gives
$$M:=M(A)\cup M(B) = \begin{cases} \{m_A\} & \text{in cases (i)-(ii)}, \\
\{m_A,m_B\} & \text{in cases (iii)-(iv)}.\end{cases}$$
When $|M|=1$, with $m_1=m_A$ the intervals to consider are $[0,\lbd_\star(\ep)]$ and $[\lbd_\star(\ep),1]$. When $|M|=2$, with $m_1=\min(m_A,m_B)$ and $m_2=\max(m_A,m_B)$, the intervals are $[0,\lbd_\1(\ep)]$, $[\lbd_\1(\ep),\lbd_2(\ep)]$ and $[\lbd_2(\ep),1]$. A table of $\lph,\beta$ and $\delta$ values is given in Table \ref{tab:bifurcases}.\\


\begin{table}\label{tab:bifurcases}
\begin{tabular}{ c | c | c | c }
Case & $(\lph(i))_{i=0}^n$ & $(\beta(i))_{i=0}^n$ & $(\dlt(i))_{i=0}^n$ \\ \hline
$M(B)=\emptyset$ & $((j_1,0),(j_2,1))$ & $((k,0),(k,0))$ & $((j_1-k,0),(j_2-k,1))$ \\
$m_B=m_A$ & $((j_1,0),(j_2,1))$ & $((k_1,0),(k_2,1))$ & $((j_1-k_1,0),(j_2-k_2,0))$ \\
$m_B<m_A$ & $((j_1,0),(j_1,0),(j_2,1))$ & $((k_1,0),(k_2,1),(k_2,1))$ & $((j_1-k_1,0),(j_1-k_2,-1),(j_2-k_2,0)$ \\
$m_B>m_A$ & $((j_1,0),(j_2,1),(j_2,1))$ & $((k_1,0),(k_1,0),(k_2,1))$ & $((j_1-k_1,0),(j_2-k_1,1),(j_2-k_2,0))$
\end{tabular}
\end{table}

\nid\tbf{Shape of $\phi_\ep$ and $\phi_\ep^\star$. }As noted in Section \ref{sec:gener}, $\phi_\ep$ increases ($\nearrow$), decreases ($\searrow$) or is constant ($\rightarrow$) if $\dlt_2(i)$ is negative, positive or zero respectively. Reading $\dlt_2(i)$ from Table \ref{tab:bifurcases}, from left to right over the two or three regions,
\enumrom
\item if $M(B)=\emptyset$ then $\phi_\ep$ is $\rightarrow$, then $\searrow$,
\item if $m_B=m_A$ then $\phi_\ep$ is constant on $[0,1]$,
\item if $m_B<m_A$ then $\phi_\ep$ is $\rightarrow$, then $\nearrow$, then $\rightarrow$ again, and
\item if $m_B>m_A$ then $\phi_\ep$ is $\rightarrow$, then $\searrow$ , then $\rightarrow$ again.
\enumend
In the exceptional case (see uprightness above) which belongs to case (iv), $\phi_\ep$ is vertical instead of decreasing in the second region. In each case, $\phi_\ep(\lbd) \to \ep^{2/(1+\dlt_1(n))} = \ep^{2/(1+j_2-k_2)}$ ($k$ in place of $k_2$ in case (i)) as $\lbd \to 1$. Since $k_2 \le j_2\le 1$, $j_2-k_2 \in \{0,1\}$. For the saddle-node, $j_2=0$ requiring $k_2=0$ and giving limit $\ep^2$, while for the other two bifurcations it depends on the value of $k_2$.\\

Similarly, $\phi_\ep^\star$ is $\nearrow$, $\searrow$ or $\rightarrow$ on $[\lbd_\star(\ep),1]$ if $\gma_\star$ is $>0$, $<0$ or $=0$. We leave it to the interested reader to determine its shape in each case. As noted in Section \ref{sec:gener}, $\phi_\ep^\star(\lbd_\star(\ep)) \asymp \phi_\ep(\lbd_\star(\ep))$ and $\phi_\ep^\star(\lbd) \to \ep$ as $\lbd \to 1$, so one way to infer its shape is to compute $\phi_\ep(\lbd_\star(\ep))$ and compare it to $\ep$.\\

\nid\tbf{Scale of $b_\ep$ and $b_\ep^\star$. }Here is a brief summary: for $\lbd$ near $0$, time scales are $\gg 1$, while for $\lbd$ near $1$, the time scale around equilibrium points $\to 1$ following the curve $1/\lbd$, and around non-equilibrium points is eventually $\ll 1$ as $\lbd \to 1$. In other words, diffusion is slow near the bifurcation point, and moving away from the bifurcation point, becomes fast around equilibria and irrelevant around non-equilibrium points. We now give and demonstrate the precise statements.\\

In all cases, $b_\ep(\lbd)\gg 1$ uniformly over $\lbd \in [0,\lbd_n(\ep)]$ as $\ep \to 0$, i.e., the dd time scale around $0$ is slow in the leftmost $n$ (of $n+1$) regions. Since $\lbd_\star(\ep)\le \lbd_n(\ep)$, in particular the time scale is $\gg 1$ at least up to the point where equilibria become distinguishable. Using observations (i)(c)-(d) of the first half of Section \ref{sec:gener}, the above is true provided $\lph_1(i) \ge 1$ for all $i<n$. This is obvious for transcritical and pitchfork as $\lph_1\ge 1$ for all $\lph\in A$, while for saddle-node, $n=1$ and $\lph_1(0)=2$. If, instead, $\lbd \in [\lbd_n(\ep),1]$ then from the formula for $b_\ep$ and the fact that in all cases, $\lph_1(n)=j_2$ and $\lph_2(n)=1$,
$$b_\ep(\lbd) = \begin{cases} 1/\lbd & \text{if} \ j_2=1, \\
\phi_\ep(\lbd)/\lbd & \text{if} \ j_2=0.\end{cases}$$
In particular, as $\lbd \to 1$, since $\phi_\ep(\lbd) \to \ep^{2/(1+k_2-j_2)}$,
$$b_\ep \to \begin{cases} 1 & \text{if} \ \ j_2=1,\\
\ep^{2/(1+k_2-j_2)} & \text{if} \ \ j_2=0\end{cases}$$
(with $k$ in place of $k_2$ if $B=\{(k,0)\}$). If $j_2=0$ then $b_\ep(\lbd)\le 1$ iff $\lbd \ge \ep^{2/(1+k_2-j_2)}$, so if $\lbd \gg \ep^{2/(1+k_2-j_2)}$ then since $b_\ep(\lbd)\ll 1$ it is sensible to declare the dd scale irrelevant at that point, as drift has completely washed out diffusion. For the time scale around $x_\star$, since we've shown that $b_\ep(\lbd_\star(\ep))\gg 1$ as $\ep \to 0$, using observation (d) from the second half of Section 5.1 it follows that $b_\ep^\star(\lbd) \downarrow 1$ as $\lbd \uparrow 1$, for each $\ep>0$. In particular, if $\lbd_\star(\ep) \le \lbd \ll 1$ then $b_\ep^\star(\lbd) \gg 1$.

\subsection{Canonical form of $G$}

With the multiplicity of functions $G$ satisfying $F=O(G)$ for a given $F$, the question arises whether there is a canonical or generic choice of $G$, one to be expected most often in applications. I wish to argue that for the above examples, where $A=\{(j_1,0),(j_2,1)\}$, an obvious choice is to take
\begin{align}\label{eq:canonB}
B=\{(k,0)\} \quad \text{where} \quad k=\min(j_1,j_2)=j_2.
\end{align}
Equivalently, $B=\{\beta\}$ where $\beta = (j_1,0)\wedge (j_2,1)$. To understand this claim, note that a bifurcation can arise from two competing mechanisms, when their contributions to the drift are equal and opposite: for example, varying the reproduction rate in a population growth model, a transcritical bifurcation occurs when birth and death rates are equal. When each mechanism occurs randomly at specified rates, the diffusion contributed from each one is additive, even if the drift is cancellative. This explains the assumption $F=O(G)$. The particular form \eqref{eq:canonB} is obtained if, apart from the terms that cause the bifurcation, no other cancellation occurs.\\

Happily, \eqref{eq:canonB} leads to the simplest diagrams, since $M(B)=\emptyset$. There are only two regions: $[0,\lbd_\star(\ep)]$ and $[\lbd_\star(\ep),1]$. The first region can be viewed as the ``critical window'', where the dd scale is $\ge$ the separation distance between equilibria, and (borrowing terminology from \cite{luczak}) the second region is the ``barely subcritical/supercritical'' region, where the equilibria have separated but $\lbd$ is still generally $\ll 1$ as $\ep\to 0$.\\

Let's begin by computing $\lbd_\star(\ep)$. Since $m_\star=m_A=m_1=1/(j_1-j_2)$ and $\lph(0)=(j_1,0),\beta(0)=(k,0)=(j_2,0)$, $s(A,m_\star)-s(B,m_\star)=m_\star (j_1-j_2)=1$ and so
$$\nu_\star = 2/(m_\star+s(A,m_\star)-s(B,m_\star)) = 2(j_1-j_2)/(1 + j_1-j_2)$$
and, noting $j_1>1\ge j_2$,
$$\lbd_\star(\ep)=\ep^{2(j_1-j_2)/(1+j_1-j_2)}.$$

\nid\tbf{Critical window. }For $\lbd \in [0,\lbd_\star(\ep)]$, both $\phi_\ep$ and $b_\ep$ are constant, and $b_\ep \gg 1$ as $\ep \to 0$. Most of this has already been noted, except that $b_\ep$ is constant, which using the formula for $b_\ep$ follows from $\phi_\ep$ being constant, and the fact that $\lph_2(0)=0$. The precise formulae are as follows: since $k=j_2$, referring to Table \ref{tab:bifurcases}, since $k=j_2$ we have $\lph(0) = (j_1,0)$ and $\dlt(0)=(j_1-j_2,0)$ and, noting $j_1>1\ge j_2$,
$$\phi_\ep = \ep^{2/(1+j_1-j_2)} \quad\text{and} \quad b_\ep = \ep^{-2(j_1-1)/(1+j_1-j_2)}.$$

\nid\tbf{Barely non-critical region. }For $\lbd \in [\lbd_\star(\ep),1]$ we first treat scales around $x=0$, then around $x=x_\star$. We have $\lph(1)=(j_2,1)$ and since $k=j_2$, $\dlt(1)=(0,1)$ and so
$$\phi_\ep(\lbd) = \ep^2/\lbd, \quad b_\ep(\lbd) = \begin{cases} 1/\lbd & \text{if} \ j_2=1, \\  \ep^2/\lbd^2 & \text{if} \ j_2=0.\end{cases}$$
In particular, $\phi_\ep(\lbd)\to \ep^2$ as $\lbd \to 1$, and for the saddle-node, diffusion is irrelevant once $\lbd \gg \ep$, while in the other cases, diffusion around $x=0$ persists as $\lbd \to 1$, approaching fast diffusion. For scales around $x=x_\star$, recall $m_\star=1/(j_1-j_2)$ and $s(A,m_\star)-s(B,m_\star)=1$, so $\gma_\star = m_\star+s(B,m_\star)-s(A,m_\star) = 1/(j_1-j_2)-1 \le 0$, and since $j_1>1\ge j_2$,
$$\phi_\ep^\star(\lbd) = \begin{cases}
\ep & \text{if} \ j_1=2, \ j_2=1,\\
\ep\,\lbd^{1/(j_1-j_2)-1} & \text{otherwise},\end{cases}$$
with $\phi_\ep^\star(\lbd) \downarrow \ep$ as $\lbd \uparrow 1$ in the second case. Then, since $m_\star-s(A,m_\star)=m_\star(1-j_1)-j_2 = -((j_1-1)/(j_1-j_2)+j_2)$,
$$b_\ep^\star(\lbd) = \lbd^{-((j_1-1)/(j_1-j_2)+j_2)}.$$
Since $(j_1-1)/(j_1-j_2)+j_2>0$, this supports what we already showed: that $b_\ep^\star(\lbd) \gg 1$ for $\lbd_\star(\ep) \le \lbd \ll 1$ and $b_\ep^\star(\lbd)\downarrow 1$ as $\lbd \uparrow 1$.\\

\section*{Acknowledgements}

The author is grateful for support from an NSERC Discovery Grant.

\section*{Appendix}

\subsection*{Proof of Lemma \ref{lem:SDExist}}

\begin{proof}
The existence and uniqueness of $(P_x)_{x \in \hat U}$, and the Feller and strong Markov properties, are given in Theorem 13.1 in Chapter 1 of \cite{pinsky}. It remains to show that if $\tau(U)<\infty$ and $X(0)\in U$ then as $t\to \tau(U)^-$ either $|X(t)|\to \infty$ or $d(X_t,z)\to 0$ for some $z \in \partial U$, where $d$ is Euclidean distance. By definition of $ \smash{ \hat \Omega_U}$, $X(t)=\c$ iff $t\ge \tau(U)$ and $X(t) \to_{\rho_U} \c$ as $t \to \tau(U)^-$. So, by definition of $\rho(U)$, one of $|X(t)| \to \infty$ or $d(X(t),\partial U) \to 0$ holds as $t\to \tau(U)^-$, where $d$ is Euclidean distance. If $\liminf_{t \to \tau(U)^-}|X(t)|<\infty$ then $A\cap \partial U \ne \emptyset$, where $A$ is the limit set of $\{X(t)\colon t<\tau(U)\}$ with respect to $d$, so it is enough to show that if $\tau(U)<\infty$ and $\liminf_{t\to\tau(U)^-}|X(t)|<\infty$ then $X(t)$ converges as $t\to \tau(U)^-$.\\

For $x,z \in U$ define $f_z(x)=|x-z|^2 = \sum_i (x_i-z_i)^2$ and with $L$ as in \eqref{eq:mg-op}, for $t<\tau(U)$ define
$$M(t;z)=f_z(X(t)) - \int_0^t (Lf_z)(X(s))ds,$$
so that $t\mapsto M(t\wedge \tau(D_n);z)$ is a martingale, by \eqref{eq:mp-mg}. We compute
$$(Lf_z)(x) = \sum_i (G_{ii}(x) + 2F_i(x)(x_i-z_i)).$$
For $r>0$ let $B_r=\{x\in\R^d \colon |x|\le r\}$. Since, by assumption, $F,G$ are bounded on bounded subsets of $U$,
\begin{align}\label{eq:Lbnd}
A:=\sup\{|(Lf_z)(x)| \colon x,z \in B_r \cap U\}<\infty.
\end{align}
If $\tau$ is a stopping time with $\tau\le \tau(D_n)$, then since stopping preserves the martingale property, $t\mapsto M(t\wedge \tau;z)$ is a martingale, and if moreover $\sup_{t <\tau} |X(t)| \le r$, then using \eqref{eq:Lbnd}, for all $t\ge 0$
$$M(t\wedge \tau;z) \ge f_z(X(t\wedge \tau)) -  (t \wedge \tau) A.$$
Using the strong Markov property and the fact that $M(0;X(0))=0$, it follows that if $\tau,\tau'$ are stopping times with $\tau \le \tau' \le \tau(D_n)$ and $M(t)$ is defined by
\begin{align}\label{eq:stop-mg}
M(t) = \1(\tau<t)\left(f_{X(\tau)}(X(t \wedge \tau')) - \int_{\tau}^{t\wedge \tau'} (Lf_{X(\tau)})(X(s))ds \right),
\end{align}
then $M$ is a martingale, and if moreover $\sup_{t \in [\tau,\tau')}|X(t)| \le r$, then
\enumrom
\item $M(t) \ge -(t\wedge \tau'- t\wedge \tau)A$ for all $t \ge 0$, and
\item $M(\tau') \ge |X(\tau')-X(\tau)|^2 - (\tau'-\tau )A$ for all $t\ge \tau'$.
\enumend
Fix $\epsilon \in (0,1)$ and $T>0$ and define the times $\tau_0=0$ and
\begin{align}\label{eq:ret-time}
&\tau_{2i+1} = \tau(D_n) \wedge T\wedge \inf\{t\ge \tau_{2i}\colon |X(t)| \le r-1\},\nonumber \\
&\tau_{2i+2} = \tau(D_n) \wedge T\wedge \inf\{t \ge \tau_{2i+1}\colon |X(t)-X(\tau_{2i+1})|^2 > \epsilon\}
\end{align}
Then for each $i$, the stopping times $\tau_{2i+1},\tau_{2i+2}$ fulfill the conditions decribed for $\tau,\tau'$ above. Let $M_i(t)$ denote the process from \eqref{eq:stop-mg} with $\tau_{2i+1},\tau_{2i+2}$ in place of $\tau,\tau'$, and let $S_i(t) = \sum_{j=1}^i M_i(t)$. Summing over $j\le i$ in (i) above, $S_i(t) \ge - (t\wedge \tau_{2i+2})A \ge - TA$ for $t\ge 0$. Using (ii) above, on the event $\{\tau_{2i+2}<\tau(D_n)\wedge T\}$, $M_i(t) \ge \epsilon - (\tau_{2i+2}-\tau_{2i+1})A$ for all $t \ge \tau_{2i+2}$, so on $\{\tau_{2i+2}<\tau(D_n)\wedge T\}$, $S_i(t) \ge i\,\epsilon -(\tau_{2i+2})A \ge i\,\epsilon - TA$ for all $t\ge \tau_{2i+2}$. We record these two facts:
\enumrom
\item $S_i(t) \ge -  T A$ for all $t\ge 0$, and
\item $S_i(t) \ge i\,\epsilon -  T A$ for all $t \ge \tau_{2i+2}$, on $\{\tau_{2i+2}<\tau(D_n)\wedge T\}$.
\enumend
Since each $M_i$ is a martingale, each $S_i$ is a martingale. In addition, $S_i(0)=0$, so together with (i), $S_i + TA$ is a non-negative martingale, with $S_i(0)+TA = TA$. Using Doob's inequality, for $C>TA$,
$$P(\sup_{t\ge 0}S_i(t)\ge C-TA) \le TA/C,$$
which also holds for $C\le TA$ since then $TA/C\ge 1$. Write the times $\tau_i$ defined by \eqref{eq:ret-time} as $\tau_i^n$ to emphasize the dependence on $n$. Combining with (ii),
\begin{align}\label{eq:tau-bound}
P(\tau_{2i+2}^n < \tau(D_n)\wedge T) \le TA/(i\,\epsilon).
\end{align}
Defining $(\tau_i)_{i \ge 1}$ as in \eqref{eq:ret-time} but with $\tau(U)$ in place of $\tau(D_n)$, clearly $\tau_i^n = \tau_i \wedge \tau(D_n)$ for all $i,n$. If $\tau_{2i+2}^n = \tau(D_n)\wedge T$ for all $n$, then since $\tau(D_n) \to \tau(U)$, $\tau_{2i+2}=\tau(U)\wedge T$, so if $\tau_{2i+2}<\tau(U)\wedge T$ then $\tau_{2i+2}^n < \tau(D_n)\wedge T$ for some $n$. Using \eqref{eq:tau-bound} and continuity of probability, it follows that
\begin{align}\label{eq:tau-bound2}
P(\tau_{2i+2} < \tau(U) \wedge T) \le TA/(i\,\epsilon).
\end{align}\\
Define the oscillation of $X$ as $t\to\tau(U)^-$ by
$$\osc = \lim_{t\to\tau(U)^-} \sup_{u,v \in [t,\tau(U))}|X(u)-X(v)|.$$
Then by completeness of $\R^d$, $X(t)$ converges as $t\to\tau(U)^-$ iff $\osc=0$. Suppose $\liminf_{t\to \tau(U)^-}|X(t)| \le r-1$, $\osc > \epsilon$ and $\tau(U)\le T$. Then, $\tau_{2i+2}<\tau(U) \wedge T$ for every $i\ge 1$. Since $\{\tau_{2i+2}<\tau(U)\wedge T\}$ is a decreasing sequence in $i$, by \eqref{eq:tau-bound2} and continuity of probability, $P(\tau_{2i+2} < \tau(U)\wedge T \ \text{for all} \ i\ge 1) =0$. Thus,
$$P(\liminf_{t\to \tau(U)^-}|X(t)| \le r-1, \ \osc > \ep \ \text{and} \ \tau(U)\le T)=0.$$
Since $\ep,T,r>0$ are arbitrary, taking a sequence $\ep_m,r_m,T_m$ with $\ep_m \to 0$ and $r_m,T_m\to\infty$ as $m\to\infty$,
$$P(\liminf_{t\to\tau(U)^-}|X(t)| <\infty , \ \osc > 0 \ \text{and} \ \tau(U) <\infty)=0.$$
In other words, if $\liminf_{t\to\tau(U)^-}|X(t)| <\infty$ and $\tau(U)<\infty$, then $\osc=0$ which implies $X(t)$ converges as $t\to\tau(U)^-$, which is what we needed to show.
\end{proof}

\subsection*{Proof of Lemma \ref{lem:diff-limit}}

As in the statement of Lemma \ref{lem:diff-limit}, given $U,F,G,x$ let $X$ denote the solution supplied by Lemma \ref{lem:SDExist} with $X(0)=x$. The goal is to show that if $(X_\ep)$ is a QD with characteristics $F,G$ defined on $U$ and $X_\ep(0)\to x$ as $\ep \to 0$ then $X_\ep \smash \lcd X$ on $U$ as $\ep\to 0$.

\begin{proof}
We will need Theorem 4.1 from Chapter 7 of \cite{ethktz}. First, we define the stopped martingale problem and give an existence and uniqueness result.\\

\nid\textit{Definition.} In the context of Definition \ref{def:mp}, say that $(P_x)$ solves the stopped martingale problem for $F,G$ on $D$ if $P_x(X(t)=X(t\wedge \tau(D))=1$ for all $x$ and if \eqref{eq:mp-mg} is a martingale for $f \in C^2(U)$, with $\tau(D)$ in place of $\tau(D_n)$.\\

\nid\textit{Existence and uniqueness.} Say that a problem is well-posed if it has a unique solution. Suppose $F,G$ satisfy the conditions of Lemma \ref{lem:SDExist}. Then, as in Theorem 13.1 of Chapter 1 in \cite{pinsky}, for domains $D\subset\subset U$ we can define $F_D,G_D$ on $\R^d$ that coincide with $F,G$ on $D$ and are such that (i) the martingale problem for $F_D,G_D$ on $\R^d$ is well-posed, and (ii) the solution for $F_D,G_D$ coincides with the solution for $F,G$ up to time $\tau(D)$, i.e., the distributions of the stopped processes coincide. It follows from (i) and Theorem 6.1 in Chapter 4 of \cite{ethktz} that the stopped martingale problem for $F,G$ on $D$ is well-posed, and from (ii) that its distribution is given by the solution for $F,G$, stopped at time $\tau(D)$.\\

\nid\textit{Convergence.} We adapt Theorem 4.1 of Chapter 7 in \cite{ethktz} to the present context. Suppose the generalized martingale problem for $F,G$ is well-posed, and fixing $x \in U$ let $X$ denote the process with distribution $P_x$. As explained above, for $D\subset\subset U$ the stopped martingale problem for $F,G$ on $D$ is well-posed, and its unique solution with initial value $x$ is given by $X(\cdot \wedge \tau(D))$. Suppose we have \cdlg~$\R^d$-valued processes $X_n$ and $B_n$ and an $M_d(\R)$-valued process $A_n$ such that $A_n(t)-A_n(s)$ is positive semidefinite for $t>s\ge 0$. Let $\F_t^n = \sigma(X_n(s),B_n(s),A_n(s) \colon s\le t)$. Let $\tau(F,n) = \inf\{t \colon X_n(t^-) \notin D \ \text{or} \ X_n(t) \notin D\}$. Suppose that $M_n := X_n-B_n$ and $M_nM_n^{\top} - A_n$ are $\F^n$-local martingales, and that for $D\subset\subset U$ and $T>0$, 
\enumrom
\item $\lim_{n\to\infty} E\left(\sup_{t \le \tau(D,n) \wedge T} |\Delta X_n(t)|^2\right)=0$,
\item $\lim_{n\to\infty} E\left(\sup_{t \le \tau(D,n) \wedge T} |\Delta B_n(t)|^2\right)=0$,
\item $\lim_{n\to\infty} E\left(\sup_{t \le \tau(D,n) \wedge T} |\Delta A_n(t)|\right)=0$,
\item $\sup_{t \le \tau(D,n) \wedge T}|B_n(t) - \int_0^t F(X_n(s))ds| \cp 0$ and
\item $\sup_{t \le \tau(D,n) \wedge T}|A_n(t) - \int_0^t (GG^{\top})(X_n(s))ds| \cp 0$.
\enumend
If these conditions are satisfied and $X_n(0) \to x$ as $n\to\infty$, then defining the increasing family $$D_r = \{x \in U\colon |x| < r \ \text{and} \ d(x,U^c) > 1/r \}$$
where $|\cdot|$ is any norm and for a point $y$ and set $A$, $d(y,A) = \inf\{d(y,z)\colon z\in A\}$ where $d$ is Euclidean distance, using the same approach as the proof of Theorem 4.1 of Chapter 7 in \cite{ethktz} it follows that for all but countably many $r$, $X_n(\cdot \wedge \tau(D_r,n))\smash \cd X(\cdot \wedge \tau(D_r))$. So, taking a strictly increasing sequence $r_i\to 0$  such that the above converges with $r=r_i$ for each $i$ and letting $D_i = D_{r_i}$, it follows that $X_n\smash \lcd X$ as $n\to\infty$.\\

\nid\textit{Application to QD.} We now apply this result to a QD $(X_\ep)$. We'll refer to (i)-(v) above as conditions and (i)-(iii) from Definition \ref{def:qd} as assumptions. Since $F,G$ are locally uniformly continuous, assumption (i) implies that $X_\ep$ can be replaced with $J_c X_\ep$ in assumptions (ii)-(iii). By a standard argument, there exist $a_\ep \to 0$ as $\ep \to 0$ such that (i)-(iii) are true with $a_\ep$ in place of $a$. Let $Z_\ep = J_{a_\ep}X_\ep$. Noting that $\tau(D,\ep) < \zeta(X_\ep)=\zeta(Z_\ep)$, from assumptions (i)-(iii) we find that
\begin{enumerate}[noitemsep,label={(\roman*)}]
\item $\sup_{t \le \tau(D,\ep)\wedge T}|X_\ep(t)-Z_\ep(t)| \cp 0$ as $\ep \to 0$,
\item $\sup_{t \le \tau(D,\ep)\wedge T} \left| \, (Z_\ep)^p(t)  - \int_0^t F(Z_\ep(s))ds \, \right| \cp 0 \ \ \text{as} \ \ \ep \to 0$, and
\item $\sup_{t \le \tau(D,\ep)\wedge T} \left| \lng (Z_\ep)^m \rng(t) - \int_0^t G(Z_\ep(s))ds \, \right| \cp 0 \ \ \text{as} \ \ \ep \to 0$,
\end{enumerate}

From now on, property (i)-(iii) refers to the above points. Let $X$ denote the solution to the generalized martingale problem for $F,G$. If we can show that $Z_\ep \smash \lcd X$ as $\ep \to 0$ then property (i) implies $X_\ep \smash \lcd X$. Let $(\ep_n)$ be a sequence with $\ep_n \to 0$ as $n\to\infty$ and let $X_n = Z_{\ep_n}$, $B_n = Z_{\ep_n}^p$ and $A_n = \lng Z_{\ep_n}^m \rng$, noting that $M_n:= X_n-B_n= Z_{\ep_n}^m$ and $M_nM_n^{\top}-A_n$ are local martingales, as required. Properties (ii) and (iii) are equivalent to conditions (iv) and (v) so we need only establish conditions (i)-(iii). By definition, $|\Delta Z_\ep| \le a_\ep$ a.s., so since $a_{\ep_n} \to 0$ as $n\to\infty$, condition (i) holds. From [I.4.24], $|Z_\ep| \le a_\ep$ implies $|\Delta Z_\ep^p| \le a_\ep$ and $|\Delta Z_\ep^m| \le 2a_\ep$. $|\Delta Z_\ep^p| \le a_\ep$ and $a_{\ep_n}\to 0$ implies condition (ii). $|\Delta Z_\ep^m| \le 2a_\ep $ is equivalent to $|\Delta M_n| \le 2 a_{\epsilon_n}$, and condition (iii) is obtained from it, as follows.\\

For a jointly measurable (with respect to the underlying $\sigma$-algebra and time) process $X$, let $^p X$ denote the predictable projection of $X$, which as defined and proved in [I.2.28], is the unique predictable process that satisfies $^pX(\tau) = E(X(\tau) \mid \F(\tau^-))$ on $\{\tau<\infty\}$, for all predictable $\tau$. Clearly, $^p(\,\cdot\,)$ is linear. If $X$ is ~\cdlg~and $\tau$ is predictable then
\begin{align}\label{eq:pp-jump}
|(\,^p(\Delta X))(\tau)| \le E(\, | \, (\Delta X)(\tau)| \  \mid \F(\tau^-)) \le E(\,|(\Delta X)(\tau)|\,).
\end{align}
If $X$ is a (\cdlg) local martingale, then by [I.2.31], $^p(\Delta X)=0$. If $X$ is \cdlg~and predictable, it follows easily from [I.2.4] and [I.2.24] that $^p(\Delta X)=\Delta X$. Since $M_nM_n^{\top}-A_n$ is a local martingale and $A_n$ is \cdlg~and predictable, $^p(\Delta (M_n M_n^{\top} - A_n))=0$ and $\Delta A_n = \,^p(\Delta A_n)$, so $\Delta A_n = \,^p(\Delta (M_nM_n^{\top}))$. Since $M_nM_n^{\top} - [M]$ is a local martingale,$^p(\Delta (M_n M_n^{\top})) = ^p(\Delta [M_n])$, so $\Delta A_n = \,^p (\Delta [M_n])$. From [I.4.47(c)], $\Delta [M_n] = \Delta M_n \Delta (M_n^{\top})$ and using the earlier estimate $|\Delta M_n| \le 2a_{\epsilon_n}$, if $\tau$ is predictable then using \eqref{eq:pp-jump},
$$|\,^p(\Delta[M_n]))(\tau)| \le E(|\Delta [M_n]|) \le 4a_{\epsilon_n}^2.$$
Since $A_n$ is predictable, by [I.2.24] there is a sequence $(T_m)$ of predictable times such that $\{t\colon \Delta A_n(t)\ne 0\} \subset \bigcup_m T_m$. Using the above display with $\tau=T_m$, it follows that a.s.~$|\Delta A_n| \le 4a_{\epsilon_n}^2$. Since $a_{\epsilon_n} \to 0$ as $n\to\infty$, condition (iii) holds.
\end{proof}

\end{document}